\documentclass[11pt]{amsart}

\usepackage[top=3cm,bottom=2.75cm,left=3cm,right=3cm]{geometry}

\RequirePackage[utf8]{inputenc} 

\usepackage{hyperref}
\usepackage[active]{srcltx}
\usepackage{mathtools}
\usepackage{psfrag}
\usepackage{tikz} 
\usetikzlibrary{decorations.pathmorphing, patterns,shapes}
\usepackage{float}
\usepackage{amsmath}
\usepackage{amsthm}
\usepackage{acronym}
\usepackage{color}
\usepackage{mathrsfs} 
\usepackage{enumitem}
\usepackage{bbm}
\usepackage{upref} 

\usepackage{graphicx}

\usepackage{float}
\usepackage{epstopdf}
\usepackage{subfigure}
\usepackage{longtable}
\usepackage{multicol}
\usepackage[normalem]{ulem}

\usepackage[square, numbers, comma, sort&compress]{natbib}

\usepackage{amsmath,amsfonts,amssymb,amscd,amsthm,xspace}
\theoremstyle{plain}

\newtheorem{theorem}{Theorem}[section]
\newtheorem{corollary}[theorem]{Corollary}
\newtheorem{lemma}[theorem]{Lemma}
\newtheorem{Prop}[theorem]{Proposition}
\newtheorem{proposition}[theorem]{Proposition}

\theoremstyle{definition}
\newtheorem{definition}[theorem]{Definition}

\theoremstyle{remark}
\newtheorem{remark}[theorem]{Remark}

\numberwithin{equation}{section}

\newcommand\restrict[1]{\raisebox{-.2ex}{$|$}_{#1}}

\usepackage{mathrsfs}

\DeclareMathOperator{\dist}{dist}
\DeclareMathOperator{\spt}{supp}

\newcommand{\norm}[1]{\left\lVert#1\right\rVert}
\newcommand{\ff}[1]{f\left( #1  \right) }
\newcommand{\hh}[1]{h\left( #1  \right) }

\newcommand*\diff{\mathop{}\!\mathrm{d}}
\newcommand{\betrag}[1]{\left\vert#1\right\vert}
\newcommand{\normb}[1]{|#1|}

\newcommand{\normf}[1]{\left\lVert#1\right\rVert_{\text{flat}}}
\newcommand{\R}{{\mathbb{R}}}

\newcommand{\Z}{{\mathbb{Z}}}
\newcommand{\N}{{\mathbb{N}}}

\newcommand{\ue}{u_{\varepsilon}}

\newcommand{\aee}{a_{\varepsilon}}
\newcommand{\ve}{v_{\varepsilon}}

\newcommand{\tgae}{\tilde{\gamma}_{\varepsilon}}
\newcommand{\tggae}{\tilde{g}_{\varepsilon}}

\newcommand{\scdot}{\, \cdot \,}

\newcommand{\go}{\gamma^{\textup{opt}}}

\newcommand{\gae}{\gamma_{\varepsilon}}

\newcommand{\ggae}{g_{\varepsilon}}

\newcommand{\Qd}{Q^{d-1}}

\newcommand{\eps}{\varepsilon}

\newcommand{\Wme}{{W^{1,1}}(\Omega)}

\newcommand{\Io}{\int_{\Omega}}

\newcommand{\Ij}{\int_{J_u}}
\newcommand{\HA}{\mathcal{H}^{d-1}}

\newcommand{\Du}{Du}

\newcommand{\Wd}{\mathcal{W}(\Omega, \R^d)}

\newcommand{\Wde}{\mathcal{W}(\Omega)}

\newcommand{\Rd}{\R^{d}}

\newcommand{\Mam}{\mathcal{M}(\Omega, \R^{d})}
\newcommand{\Mamone}{\mathcal{M}(\Omega)}

\newcommand{\Lm}{{L^{1}}(\Omega, \R^d)}
\newcommand{\Lme}{{L^{1}}(\Omega)}

\newcommand{\F}{\mathcal{F}}
\newcommand{\G}{\mathcal{G}}

\newcommand{\Sm}{\mathbb{S}^{d-1}}
\newcommand{\Am}{\mathcal{A}(\Omega)}
\newcommand{\Bm}{\mathcal{B}(\Omega)}
\newcommand{\dd}{ \, \textup{d}}

\newcommand{\cff}{\psi_{\varepsilon}}

\newcommand{\Lmd}{\mathcal{L}^d}

\newcommand{\Mned}{\mathcal{M}^{a,b}}
\newcommand{\Mneds}{\mathcal{M}^{a,0}}

\newcommand{\mrs}{\mathbin{\vrule height 1.6ex depth 0pt width
0.13ex\vrule height 0.13ex depth 0pt width 1.3ex}}

\newcommand\restr[2]{{
  \left.\kern-\nulldelimiterspace 
  #1 
  \vphantom{\big|} 
  \right|_{#2} 
  }}

\def\Xint#1{\mathchoice
{\XXint\displaystyle\textstyle{#1}}%
{\XXint\textstyle\scriptstyle{#1}}%
{\XXint\scriptstyle\scriptscriptstyle{#1}}%
{\XXint\scriptscriptstyle\scriptscriptstyle{#1}}%
\!\int}
\def\XXint#1#2#3{{\setbox0=\hbox{$#1{#2#3}{\int}$ }
\vcenter{\hbox{$#2#3$ }}\kern-.6\wd0}}

\def\dashint{\Xint-}

\newcommand{\Ibb}{\dashint_{B_{\varepsilon}(x)}}

\newcommand{\weakly}{\rightharpoonup}
\newcommand{\weakstar}{\stackrel{\ast}{\rightharpoonup}}

\newcommand{\overbar}[1]{\mkern 1.5mu\overline{\mkern-2.5mu#1\mkern-1.5mu}\mkern 1mu}
\newcommand{\oOmega}{\overbar{\Omega}}

\allowdisplaybreaks[3]

\begin{document}

\title[Eigendamage: an Eigendeformation model 
for the variational approximation of 
cohesive fracture in antiplane shear deformations ]{Eigendamage: an Eigendeformation model 
for the \\ variational approximation of 
cohesive fracture \\ in antiplane shear deformations }

\author[V. Auer-Volkmann]{Veronika Auer-Volkmann}
\address{V. Auer-Volkmann: Institut f\"{u}r Mathematik, Universit\"{a}t Augsburg, Universit\"{a}tsstr. 14, 86159 Augsburg, Germany} 
\email{veronika.auervolkmann@gmail.com}

\author[L. Beck]{Lisa Beck}
\address{L. Beck: Institut f\"{u}r Mathematik, Universit\"{a}t Augsburg, Universit\"{a}tsstr. 12a, 86159 Augsburg, Germany} 
\email{lisa.beck@math.uni-augsburg.de}

\author[B. Schmidt]{Bernd Schmidt}
\address{B. Schmidt: Institut f\"{u}r Mathematik, Universit\"{a}t Augsburg, Universit\"{a}tsstr. 14, 86159 Augsburg, Germany} 
\email{bernd.schmidt@math.uni-augsburg.de}

\date{\today}



\begin{abstract}
We set up and investigate an approximation scheme for a variational theory of cohesive fracture in antiplane shear deformations. We consider a family of functionals acting on scalar valued functions in general dimensions depending on a small parameter $0 < \varepsilon \ll 1$ and on two fields: the elastic part of the displacement field and an eigendeformation field that describes the inelastic response of the material beyond the elastic regime. We measure the inelastic contributions of the latter in terms of a non-local energy functional. Our main result shows that, as $\varepsilon \to 0$, the approximate functionals $\Gamma$-converge to a cohesive zone model. 
\end{abstract}

\maketitle

\section{Introduction} 

The description of materials undergoing different deformation mechanisms under the action of external loads---ranging from purely elastic and plastic responses to the formation of fractures---is a central topic in mathematical materials science. In recent decades, variational approaches have proved particularly successful in this regard. Since the seminal work of Francfort and Marigo \cite{FrancfortMarigo:98}, in which---based on Griffith’s classical work \cite{Griffith:20} on fracture propagation---variational models were formulated in terms of free discontinuity problems, a broad range of generalized Griffith-type theories has been developed and applied to fracture-mechanics problems; see, for instance, the treatises \cite{Braides:98,afp,BourdinFrancfortMarigo:08,DelPiero:13} and the references therein. 

In ductile materials, deformations beyond the elastic regime may lead to the nucleation of micro-voids and the formation of damage zones in which the strain exhibits diffuse-singular contributions and the specimen undergoes plastic deformation; see, e.g., \cite{Dugdale:60,Barenblatt:62,FremondNedjar:96,FrancfortGarroni:06,GarroniLarsen:09,DelPieroTruskinovsky:09,CagnettiToader:11,PhamMarigo:13,AlessiMarigoVidoli:14}. This naturally leads to variational models whose bulk energy density has linear growth at infinity and whose surface energy scales linearly with small crack openings. In contrast to the brittle setting, bulk and surface contributions may thus interact in a non-trivial manner, cf.\ \cite{BouchitteBraidesButtazzo:95}, making the resulting models challenging both from the analytical and numerical perspectives.

Motivated by the need for efficient computational methods, a variety of numerical approximation schemes for generalized Griffith models have been proposed. In particular, the Ambrosio–Tortorelli scheme \cite{AT92,AT90} laid the foundation for a rapid and widespread development of so-called \emph{phase field models}. In addition to the deformation field, these models introduce an auxiliary variable---the phase field---which can be interpreted as a degradation parameter governing the loss of stiffness of the material; see, for example, \cite{F01,Iurlano:13,ContiFocardiIurlano:16,DalMasoOrlandoToader:16, CrismaleLazzaroniOrlando:18,ChambolleCrismale:19,ContiFocardiIurlano:24,MaggiorelliNegriVicentiniDeLorenzis:26}. Alternative approaches approximate the strain field itself through nonlocal convolution operators; see  \cite{BrGa98,Braides:98,CortesaniToader:99b,LuVi97,LVDD07,Neg06,Lussardi:08,LM13,ScillaSolombrino:21,MarzianiSolombrino:24}. 

Specifically for brittle materials, the \emph{eigenfracture model} was developed in \cite{sfo}. As in conventional phase-field approaches, a second field variable is employed. In contrast to phase-field models, however, this variable physically describes the inelastic response of the material; cf.\ \cite{M13} for the notion of eigendeformations. The corresponding inelastic contribution is penalized by means of nonlocal convolution-type functionals. We refer to \cite{PandolfiOrtiz:12,PandolfiLiOrtiz:14,WangSun:17,StochinoQinamiKaliske:17,QinamiPandolfiKaliske:20,ZhangShenZhou:20,PWO:21,ChihadehStormKaliske:23,DuongFriedrich:26,BozinisLammenKurzejaRoerentropSchmidtMosler:26} for further developments and applications of the eigenfracture framework. 

The main objective of the present work is to develop an eigendeformation based model that can be applied to ductile materials in the scalar antiplane shear setting. To this end, we introduce and analyze an `eigendamage model' capable of efficiently capturing the elastic, plastic, and fracture phenomena within such a framework.

As in \cite{sfo}, we consider a two-field model in which, for a small parameter $\varepsilon > 0$, an eigendeformation field $\gamma_\varepsilon$ is considered alongside the deformation variable $u_\varepsilon$ to describe the inelastic part of the deformation. 
We now, however, consider a model for which the limiting functional, as $\varepsilon \to 0$, exhibits a non-trivial coupling between plastic deformation and crack opening. Similar to the eigenfracture model and in contrast to common damage models, the constitutive laws do not need to be explicitly made dependent on a damage variable. Instead, an increase of damage is associated with a transition from the elastic to the inelastic eigendeformation field. From a mathematical perspective, this is reflected by a stored energy density that grows only linearly with respect to large bulk terms $\nabla u_\varepsilon - \gamma_\varepsilon$ while it scales linearly with respect to small inelastic terms, which are given by a nonlocal approximation of~$\gamma_\varepsilon$.    

The present work continues the investigation initiated in \cite{AuerVolkmannBeckSchmidt:22}, where the special case of a one-dimensional beam was studied. As already observed there, the multidimensional setting is considerably more challenging from a technical standpoint. 

Several aspects of our analysis build upon the contributions \cite{LuVi97,LVDD07} of Lussardi and Vitali, from which a number of auxiliary results can be adapted. Nevertheless, the coupling between $u_\varepsilon$ and $\gamma_\varepsilon$ in the present model poses substantial additional difficulties. Indeed, these fields are linked only through the regularity constraint $\nabla u_\varepsilon - \gamma_\varepsilon \in L^2$, while their convergence is available only in a very weak sense (with respect to suitable negative Sobolev norms). As a consequence, standard slicing techniques in the bulk are no longer applicable, which necessitates the development of new analytical arguments.

\subsection*{Outline}

We introduce our eigendamage model and state the main $\Gamma$-convergence theorem in Section~\ref{sec:settingBV}. Section~\ref{sec_preliminaries} collects the general notation, basics on functions of bounded variation and convergence in negative Sobolev spaces and in the flat topology. Our theorem on compactness of bounded energy sequences is formulated and proved in Section~\ref{sec_Compactness}. In Section~\ref{sec:relBV} we prove a novel relaxation result for free discontinuity functionals with linear growth in the space $BV$, which will be needed in Section~\ref{sec:upper-lim} and, we believe, might be of some independent interest. Section~\ref{sec: estimate from below} is devoted to the proof of the $\Gamma$-$\liminf$ inequality of our main result by the localization method. Correspondingly, Section~\ref{sec:upper-lim} contains the  proof of the $\Gamma$-$\limsup$ inequality. Finally, in Section~\ref{sec_optimzed-gamma} we prove, as a corollary to our main theorem, a $\Gamma$-convergence result for energies with optimized second variable $\gamma$.

\section{Setting of the problem and main results}\label{sec:settingBV}
Let $\Omega \subset \R^d$ be a bounded open set with Lipschitz boundary $\partial \Omega
$. We consider pure antiplane shear deformations of an elastic body, which are described by a deformation field $u \colon \Omega \to \R$ of bounded variation and an $\R^d$-vector valued measure $\gamma$ on~$\Omega$ as the eigendeformation. We model the energy of such a deformation by approximation of the deformation and eigendeformation. To this end, we consider functions $u \in \Wme$ and measures $\gamma = g \mathcal{L}^d$ with $g \in L^1(\Omega, \Rd)$. For the associated energy we take the function $f \colon [0,\infty) \to [0, \infty)$ defined via
\begin{align*}
f(t) = \left\{
\begin{array}{l l}
c_0\, t & \quad   \text{if } t < 1, \\
c_0 & \quad  \text{if } t \geq 1,
\end{array}
\right.
\end{align*} 
where $c_0$ is a fixed positive constant (which constitutes the simplest continuous function with $f(0)=0$, $\lim_{t \to 0^+} \frac{f(t)}{t}=c_0$ and $\lim_{t \to \infty} f(t)=c_0$). 
We further fix a constant $K>0$ and we then consider the functional $E_{\varepsilon} \colon L^1(\Omega) \times \Mam \to [0,\infty]$ given by
\begin{align*}  
E_{\varepsilon}(u,\gamma) \coloneqq  
\begin{cases}
\int_\Omega { |\nabla u- g |^2 \dd x }& \quad \text{if } u \in W^{1,1}(\Omega), \,   \norm{u}_{L^{\infty}(\Omega)} \leq K, \\ + \frac{1}{\varepsilon} \int_\Omega {f \big( \varepsilon \dashint_{B_{\varepsilon}(x) \cap \Omega}{ |g | } \dd y} \big) \dd x  & \quad \phantom{if }  \gamma= g \mathcal{L}^d, \, g \in L^1(\Omega,\R^d), \\
 & \quad \phantom{if } \text{and } \nabla u- g \in L^2(\Omega,\R^d), \\[0.2cm]
\infty & \quad \text{otherwise},
\end{cases}
\end{align*}
for $\varepsilon >0$. We observe that~$E_{\varepsilon}$ can only be finite if~$\gamma$ is absolutely continuous with respect to the $d$-dimensional Lebesgue measure~$\mathcal{L}^d$. In this case, $\nabla u-g$ represents the elastic strain, while~$\gamma$ describes the deformation beyond the elastic regime (which is permanent if $\gamma \neq 0$). We are interested in studying the asymptotic behavior of the functionals $\{E_\varepsilon\}_{\varepsilon>0}$ in the limit $\varepsilon \searrow 0$ in the sense of $\Gamma$-convergence. A natural choice of topology on the space $L^1(\Omega) \times \Mam$ is the topology of strong convergence in $L^1(\Omega)$ for the first component and the flat topology,  i.e., the norm topology in the dual of the space $W^{1,\infty}_0(\Omega, \Rd)$), for the second component. As anticipated above, we consider configurations $(u,\gamma) \in BV(\Omega) \times \Mam$ in the limit. However, for finite energy, the singular part~$\gamma^s$ of the measure~$\gamma$ with respect to the Lebesgue measure~$\mathcal{L}^d$ necessarily needs to coincide with the singular part~$D^s u$ of the measure derivative of~$u$. This compensation of the singular measure derivative in turn leads to energy contributions in the non-local energy term in~$E_\varepsilon$. Before giving the specific form, let us recall that~$D^s u$ can be decomposed into the Cantor part $D^c u$ and the jump part $D^j u = [u] \nu_u \HA \mrs J_u$, where $[u] \coloneqq (u\cdot +)- u(\cdot -)$ is the crack opening with direction $\nu_u \in \Sm$. In the approximation with $W^{1,1}(\Omega)$-function, both contributions result from a locally large derivative, but they behave differently with respect to $(d{-}1)$-dimensional hypersurfaces:  the Cantor part~$D^c u$ vanishes on them, while jumps happen across them. As a consequence, the Cantor part is penalized by the asymptotic value $c_0$ of the function~$f$, while for the jumps we may have to consider all values of~$f$ (without dependence on the direction of the jump), which is taken into account via the function $\theta \colon \R \to [0, \infty)$ defined by 
\begin{equation}
\theta(t) \coloneqq 2 \int_0^1 f \Big(  \frac{\omega_{d-1}}{\omega_{d}} (1-s^2 )^{\frac{d-1}{2}} |t| \Big) \dd s \quad \text{ for all } t \in \R. \label{definitiontheta}
\end{equation}
Here, $\omega_{k}$ denotes  the volume of the $k$-dimensional ball in $\R^k$, for every $k \in \N_0$. We further notice that $\theta(t)$ simplifies to $f(\tfrac{1}{2} |t|)$ for $d=1$, which specified the energy contribution of the jumps in the one-dimensional case study in~\cite{AuerVolkmannBeckSchmidt:22}. Concerning the argument of the integrand in~\eqref{definitiontheta}, we notice from Fubini's theorem that 
\begin{equation}
\omega_{d}= 2 \int_{0}^{1} \omega_{d-1} (1-s^2)^{\frac{d-1}{2}} \dd s,
\label{eqn_omega_d_Fubini}
\end{equation}
which, in view of the concavity of the integrand~$f$, allows to show that the growth of~$\theta(t)$ is bounded from below and above, up to a constant, by $\min\{|t|,1\}$, see Lemma~\ref{propertiesgBV} below. Moreover,~$\theta$ also inherits  the property of subadditivity from~$f$. With the function~$\theta$ from~\eqref{definitiontheta} we can now introduce the limit energy functional $E \colon L^1(\Omega) \times \Mam \to [0, \infty]$ defined by
\begin{equation*}
E(u,\gamma) \coloneqq
\begin{cases}
\int_\Omega {|\nabla u- g|^2 \dd x }+ c_0 \int_\Omega {  |g| \dd x } & \quad \text{if } u \in BV(\Omega), \,  \norm{u}_{L^{\infty}(\Omega)} \leq K, \\ 
\quad + \Ij \theta([u]) \dd \HA  & \quad \phantom{if } \gamma = D^su +g \mathcal{L}^d , \, g \in L^1(\Omega,\R^d),\\  \quad +c_0 |D^cu|(\Omega)
& \quad \phantom{if } \text{and }  \nabla u- g \in L^2(\Omega,\R^d),\\[0.2cm] 
\infty & \quad \text{otherwise}. 
\end{cases} 
\end{equation*}
We emphasize once again that for a finite energy configuration the displacement field~$u$ and the eigendeformation~$\gamma$ are coupled in the following sense: the singular parts~$\gamma^s$ and~$D^s u$ of~$\gamma$ and~$Du$ with respect to~$\mathcal{L}^d$ need to coincide, while their absolutely continuous parts $g \mathcal{L}^d$ and $\nabla u \mathcal{L}^d$ are only required to satisfy an integrability condition $\nabla u - g \in L^2(\Omega,\R^d)$. 

Our main result identifies the energy $E$ as the $\Gamma$-limit of the functionals $\{E_{\varepsilon}\}_{\varepsilon}$.

\begin{theorem}\label{mainresultmBV}
Let $L^1(\Omega)$ be equipped with the strong topology and $\mathcal{M}(\Omega,\R^d)$ be equipped with the flat topology. Then the family $\{E_{\varepsilon}\}_{\varepsilon>0}$ $\Gamma$-converges to~$E$ in $L^1(\Omega) \times \mathcal{M}(\Omega,\R^d)$, i.e., we have
\begin{enumerate}[font=\normalfont, label=(\roman{*}), ref=(\roman{*})]
\item\label{mainresultmBV_1} \emph{\textup{(}$\liminf$-inequality\textup{)}} For every sequence $\{(u_{\varepsilon},\gamma_{\varepsilon})\}_{\varepsilon}$ in $ L^1(\Omega) \times \Mam$ converging to  $(u ,\gamma) \in L^1(\Omega) \times \Mam$, i.e., $u_{\varepsilon} \to u$ in $\Lme$ and $\gamma_{\varepsilon} \to \gamma$ in the flat norm, we have
\begin{equation*}
\liminf_{\varepsilon \to 0} E_{\varepsilon}(u_{\varepsilon},\gamma_{\varepsilon})\geq E(u,\gamma).
\end{equation*} 
\item\label{mainresultmBV_2} \emph{($\limsup$-inequality\textup{)}} For every $(u,\gamma) \in  L^1(\Omega) \times\Mam$, there exists a sequence $\{(u_{\varepsilon},\gamma_{\varepsilon})\}_{\varepsilon}$ in $L^1(\Omega) \times \Mam$ such that  $u_{\varepsilon} \to u $ in $\Lme$, $\gamma_{\varepsilon} \to \gamma$ in the flat norm, and
\begin{equation*}
\limsup_{\varepsilon \to 0} E_{\varepsilon}(u_{\varepsilon}, \gae) \leq E(u,\gamma).
\end{equation*}
\end{enumerate}
\end{theorem}

The corresponding compactness result is given in Theorem~\ref{commBV}, where we even obtain the (stronger) convergence in the negative Sobolev $W^{-1,q}(\Omega, \R^d)$ for all $1< q < 2^*$ for the second component. Hence, the $\Gamma$-convergence of $\{E_{\varepsilon}\}_{\varepsilon}$ to~$E$ in $L^1(\Omega) \times \mathcal{M}(\Omega,\R^d)$ holds also, when the flat topology for $\mathcal{M}(\Omega,\R^d)$ is replaced by the topology of convergence in $W^{-1,q}(\Omega, \R^d)$ for some $1 < q < 2^*$.

For a fixed function $u \in BV(\Omega)$ in the first entry of $E$ a particularly interesting choice of $\gamma$ is given by 
\begin{equation*}
\go \coloneqq D^su + g^* \mathcal{L}^{d}
\end{equation*}
where~$g^*$ is the unique minimizer of the optimization problem
\begin{equation}
\label{optimization_problem}
  \text{to minimize } \int_\Omega |\nabla u- g|^2 \dd x + c_0 \int_\Omega |g| \dd x  \quad \text{among all } g \in L^1(\Omega,\R^d).
\end{equation}
By a pointwise minimization of the integrand, the minimizer~$g^\ast$ can be determined as
\begin{equation}
g^* = \alpha(|\nabla u|) \nabla u \quad \text{with }\alpha(t) = \begin{cases} 1 - \frac{c_0}{2 t} & \text{if } t > \frac{c_0}{2},\\
0 & \text{if } t \leq \frac{c_0}{2}. \label{gammaaopt}
\end{cases}
\end{equation}
More generally, we consider the minimal energies with respect to the second variable $\gamma$, which are given as
\begin{equation}
\label{def_minimal_energy_eps}
\tilde{E}_{\varepsilon}(u) \coloneqq \inf_{\gamma \in \Mam} E_{\varepsilon}(u,\gamma)= \inf_{g \in \Lm} E_{\varepsilon}(u,g \mathcal{L}^d)
\end{equation}
and 
\begin{equation}
\label{def_minimal_energy_limit}
\tilde{E}(u) \coloneqq  \inf_{\gamma \in \Mam} E(u,\gamma) = \inf_{g \in \Lm} E(u,D^su+ g \mathcal{L}^d)
\end{equation}
for $u \in L^1(\Omega)$. We notice that for $u \in BV(\Omega)$ the infimum in~\eqref{def_minimal_energy_limit} is attained, with $\tilde{E}(u) = E(u,\go)$. As a consequence of Theorem~\ref{mainresultmBV}, we obtain a $\Gamma$-convergence result for these energies:

\begin{corollary}\label{mainresult2BV}
Let $L^1(\Omega)$ be equipped with the strong topology. Then the family $\{\tilde{E}_{\varepsilon}\}_{\varepsilon>0}$ $\Gamma$-converges to $\tilde{E}$ in $\Lme$.
\end{corollary}


\section{Preliminaries} \label{sec_preliminaries}

In this section, we collect the general notation used throughout the paper, we recall some basic concepts related to $BV$-functions, and we summarize some facts on the convergence of measures.

\paragraph{General notation.}
In what follows, $\Omega\subset\R^{d}$ denotes an open, bounded set with Lip\-schitz boundary, for some $d \geq 1$. For two vectors $x,y \in \R^d$ the scalar product in $\R^d$ is denoted by $x \cdot y$, the euclidean norm by $|x|$ and the vector $(x_2,\ldots,x_d) \in \R^{d-1}$ of the final $(d-1)$ components by~$x'$. For $x \in\R^{d}$ and $r>0$, we use the abbreviation $B_{r}(x) \coloneqq \{y\in\R^{d}\colon |y-x|<r\}$ for open balls and we write $\mathbb{S}^{d-1}\coloneqq \partial B_{1}(0)$ for the $(d-1)$-dimensional unit sphere. Similarly, for $\nu \in \mathbb{S}^{d-1}$ we use $Q^\nu_{r}(x)$ for a cube centered in~$x$ of side length~$2r$, which has two faces orthogonal to $\nu$, and we write $Q_{r}(x)$ for the axes-aligned one. We denote the $d$-dimensional Lebesgue measure and the $(d-1)$-dimensional Hausdorff measure by $\mathcal{L}^d$ and $\HA$, respectively, for sets $A \subset \R^d$ we usually write $|A| = \mathcal{L}^d(A)$ and we set $\omega_d \coloneqq |B_{1}(0)|$. The family of all open subsets of $\Omega$ is indicated by $\Am$ and the family of all Borel subsets of $\Omega$ by $\Bm$. For a set $A \subset \R^d$ we write $\partial A$ for its topological boundary and $\overline{A}$ for its closure, and we abbreviate the inner and outer parallel sets by 
\begin{align*}
A_{\varrho-} & \coloneqq \lbrace x \in A \colon \dist(x, \partial A) > \varrho \rbrace \\
A_{\varrho +}& \coloneqq \lbrace x \in \Omega \colon \dist(x,A) < \varrho \rbrace, 
\end{align*}
for $\varrho > 0$. Finally, if $A$ is measurable and of finite positive measure $0 < |A| < \infty$ and if $g$ is an integrable function in~$A$, we indicate the mean value of~$g$ on~$A$ by $\dashint_{A} g \dd x$. 

\paragraph{Functions of bounded variation.}  A function $u\in L^1(\Omega)$ is said to belong to the space $BV(\Omega)$ of \emph{functions of bounded variation} if its distributional derivative is a finite $\R^{d}$-valued Radon measure, i.e, if the integration-by-parts formula
\begin{equation*}
\int_{\Omega} u \frac{\partial \varphi}{\partial x_i} \dd x = - \int_{\Omega} \varphi \dd D_iu \quad \text{ for every } \varphi \in C_c^{1}(\Omega) \text{ and } i=1, \dots, d
\end{equation*}
is valid for a (unique) vector measure $Du \in \mathcal{M}(\Omega;\R^{d})$. The space $BV(\Omega)$ is a Banach space endowed with the norm
\begin{equation*}
\norm{u}_{BV(\Omega)} \coloneqq  \norm{u}_{L^1(\Omega)} + |Du|(\Omega),
\end{equation*}
where $|Du|(\Omega)$ is the total variation of~$Du$. We here collect some basic facts on fine properties of $BV$-functions from~\cite{afp}, which are relevant for our paper. We start to discuss approximate continuity and discontinuity properties of $L^1$-functions. We say that a function~$u \in L^1_{\textnormal{loc}}(\Omega)$ has an \emph{approximate limit} at $x \in \Omega$ if there exists a (unique) $\bar{u}(x) \in \R$ such that
\begin{equation*}
\lim_{\varrho \searrow 0} \dashint_{B_{\varrho}(x)} |u(y)-\bar{u}(x)| \dd y =0.  
\end{equation*}
We denote by~$S_u$ the set, where this condition fails, and call it the \emph{approximate discontinuity set} of~$u$. It is $\mathcal{L}^d$-negligible, and $\bar{u}$ coincides $\mathcal{L}^d$-a.e.~in $\Omega \setminus S_u$ with~$u$.  Furthermore, we say that~$u$ has an \emph{approximate jump point} at $x \in S_u$ if there exist $\nu=\nu_u(x) \in \Sm$ and $u(x+), u(x-) \in \R$ with $u(x+) \neq u(x-)$ such that 
\begin{equation*}
\lim_{\varrho \searrow 0} \dashint_{B^+_{\varrho}(x,\nu)} |u(y)-u(x+)| \dd y =0 \quad \text{and} \quad \lim_{\varrho  \searrow 0} \dashint_{B^-_{\varrho}(x,\nu)} |u(y)-u(x-)| \dd y =0, 
\end{equation*}
where $B^{\pm}_{\varrho}(x,\nu)  \coloneqq  \{ y \in B_{\varrho}(x) \colon (y-x) \cdot \nu \in \R^{\pm} \}$. The triplet $(u(x+), u(x-), \nu)$ is uniquely determined up to a permutation of $(u(x+), u(x-))$ and a change of sign of~$\nu$. We denote by~$J_u$ the set of approximate jump points and call it the \emph{jump set} of~$u$. We further notice that, if $u \in BV(\Omega)$, then the set~$S_u$ is countably $\mathcal H^{d-1}$-rectifiable with $\mathcal H^{d-1}(S_u\setminus J_u)=0$ (see \cite[Theorem 3.78]{afp}). 

According to Lebesgue's decomposition theorem we can decompose the measure derivative~$Du = D^au \mathcal{L}^d + D^su$ into the absolutely continuous and the singular part with respect to the Lebesgue measure~$\mathcal{L}^d$. We then define the jump and the Cantor part of~$Du$ as
\begin{equation*}
D^j u  \coloneqq  D^s u \mrs J_u \quad \text{and} \quad  D^c u  \coloneqq   D^s u \mrs (\Omega\setminus S_u).
\end{equation*}
From the identifications $D^au = \nabla u \mathcal{L}^d$ with the approximate gradient~$\nabla u$ for the absolutely continuous part and $D^ju = [u] \nu_u \HA \mrs J_u$ with $[u] \coloneqq (u\cdot +)- u(\cdot -)$ for the jump part (see \cite[Theorem~3.83 and formula (3.90)]{afp}) we arrive at the decomposition
\begin{equation*}
Du= \nabla u \Lmd + [u] \nu_u \HA \mrs J_u + D^c u.
\end{equation*} 
Moreover,  we observe that the Cantor part $D^c u$ vanishes on all Borel sets which are $\sigma$-finite with respect to $\HA$ as well as on all preimages of~$\bar{u}$ of $\mathcal{H}^1$-negligible sets (see \cite[Proposition~3.92]{afp}).  

Finally, we mention the subspace $SBV(\Omega)$ of \emph{special functions of bounded variation}, which contains all functions $u \in BV(\Omega)$ with $D^cu= 0$, and we define
\begin{equation*}
SBV^2(\Omega) \coloneqq  \lbrace u \in SBV(\Omega) \colon \nabla u \in L^2(\Omega, \Rd) \text{ and } \HA(J_u) < \infty \rbrace.
\end{equation*}

\paragraph{Convergence in negative Sobolev spaces and in the flat topology.} The negative Sobolev spaces $W^{-1,q}(\Omega)$ with $1 < q \leq \infty$ are defined as usual as the dual spaces of $W^{1,q'}_0(\Omega)$ (with $q' \in [1,\infty)$ denoting the conjugate exponent to~$q$ with $\tfrac{1}{q} + \tfrac{1}{q'}=1$), and correspondingly the norm is defined via the duality pairing as 
\begin{equation*}
 \| T \|_{W^{-1,q}(\Omega)}  \coloneqq  \sup \Big \{ T(\varphi) \colon \varphi \in W^{1,q'}_0(\Omega) \text{ with }\norm{\varphi}_{W^{1,q'}(\Omega)} \leq 1 \Big\},
\end{equation*}
for every $T \in W^{-1,q}(\Omega)$. Consequently, the spaces $W^{-1,q}(\Omega)$ with $1 < q < \infty$ are reflexive and separable. We mention two specific situations. Let $v \in L^q(\Omega,\R^d)$ and $w \in L^r(\Omega)$ with~$1 \leq r \leq \infty$ such that the embedding $W^{1,q'}_0(\Omega) \subset L^{r'}(\Omega)$ is continuous. If we set
\begin{equation*}
 T_{{\rm div} \, v}(\varphi)  \coloneqq  - \int_\Omega v \cdot \nabla \varphi \diff x \quad \text{and} \quad  T_w(\varphi)  \coloneqq  \int_\Omega w \cdot \varphi \diff x \quad \text{for all } \varphi \in W^{1,q'}_0(\Omega), 
\end{equation*}
then, by the H{\"o}lder inequality and the continuous embedding $W^{1,q'}_0(\Omega) \subset L^{r'}(\Omega)$ (with constant~$C'$), we obtain $T_{{\rm div} \, v}, T_w \in W^{-1,q}(\Omega)$ with 
\begin{equation*}
 \| T_{{\rm div} \, v} \|_{W^{-1,q}(\Omega)} \leq \norm{v}_{L^q(\Omega,\R^d)} \quad \text{and} \quad \| T_w \|_{W^{-1,q}(\Omega)} \leq C' \norm{w}_{L^r(\Omega)}. 
\end{equation*}
Thus, in particular (partial deriviatives of) $L^q$-functions can be considered as elements of the space $W^{-1,q}(\Omega)$. Furthermore, because of the compact embedding $W^{1,q'}_0(\Omega) \Subset C_0(\Omega)$ for $q'>d$, the negative Sobolev norms can actually be considered on the space~$\Mamone$ of all finite Radon measures on~$\Omega$: for $\mu \in \Mamone$, the duality pairing reads as 
\begin{equation*}
 \| \mu \|_{W^{-1,q}(\Omega)} = \sup \bigg\{ \int_\Omega \varphi \dd \mu \colon \varphi \in W^{1,q'}_0(\Omega) \text{ with } \norm{\varphi}_{W^{1, q'}_0(\Omega)}\leq 1 \bigg\},
\end{equation*}
and if we allow $q' = \infty$ in this expression, we obtain the flat norm
\begin{equation*}
\normf{\mu} \coloneqq  \sup \bigg\{ \int_\Omega \varphi \dd \mu \colon \varphi \in W^{1, \infty}_0(\Omega) \text{ with } \norm{\varphi}_{W^{1, \infty}_0(\Omega)}\leq 1 \bigg\},
\end{equation*}
with $\norm{\varphi}_{W^{1, \infty}_0(\Omega)} \coloneqq \norm{\nabla \varphi}_{L^{\infty}(\Omega,\R^d)}$. Here we have the corresponding inequalities 
\begin{equation*}
 \normf{{\rm div} \, v} \leq  \norm{v}_{L^1(\Omega,\R^d)}  \quad \text{and} \quad  \|w \mathcal{L}^d\|_{\text{flat}} \leq \norm{w}_{L^1(\Omega)} 
\end{equation*}
for all functions $v \in BV(\Omega,\R^d)$ and $w \in L^1(\Omega)$. We notice that due to Schauder's theorem and the compact embedding $W^{1,\infty}_0(\Omega) \Subset C_0(\Omega)$, the flat topology metrizes weak-$\ast$ convergence of (uniformly bounded) measures. Therefore, we have the following relations for the convergence of measure with respect to convergence in $W^{-1,q}(\Omega)$, the flat norm and in the weak-$\ast$ sense. 

\begin{lemma}[on convergence of measures] 
\label{Lemma_weak_negativ_flat}
Let $\Omega  \subset \R^d$ be an open, bounded set. For a measure $\mu \in \Mamone$ and a sequence $\{\mu_n\}_{n \in \N}$ of measures in $\Mamone$, we have: 
\begin{enumerate}[font=\normalfont, label=(\roman{*}), ref=(\roman{*})]
 \item\label{Lemma_weak_negativ_flat_1} If $\mu_n \to \mu$ in $W^{-1,q}(\Omega)$ for some  $1 < q \leq \infty$, then $\mu_n \to \mu$ in the flat norm.
 \item\label{Lemma_weak_negativ_flat_2} If $\mu_n \overset{*}{\rightharpoonup} \mu$ in $\Mamone$, then $\mu_n \to \mu$ in the flat norm.
 \item\label{Lemma_weak_negativ_flat_3} If $\mu_n \to \mu$ in the flat norm and $\sup_{n \in \N} |\mu_n|(\Omega) < \infty$, then $\mu_n \overset{*}{\rightharpoonup} \mu$ in $\Mamone$.
\end{enumerate}
\end{lemma}

\section{Compactness}\label{sec_Compactness}

In this section we show a compactness result for sequences in $L^1(a,b) \times \mathcal{M}(a,b)$ with bounded energy~$E_\varepsilon$, which, together with the $\Gamma$-convergence result of Theorem \ref{mainresultmBV}, implies the convergence of (almost) minimizers. It is convenient to introduce localized versions of the functionals~$\{E_{\varepsilon}\}_{\varepsilon}$ and~$E$, which are defined for every open set~$A \in \Am$ via
\begin{align*}  
E_{\varepsilon}(u,\gamma,A)  \coloneqq  
\begin{cases}
\int_A |\nabla u- g|^2 \dd x & \quad \text{if } u \in W^{1,1}(\Omega), \,  \norm{u}_{L^{\infty}(\Omega)} \leq K, \\ + \frac{1}{\varepsilon} \int_A {f \big( \varepsilon \dashint_{B_{\varepsilon}(x) \cap \Omega}{ |g| } \dd y} \big) \dd x  & \quad \phantom{if } \gamma= g\mathcal{L}^d, \, g \in L^1(\Omega,\R^d),\\
 & \quad \phantom{if } \text{and } \nabla u- g \in L^2(\Omega,\R^d), \\[0.2cm]
\infty & \quad \text{otherwise},
\end{cases}
\end{align*}
and 
\begin{equation*}
E(u,\gamma,A) \coloneqq
\begin{cases}
\int_A {|\nabla u- g|^2 \dd x }+ \int_A { c_0 |g| \dd x } & \quad \text{if } u \in BV(\Omega), \,  \norm{u}_{L^{\infty}(\Omega)} \leq K, \\ 
+  \int_{J_u \cap A} \theta([u]) \dd \HA & \quad \phantom{if } \gamma = D^su + g \mathcal{L}^d,, \, g \in L^1(\Omega,\R^d),\\ 
+ c_0 |D^cu|(A) & \quad \phantom{if } \text{and }  \nabla u- g \in L^2(\Omega,\R^d),\\[0.2cm] 
\infty & \quad \text{otherwise}. 
\end{cases} 
\end{equation*}
For the $\Gamma$-lower and $\Gamma$-upper limit of $\{E_{\varepsilon}\}_{\varepsilon}$ we write
\begin{align*}
E^{\prime}(u, \gamma) & \coloneqq  \inf \big\{ \liminf_{\varepsilon \to 0} E_{\varepsilon}(\ue,\gae) \colon \ue \to u \text{ in } L^1(\Omega), \, \gae \to \gamma \text{ in the flat norm}\big\}, \\
 E^{\prime \prime}(u,\gamma) & \coloneqq  \inf \big\{ \limsup_{\varepsilon \to 0} E_{\varepsilon}(\ue,\gae) \colon \ue \to u \text{ in } L^1(\Omega), \, \gae \to \gamma \text{ in the flat norm}\big\},
\end{align*}
respectively. Moreover, we use the notation $E^\prime(\cdot,\cdot,A)$ for the localized $\Gamma$-lower limit, for every $A \in \Am$. We start with a technical proposition, which allows to pass from a finite $E_\varepsilon$-energy configuration $(u,\gamma) \in \Wme \times \Mam$ to a configuration $(v,\breve{\gamma}=\breve{g}\mathcal{L}^d) \in SBV(A) \times \mathcal{M}(A, \R^{d})$ such that essentially~$u$ and~$v$ are $L^1$-close, the jump set of~$v$ is energy bounded in some inner parallel set and $L^1$-bounds of~$\nabla v$ and~$\breve{g}$ are available. 

\begin{Prop}\label{clmBV}
Let $A \subset \Omega$ be open, $\varepsilon > 0$ and $\delta >0$. For every $(u,\gamma = g\mathcal{L}^d) \in \Wme \times \Mam$ with $E_\varepsilon(u,\gamma)<\infty$ there exists $(v,\breve{\gamma}=\breve{g}\mathcal{L}^d) \in SBV(A) \times \mathcal{M}(A, \R^{d})$ such that 
\begin{align*}
&  \int_A |\nabla v - \breve{g}|^2 \dd x + (1- \delta) c_0 \int_A |\breve{g}|\dd x \leq E_{\varepsilon}(u,\gamma,A), \\
&(1-\delta) c_0 \int_A \normb{\nabla v}  \dd x  \leq E_{\varepsilon}(u,\gamma,A) + c_0  E_{\varepsilon}(u,\gamma,A)^{\frac{1}{2}} |A|^{\frac{1}{2}} ,\\
&\HA(J_v \cap A_{6\varepsilon-}) \leq c E_{\varepsilon}(u,\gamma,A),\\
& \big| \{x \in A_{6 \varepsilon -} \colon \nabla v(x) - \breve{g}(x) \neq \nabla u(x) - g(x) \} \big| \leq \varepsilon c E_{\varepsilon}(u,\gamma,A), \\
&\norm{v}_{L^{\infty}(A)} \leq \norm{u}_{L^{\infty}(A)},\\
&c_0 \norm{v-u}^q_{L^q(A_{6 \varepsilon-})} \leq \varepsilon c E_{\varepsilon}(u,\gamma,A)  \norm{u}^q_{L^{\infty}(A)} \quad \text{for all } 1 \leq q < \infty,
\end{align*}
where~$c$ is a constant depending only on~$d$ and~$\delta$.
\end{Prop}

\begin{proof}
We follow a slightly modified strategy of proof from \cite[Proposition 4.1]{BDM97}. The main (technical) difficulty is the passage from open squares (which are disjoint and whose union covers~$A$ up to a negligible set) to open balls (which appear in the definition of the nonlocal part of the energy). To overcome this problem, we rely upon the inequality
\begin{equation}
\sum_{\alpha \in \Z^d} \frac{(s \varepsilon)^d}{|B_{\varepsilon}|} \chi_{B_{\varepsilon}(s \varepsilon \alpha)}(x) \geq 1 - \delta \quad \text{on } \R^d, \label{saBV}
\end{equation}
which holds for a suitable $s \in (0, {1}/{\sqrt{d}})$ depending only on~$d$ and on~$\delta$ (see \cite[Lemma 4.3]{BDM97}).

Let $\phi_{\varepsilon} \in C^{\infty}_0(A)$ be a cut-off function with $0 \leq \phi_{\varepsilon} \leq 1$ in~$A$ and $\phi_{\varepsilon} \equiv 1$ in~$A_{\varepsilon-}$. We then define a function $\psi_{\varepsilon} \in C_0(\R^d)$ via 
\begin{equation*}
\cff(x) \coloneqq \phi_\varepsilon(x) f \bigg( \varepsilon \dashint_{B_{\varepsilon}(x) \cap \Omega} |g| \dd y \bigg) 
\end{equation*}
for $x \in A$ and we set $\cff(x) \coloneqq 0$ for $x \in \R^d \setminus A$. From \cite[Lemma 4.2]{BDM97} (with $\eta = s \varepsilon$), we infer
\begin{equation*}
\int_{\R^d} \cff(x)\dd x=\sum_{\alpha \in \Z^d} (s \varepsilon)^d \, \cff(x_\varepsilon + s \varepsilon \alpha) 
\end{equation*}
for a suitable $x_\varepsilon \in \R^d$. In what follows, we shall work with the (disjoint) cubes $Q_{s \varepsilon/2} (x_\alpha)$ and the (overlapping) balls $B_\varepsilon(x_\alpha)$ with centers $x_\alpha \coloneqq x_\varepsilon + s \varepsilon \alpha$ for $\alpha \in \Z^d$, for which we notice the inclusion  $Q_{s \varepsilon/2} (x_\alpha) \subset B_\varepsilon(x_\alpha)$ from $s \sqrt{d} < 1$. We further set 
\begin{equation*}
\mathscr{G}_{\varepsilon} \coloneqq \big\{ x_\alpha = x_\varepsilon + s \varepsilon \alpha \in \R^d \colon \alpha \in \Z^d \text{ and } x_\alpha \in A_{\varepsilon-} \big\}
\end{equation*}
(meaning that we have in particular $B_{\varepsilon}(x_\alpha) \subset A$ for all $x_\alpha \in \mathscr{G}_{\varepsilon}$). From the non-negativity of~$f$ and the properties of the function~$\phi_\varepsilon$ we then observe 
\begin{align}
\frac{1}{\varepsilon} \int_A f \bigg(\varepsilon \dashint_{B_{\varepsilon}(x) \cap \Omega} |g| \dd  y \bigg) \dd x & \geq \frac{1}{\varepsilon} \int_{\R^d} \cff(x)\dd x \geq \sum_{x_\alpha \in \mathscr{G}_{\varepsilon}} \frac{(s \varepsilon)^d}{\varepsilon} \cff(x_\alpha)\notag \\
&  =  \sum_{x_\alpha \in \mathscr{G}_{\varepsilon}}  \frac{(s \varepsilon)^d}{\varepsilon} f \bigg( \varepsilon \dashint_{B_{\varepsilon}(x_\alpha)} |g| \dd y \bigg) . \label{clm3BV}
\end{align}
In order to modify~$g$ in such a way that an $L^1$-estimate is available, we select two types of centers $x_\alpha \in \mathscr{G}_{\varepsilon}$ such that~$f$ on the right-hand side of~\eqref{clm3BV} behaves linearly, namely
\begin{align*}
\mathscr{G}_{\varepsilon}' & \coloneqq \bigg\{ x_\alpha \in \mathscr{G}_{\varepsilon} \colon \varepsilon \dashint_{B_{\varepsilon}(x_\alpha)} |g| \dd y < 1 \bigg\}, \\
\mathscr{G}_{\varepsilon}'' & \coloneqq \bigg\{ x_\beta \in \mathscr{G}_{\varepsilon} \cap A_{5\varepsilon -} \colon \varepsilon \dashint_{B_{3 \varepsilon}(x_\beta)} |g| \dd y < 3^{-d} \bigg\}.
\end{align*}
From these choices we immediately find the implication
\begin{equation*}
 x_\beta \in \mathscr{G}_{\varepsilon}'' \text{ with } Q_{s\varepsilon/2}(x_\beta) \cap B_\varepsilon(x_\alpha) \neq \emptyset \quad \Rightarrow \quad B_\varepsilon(x_\alpha) \subset B_{3\varepsilon}(x_\beta) \text{ and } x_\alpha \in \mathscr{G}_{\varepsilon}'
\end{equation*}
for each $\alpha \in \Z^d$ (which in particular yields the inclusion $\mathscr{G}_{\varepsilon}'' \subset \mathscr{G}_{\varepsilon}'$). With $f(t) = c_0 t$ for $t \in [0,1]$ and inequality~\eqref{saBV} (with centers of the balls shifted by~$x_\varepsilon$) we then continue to estimate~\eqref{clm3BV} from below by
\begin{align}
 \frac{1}{\varepsilon} \int_A f \bigg(\varepsilon \dashint_{B_{\varepsilon}(x) \cap \Omega} |g| \dd y \bigg) \dd x & \geq \sum_{x_\alpha \in \mathscr{G}_{\varepsilon}'} \frac{(s \varepsilon)^d}{|B_\varepsilon|} c_0 \int_{B_{\varepsilon}(x_\alpha)} |g| \dd y \notag \\
 & \geq  \sum_{x_\beta \in \mathscr{G}_{\varepsilon}''} \sum_{x_\alpha \in \mathscr{G}_{\varepsilon}'} \frac{(s \varepsilon)^d}{|B_\varepsilon|} c_0 \int_{ Q_{s\varepsilon/2}(x_\beta) \cap B_{\varepsilon}(x_\alpha)} |g| \dd y \notag \\
 & =  \sum_{x_\beta \in \mathscr{G}_{\varepsilon}''} \sum_{\alpha \in \Z^d} \frac{(s \varepsilon)^d}{|B_\varepsilon|} c_0 \int_{ Q_{s\varepsilon/2}(x_\beta) } \chi_{B_{\varepsilon}(x_\alpha)}  |g| \dd y \notag \\
 & \geq \sum_{x_\beta \in \mathscr{G}_{\varepsilon}''} (1-\delta) c_0 \int_{ Q_{s\varepsilon/2}(x_\beta) } |g| \dd y . \label{comp_energy_linear}
\end{align} 
We next define functions $v \colon A \to \R$ and $\breve{g}\colon A \to \R^d$ via 
\begin{align}
\label{vemBV}
v(x) \coloneqq  \begin{cases}
u(x) & \quad  x \in \bigcup \{Q_{s\varepsilon/2}(x_\beta) \colon x_\beta \in \mathscr{G}_{\varepsilon}''\},  \\ 
0 & \quad \text{otherwise},  
\end{cases} \\
\label{gvemBV}
\breve{g}(x) \coloneqq  \begin{cases}
g(x) & \quad  x \in \bigcup \{Q_{s\varepsilon/2}(x_\beta) \colon x_\beta \in \mathscr{G}_{\varepsilon}''\},  \\ 
0 & \quad \text{otherwise}.
\end{cases} 
\end{align}
In view of $u \in W^{1,1}(\Omega) \cap L^{\infty}(\Omega)$, we have $v \in SBV(A)$ with $\norm{v}_{L^{\infty}(A)}\leq \norm{u}_{L^{\infty}(A)}$, and since~$g \in L^1(\Omega,\R^d)$ we also get $\breve{g} \in L^1(A, \Rd)$. Moreover, with the inequality $|\nabla u - g| \geq |\nabla v - \breve{g}|$ in~$A$ and the previous estimate~\eqref{comp_energy_linear} we conclude
\begin{align}
E_{\varepsilon}(u,\gamma,A) & = \int_A |\nabla u - g|^2 \dd x + \frac{1}{\varepsilon} \int_A f \bigg(\varepsilon \dashint_{B_{\varepsilon}(x) \cap \Omega} |g| \dd y \bigg) \dd x \notag \\
& \geq \int_A |\nabla v - \breve{g}|^2 \dd x + (1- \delta) c_0 \int_A |\breve{g}|\dd y , \label{egaBV}
\end{align}
which proves the first claim of the proposition. Via~\eqref{egaBV} and the Cauchy--Schwarz inequality, we immediately also find the second claim:
\begin{align*}
 (1- \delta)  c_0  \int_{A}|\nabla v|\dd y & \leq (1- \delta) c_0 \int_A |\breve{g}|\dd y +  c_0 \bigg( \int_{A}|\nabla v-\breve{g}|^2 \dd y \bigg)^{\frac{1}{2}} |A|^{\frac{1}{2}}\\
 & \leq E_{\varepsilon}(u,\gamma,A) + c_0  E_{\varepsilon}(u,\gamma,A)^{\frac{1}{2}} |A|^{\frac{1}{2}}.
\end{align*}
We next address the size of jump set~$J_v$ of~$v$ in $A_{6\varepsilon-}$. By construction, there holds
\begin{equation*}
J_v \cap A_{6\varepsilon-} \subset \bigcup \big\{ \partial Q_{s\varepsilon/2}(x_\beta) \colon x_\beta \in (\mathscr{G}_{\varepsilon} \cap A_{5\varepsilon -}) \setminus \mathscr{G}_{\varepsilon}''\big\}.
\end{equation*}
Since points $x_\beta \in (\mathscr{G}_{\varepsilon} \cap A_{5\varepsilon -}) \setminus \mathscr{G}_{\varepsilon}''$ are characterized via estimates from below for mean values on balls with radius $3 \varepsilon$, we need to find related estimates from below for mean values on balls with radius~$\varepsilon$, in order to bound the number $\# ((\mathscr{G}_{\varepsilon}\cap A_{5\varepsilon -}) \setminus \mathscr{G}_{\varepsilon}'')$ in terms of the energy. To this end, we first notice 
\begin{equation*}
 \frac{\varepsilon}{3^d} \sum_{\{x_\alpha \in \mathscr{G}_{\varepsilon} \colon B_\varepsilon(x_\alpha) \cap B_{3 \varepsilon}(x_\beta) \neq \emptyset\} } \dashint_{B_{\varepsilon}(x_\alpha)} |g| \dd y \geq \varepsilon \dashint_{B_{3\varepsilon}(x_\beta)} |g| \dd y \geq 3^{-d}
\end{equation*}
for $x_\beta \in (\mathscr{G}_{\varepsilon} \cap A_{5\varepsilon -}) \setminus \mathscr{G}_{\varepsilon}''$, which means that there exists $x_\alpha  \in \mathscr{G}_{\varepsilon}$ with $B_\varepsilon(x_\alpha) \cap B_{3 \varepsilon}(x_\beta) \neq \emptyset$ and 
\begin{equation*}
 \varepsilon \dashint_{B_{\varepsilon}(x_\alpha)} |g| \dd y \geq \frac{1}{\sigma_{s,d}} \quad \text{with  }\sigma_{s,d} \coloneqq \# \{ x_\alpha \in \mathscr{G}_{\varepsilon} \colon B_\varepsilon(x_\alpha) \cap B_{3 \varepsilon}(x_\beta) \neq \emptyset \}.
\end{equation*}
Since by a shifting argument it is in turn clear that
\begin{equation*}
 \sigma_{s,d} = \# \{ x_\beta \coloneqq x_\varepsilon + s \varepsilon \beta \in \R^d \colon B_\varepsilon(x_\alpha) \cap B_{3\varepsilon}(x_\beta)\}, 
\end{equation*}
we obtain
\begin{equation*}
  \# \big((\mathscr{G}_{\varepsilon}\cap A_{5\varepsilon -}) \setminus \mathscr{G}_{\varepsilon}''\big) \leq \sigma_{s,d} \#  \bigg\{ x_\alpha \in \mathscr{G}_{\varepsilon} \colon \varepsilon \dashint_{B_{\varepsilon}(x_\alpha)} |g| \dd y \geq  \frac{1}{\sigma_{s,d}}  \bigg\}.
\end{equation*}
In consequence, we observe from~\eqref{clm3BV} that
\begin{equation} 
\frac{(s \varepsilon)^d}{\varepsilon} \# \big((\mathscr{G}_{\varepsilon}\cap A_{5\varepsilon -}) \setminus \mathscr{G}_{\varepsilon}''\big) \leq  \frac{\sigma_{s,d}}{f(\sigma_{s,d}^{-1})} E_\varepsilon(u,\gamma,A) . \label{clm3bBV}
\end{equation} 
Thus, we obtain
\begin{align*}
\HA(J_v  \cap A_{6\varepsilon-}) &\leq \# \big((\mathscr{G}_{\varepsilon}\cap A_{5\varepsilon -}) \setminus \mathscr{G}_{\varepsilon}''\big) \HA ( \partial Q_{s \varepsilon/2}) \\
 &  = \# \big((\mathscr{G}_{\varepsilon}\cap A_{5\varepsilon -}) \setminus \mathscr{G}_{\varepsilon}''\big)  2 d (s \varepsilon)^{d-1} \leq \frac{2 d}{s} \frac{\sigma_{s,d}}{f(\sigma_{s,d}^{-1})} E_{\varepsilon}(u,\gamma,A). 
\end{align*}
We finally show the estimate for the Lebesgue measure of the set on which $\nabla v - \breve{g}$ and $\nabla u - g$ differ as well as the estimate for $v-u$ in $L^q(A_{2\varepsilon-})$. Since~$v$ and $\breve{g}$ coincide with~$u$ and $g$, respectively, on the set $\cup \{Q_{s\varepsilon/2}(x_\alpha) \colon x_\alpha \in \mathscr{G}_{\varepsilon}''\}$, we can use once again the estimate~\eqref{clm3bBV} and find with $ |Q_{s \varepsilon/2}| =(s\varepsilon)^d$
\begin{multline*}
\big| \{x \in A_{6 \varepsilon -} \colon \nabla v(x) - \breve{g}(x) 
 \neq \nabla u(x) - g(x) \} \big|  \\
  \leq  \# \big((\mathscr{G}_{\varepsilon}\cap A_{5\varepsilon -}) \setminus \mathscr{G}_{\varepsilon}''\big)  |Q_{s \varepsilon/2}|  \leq  \varepsilon \frac{\sigma_{s,d}}{f(\sigma_{s,d}^{-1})} E_{\varepsilon}(u,\gamma,A)
\end{multline*}
and analogously
\begin{equation*}
\norm{v-u}^{q}_{L^{q}(A_{6\varepsilon-})} \leq \varepsilon \frac{\sigma_{s,d}}{f(\sigma_{s,d}^{-1})} E_{\varepsilon}(u,\gamma,A)\norm{u}^q_{L^{\infty}(A)}.
\end{equation*}
This completes the proof of the proposition.
\end{proof}

\begin{theorem}[Compactness]\label{commBV}
Let $\{(u_{\varepsilon},\gamma_\varepsilon)\}_{\varepsilon}$ be a sequence in $L^1(\Omega) \times \Mam$ with 
\begin{equation}
E_{\varepsilon}(u_{\varepsilon}, \gamma_{\varepsilon}) \leq C_0 \quad \text{for all } \varepsilon>0 \label{bmBV}
\end{equation} 
for a positive constant~$C_0$. There exist a function $u \in BV(\Omega)$ with $\norm{u}_{L^{\infty}(\Omega)} \leq K$ and a measure $\gamma \in \Mam$ such that, up to subsequences, $\{u_{\varepsilon}\}_{\varepsilon}$ converges to $u$ in $L^1(\Omega)$ and $\{\gamma_{\varepsilon}\}_{\varepsilon}$ converges to $\gamma=\gamma^s+ g \mathcal{L}^d$ in $W^{-1,q}(\Omega, \R^d)$ for all $q < 2^*$ and in particular in the flat norm. Moreover, there holds $\gamma^s = D^su$ and $\nabla u- g \in L^2(\Omega,\Rd)$.
\end{theorem}

\begin{proof}
Let $A \Subset \Omega$ with smooth boundary $\partial A$. By Proposition \ref{clmBV} there exists a sequence $\{v_{\varepsilon}\}_{\varepsilon}$ in $SBV(A)$ with $\norm{v_{\varepsilon}}_{L^{\infty}(A)} \leq K$ and $\norm{v_{\varepsilon}}_{BV(A)} \leq M$ for all $\varepsilon>0$ and a constant~$M$ which is independent of~$A$ and~$\varepsilon$. By the Rellich--Kondrachov theorem there exists a function $u \in BV(A)$ with $\norm{u}_{BV(A)}\leq M$ such that $\{v_{\varepsilon}\}_{\varepsilon}$ converges, up to subsequences, to $u$ in $L^q(A)$ for all $q < 1^* = {d}/{(d-1)}$. Since  $\norm{v_{\varepsilon}-u_{\varepsilon}}_{L^q(A_{6\varepsilon-})} \to 0$ as $\varepsilon \to 0$, we deduce that $\{u_{\varepsilon}\}_{\varepsilon}$ converges to $u$ in $L^q(A_{6\varepsilon-})$ for all $q < 1^*$. By arbitrariness of~$A$, a diagonal argument and the uniform bound $\norm{u}_{BV(A)}\leq M$, we find a subsequence which converges in $L^q_{\text{loc}}(\Omega)$ to a function $u \in BV(\Omega)$. Finally, the uniform bound $\norm{u_{\varepsilon}}_{L^{\infty}(\Omega)} \leq K$ implies also strong convergence $u_\varepsilon \to u$ in $L^q(\Omega)$ and thus $\Du_{\varepsilon} \to Du$ in $W^{-1,q}(\Omega, \R^d)$ for all $1 \leq q < \infty$. Moreover, it holds $\norm{u}_{L^{\infty}(\Omega)}\leq K$. From the energy bound~\eqref{bmBV} we further infer that
$\{\nabla u_{\varepsilon} - \ggae\}_{\varepsilon}$ is bounded in $L^2(\Omega, \R^d)$. Therefore, it converges, up to a subsequence, weakly in $L^2(\Omega,\R^d)$ to some $w \in L^2(\Omega, \Rd)$.  Since $L^2(\Omega, \Rd)$ is compactly embedded in $W^{-1,q}(\Omega, \R^d)$ for all $q< 2^*$, this shows that $\{\gae= g_\varepsilon \mathcal{L}^d\}_{\varepsilon}$ converges in $W^{-1,q}(\Omega, \R^d)$ for all $q < 2^*$ and in the flat norm (cf.~Lemma~\ref{Lemma_weak_negativ_flat}) to some 
$\gamma \in W^{-1,q}(\Omega, \R^d)$.  By the uniqueness of the limit we obtain $\gamma = Du - w \mathcal{L}^d \in \Mam$. This in particular shows that $\gamma^s = D^su$ and $w =\nabla u - g \in L^2(\Omega,\R^d)$.
\end{proof}


\section{A relaxation result in $BV(\Omega)$}\label{sec:relBV}

In this section we prove a relaxation result for free discontinuity functionals with linear growth in the space ~$BV$ with respect to $L^1$-convergence, which will be needed in Section~\ref{sec:upper-lim}. We here consider only functionals with a specific bulk density, depending only on the approximate gradient and a continuous perturbation. To introduce this functional, we fix a function $v \in C^{\infty}_0(\overbar{\Omega},\Rd)$ and define $\psi \colon \oOmega \times \Rd \to \R$ via
\begin{equation*}
\psi(x, \xi) \coloneqq |v(x)|^2 + c_0 |\xi + v(x)| \quad \text{for } (x,\xi) \in \oOmega \times \Rd.
\end{equation*}
For later convenience, we introduce the functional $\F_v \colon L^1(\Omega) \times \Am \to [0, \infty]$ in its localized version, defined by 
\begin{align*}
\F_v(u,A) \coloneqq \left\{
\begin{array}{l l}
\int_{A}  \psi (x, \nabla u)\dd x + \int_{J_u \cap A} {\theta}([u]) \dd \HA &  \text{if } u\restrict{A}  \in SBV^2(A)\cap L^{\infty}(A), \\ 
\infty & \text{otherwise},
\end{array}
\right.
\end{align*}
with $\theta \colon \R \to [0,\infty)$ as in~\eqref{definitiontheta}, and we want to study its relaxation $\overbar{\F}_v \colon L^1(\Omega) \times \Am \to [0, \infty]$ with respect to strong convergence in $L^1(\Omega)$, given as 
\begin{align*}
\overbar{\F}_v(u,A) \coloneqq \inf\big\{   \liminf_{h \to \infty}  \F_v(u_h,A) \colon \{u_h\}_h \text{ in } L^1(\Omega) \text{ with } u_h \to u \text{ in } L^1(\Omega) \big\}. 
\end{align*}
The main result of this section is the following upper bound on the relaxed functional, needed only on the whole domain~$\Omega$.

\begin{theorem}\label{rel2BV}
The relaxed functional $\overbar{\F}_v(\cdot,\Omega)$ of $\F_v(\cdot,\Omega)$ with respect to strong convergence in $L^1(\Omega)$ satisfies the estimate
\begin{equation}
\overbar{\F}_v(u,\Omega) \leq   \int_{\Omega}  \psi(x, \nabla u )\dd x+ \Ij {\theta}([u])  \dd \HA  + c_0 |D^cu|(\Omega) \label{rfBV}
\end{equation}
for every $u \in BV(\Omega) $.
\end{theorem}

We first collect some relevant properties of the functions~$\psi$ given above and~$\theta$ from~\eqref{definitiontheta}. 

\begin{lemma}\label{propertiesgBV}
The functions~$\psi$ and~$\theta$ satisfy the following properties:
\begin{enumerate}[font=\normalfont, label=(\roman{*}), ref=(\roman{*})]
\item\label{propertiesgBV_1} $\psi$ is \textup{(}jointly\textup{)} continuous on $\overbar{\Omega} \times \Rd$,
\item\label{propertiesgBV_2} $\psi$ is of linear growth, i.e. there exists a constant $C_1 = C_1(c_0,\norm{v}_{L^{\infty}(\Omega, \Rd)})$ with
 \begin{equation}\label{minest2BV}
 c_0 |\xi| - C_1 \leq \psi(x,\xi) \leq C_1(1+|\xi|) \quad \text{for all } (x,\xi) \in \oOmega \times \Rd,
 \end{equation}
 \item\label{propertiesgBV_3} the \textup{(}strong\textup{)} recession function of~$\psi$ is given by 
 \begin{equation*}
  \psi^{\infty}(x, \xi) \coloneqq  \lim \limits_{ \substack {t \to \infty \\ (y,\zeta) \to (x,\xi) } } \frac{\psi (y ,t \zeta)}{t}  = c_0 |\xi| \quad \text{ for all } (x,\xi) \in \overbar{\Omega} \times \Rd,
 \end{equation*} 
 \item\label{propertiesgBV_4} there exists a constant $c_1=c_1(d)>0$ with
 \begin{equation}
 \label{minestBV}
 c_1 c_0 \min\{ |t|, 1\} \leq  {\theta}(t) \leq 2 c_0 \min\{ |t|, 1\} \quad \text{for all } t  \in \R.
 \end{equation}
\end{enumerate}
\end{lemma}

\begin{proof}
We first notice that the smoothness of~$\psi$ in~\ref{propertiesgBV_1} follows directly from its definition and the smoothness of~$v$. Next, the estimate~\eqref{minest2BV} in~\ref{propertiesgBV_2} is obtained as a consequence of the smoothness of $v$, since we have
 \begin{equation*}
 \psi(x, \xi) \leq |v(x)|^2 + c_0 |\xi| + c_0| v(x)| \leq c_0 |\xi| +  c_0^2 +  \norm{v}_{L^{\infty}(\Omega, \Rd)}^2 
 \end{equation*}
 and
 \begin{equation*}
 \psi(x, \xi) \geq  |v(x)|^2 + c_0 |\xi| -c_0| v(x)| \geq c_0 |\xi| - c_0^2. 
 \end{equation*}
These bounds in turn directly imply the form of the recession function as stated in~(iii). Concerning the last claim~(iv) we first notice via the concavity of~$f$, Jensen's inequality and the formula~\eqref{eqn_omega_d_Fubini}
\begin{equation*}
 \theta(t) = 2 \int_0^1 f \Big(  \frac{\omega_{d-1}}{\omega_{d}} (1-s^2 )^{\frac{d-1}{2}} |t| \Big) \dd s \leq 2 f \Big( \frac{|t|}{2} \Big) = c_0 \min\{|t|,2\}.
\end{equation*}
Morover, using once again~\eqref{eqn_omega_d_Fubini}, we further find 
\begin{equation*}
 \theta(t) \geq 2 c_0 \min\Big\{ |t| \frac{\omega_{d-1}}{\omega_{d}},1\Big\} \int_0^1 (1-s^2 )^{\frac{d-1}{2}} \dd s =  c_0 \min\Big\{ |t|,\frac{\omega_{d}}{\omega_{d-1}}\Big\}. 
\end{equation*}
These two inequalities directly yield~\eqref{minestBV}. 
\end{proof}

As the first important ingredient of the proof of Theorem~\ref{rel2BV} we show in several steps that for every $u \in BV(\Omega)$, the set function $\overbar{\F}_v(u,\, \cdot \,) \colon \Am \to [0,\infty]$ is the restriction to $\Am$ of a regular Borel measure. 

\begin{lemma}[linear growth]\label{lgBV}
For every $ u \in BV(\Omega)$ and $ A \in \Am$ we have 
\begin{equation*} 
\overbar{\F}_v(u,A) \leq C_1 \big( \mathcal{L}^d(A) + |Du|(A) \big).
\end{equation*}
\end{lemma}

\begin{proof}
We first notice, that in view of the linear growth condition~\eqref{minest2BV} for~$\psi$, there holds
\begin{equation}
\overbar{\F}_v(u,A)\leq \F_v(u,A) =\int_{A}  \psi (x,\nabla u) \dd x  \leq  C_1 (\mathcal{L}^d(A) + |Du|(A))\label{gaBV}
\end{equation}
for a smooth function $u \in C^{\infty}(\overbar{\Omega})$. For a general function $u \in BV(\Omega)$ we use the density of $C^{\infty}(\overbar{\Omega})$ in~$BV(\Omega)$ with respect to strict convergence (see~\cite[Theorem~3.9 \& Remark~3.22]{afp}) to find a sequence $\{u_h\}_h$ in $C^{\infty}(\Omega)$ with $u_h \to u$ in $L^1(\Omega)$ and $ |Du_h|(A) \to |Du|(A)$ as $h \to \infty$. With the definition of~$\overbar{\F}_v$, the claim then follows directly from~\eqref{gaBV}.
\end{proof}

\begin{lemma}[almost subadditivity]\label{weaksubaddBV}
Let $A^{\prime}_1,A_1, A_2 \in \Am$ with $A^{\prime}_1 \Subset A_1$. For every $u  \in L^1(\Omega)$ there holds
\begin{equation*}
\overbar{\F}_v(u, A^{\prime}_1 \cup A_2) \leq \overbar{\F}_v(u, A_1) + \overbar{\F}_v(u,A_2).
\end{equation*}
\end{lemma}

\begin{proof}
We consider two sequences $\{u_h\}_h$ and $\{v_h\}_h$ in $L^1(\Omega)$ converging to~$u$ in $L^1(\Omega)$ with
\begin{equation*}
\overbar{\F}_v(u, A_1) = \lim_{h \to \infty} \F_v(u_h, A_1) \quad \text{and} \quad
\overbar{\F}_v(u, A_2) = \lim_{h \to \infty} \F_v(v_h, A_2). 
\end{equation*}
We can restrict ourselves to the case that the right-hand sides are finite, which means that we have $u_h \restrict{A_1} \in SBV^2(A_1) \cap L^{\infty}(A_1)$ and $v_h\restrict{A_2} \in SBV^2(A_2)\cap L^{\infty}(A_2)$ for all $h \in \N$.
We now follow  a well-known procedure, introduced by De Giorgi~\cite{DEGIORGI75} (see also~\cite{dm}), that allows to glue the two sequences in order to construct a sequence  $\{w_h\}_h$ in $L^1(\Omega)$, which converges to~$u$ in $L^1(\Omega)$, satisfies $w_h \restrict{A^{\prime}_1 \cup A_2} \in SBV^2(A^{\prime}_1 \cup A_2)  \cap L^{\infty}(A^{\prime}_1 \cup A_2)$ for each $h \in \N$, and provides the desired energy bound.

Let $k \in \N$, set $\delta \coloneqq \dist(A^{\prime}_1, \partial A_1) > 0$, and consider a chain of open sets 
\begin{equation*}
 B_0 \coloneqq A^{\prime}_1 \Subset B_1 \Subset \ldots \Subset B_k \Subset B_{k+1} \coloneqq A_1
\end{equation*}
with $\dist(B_i, \partial B_{i+1}) \geq {\delta}/(k+1)$ for every $i \in \{0,\ldots,k\}$.
We fix $i \in \{0,\ldots,k\}$, select a cut-off function $\varphi_i \in C^{\infty}_c(\Omega)$ between~$B_i$ and~$B_{i+1}$, i.e., with $0 \leq \varphi_i \leq 1$ in~$\Omega$, $\spt(\varphi_i) \subset B_{i+1}$, $\varphi_i \equiv 1$ in an open neighborhood~$U_i$ of $\overbar{B}_{i}$, and $\norm{\nabla \varphi_i}_{L^{\infty}(\Omega)} \leq 2 (k+1)/{\delta}$, and define  
\begin{equation*}
w^i_h \coloneqq \varphi_i u_h + (1-\varphi_i) v_h.
\end{equation*}
Obviously, there holds $w^i_h \restrict{A^{\prime}_1 \cup A_2} \in SBV^2(A^{\prime}_1 \cup A_2)  \cap L^{\infty}(A^{\prime}_1 \cup A_2)$ for each $h \in \N$ and $\{w^i_h\}_h$ converges to~$u$ in $L^1(\Omega)$. Since we have $w^i_h=u_h$ on $U_i \supset \overbar{B}_{i}$ and $w^i_h=v_h$ on $A_2 \setminus \spt(\varphi_i) \supset A_2 \setminus B_{i+1}$, we infer from the locality and non-negativity~of~$\F_v$ 
\begin{align*}
\F_v(w^i_h,A^{\prime}_1\cup A_2) &\leq \F_v(w^i_h,U_i) + \F_v(w^i_h,A_2 \cap(B_{i+1}\setminus \overbar{B}_{i})) +\F_v(w^i_h, A_2 \setminus \spt(\varphi_i) ) \\
& \leq  \F_v(u_h,A_1) + \F_v(w^i_h, A_2  \cap (B_{i+1}\setminus \overbar{B}_{i})) +\F_v(v_h, A_2). 
\end{align*}
We next study the energy of~$w^i_h$ on the open set $D_i \coloneqq A_2  \cap (B_{i+1}\setminus \overbar{B}_{i})$, that is
\begin{align*}
\F_v(w^i_h,D_i) = \int_{D_i} \psi(x, \nabla w^i_h) \dd x + \int_{J_{w^i_h}\cap D_i} {\theta}([w^i_h]) \dd  \HA.
\end{align*}
As the measure derivative of~$w^i_h$ is given by 
\begin{equation*}
Dw^i_h= \varphi_i Du_h + (1-\varphi_i)Dv_h + (\nabla \varphi_i  (u_h-v_h)) \mathcal{L}^d, 
\end{equation*}
we deduce from the linear growth condition~\eqref{minest2BV} for~$\psi$
\begin{align*}
\lefteqn{\int_{D_i} \psi(x, \nabla w^i_h) \dd x  = \int_{D_i} \psi(x, \varphi_i \nabla u_h + (1-\varphi_i) \nabla v_h + \nabla \varphi_i (u_h-v_h))\dd x }\\  
 &\leq C_1 \left[ \mathcal{L}^d(D_i)+ \int_{D_i} |\nabla u_h| \dd x  + \int_{D_i} |\nabla v_h| \dd x+ \int_{D_i} |\nabla \varphi_i  (u_h-v_h)|\dd x \right]\\
 &\leq C^{\prime} \left[ \mathcal{L}^d(D_i)+ \int_{D_i} \psi(x,\nabla u_h) \dd x  + \int_{D_i} \psi(x, \nabla v_h) \dd x+ (k+1)\int_{D_i} |u_h-v_h|\dd x \right],
\end{align*}
with a constant~$C^{\prime}$ depending only on~$c_0$, $v$ and~$\delta$. 
In order to estimate the surface term, we notice that the jump set of the sum of two $BV$-functions is included, by the chain rule, in the union of the jump sets of the single functions.
Thus, via~\eqref{minestBV} we obtain
\begin{align*}
\lefteqn{\int_{J_{w^i_h}\cap D_i} {\theta}([w^i_h])\dd \HA \leq 
2c_0 \int_{J_{w^i_h}\cap D_i} \min\left\lbrace |[w^i_h]|, 1\right\rbrace \dd  \HA }\\
&\leq 
 2c_0  \int_{J_{u_h}\cap D_i} \min\left\lbrace |[u_h]|, 1\right\rbrace \dd \HA + 2c_0 \int_{J_{v_h}\cap D_i} \min\left\lbrace |[v_h]|, 1\right\rbrace \dd \HA\\
 &\leq \frac{2}{c_1} \int_{J_{u_h}\cap D_i}{\theta}([u_h]) \dd \HA + \frac{2}{c_1} \int_{J_{v_h}\cap D_i} {\theta}([v_h]) \dd \HA.
\end{align*}
Collecting the previous estimates, we get 
\begin{align*}
\F_v(w^i_h,A^{\prime}_1\cup A_2) &\leq \F_v(u_h,A_1) + \F_v(v_h,A_2)  + C^{\prime \prime} \big[\F_v(u_h,D_i) + \F_v(v_h,D_i) + \mathcal{L}^d(D_i) \big]  \\& \quad +C^{\prime \prime} (k+1)\int_{D_i} |u_h-v_h|\dd x ,
\end{align*}
with $C^{\prime \prime} = \max\{ C^{\prime}, \tfrac{2}{c_1} \}$. Summing this inequality for $i \in\{0, \ldots, k\}$, taking the average with respect to~$i$ and using $\dot{\cup}_{i=1}^{k} D_i = A_2 \cap (A_1 \setminus \overbar{A^{\prime}_1})$ (which is contained in each of the sets~$A_1$, $A_2$ and~$\Omega$), we can choose an index~$i_h$ with
\begin{align*}
\F_v(w^{i_h}_h,A^{\prime}_1\cup A_2) 
& \leq \F_v(u_h,A_1) + \F_v(v_h,A_2) + \frac{ C^{\prime \prime}}{k+1} \left[\F_v(u_h,A_1) + \F_v(v_h,A_2) + \mathcal{L}^d(\Omega) \right]  \\& \quad  +C^{\prime \prime} \int_{A_2 \cap (A_1 \setminus \overbar{A^{\prime}_1})} |u_h-v_h|\dd x.
\end{align*}
We now define $w_h=w^{i_h}_h$ for $h \in \N$. Letting first $h \to \infty$ (recall that $\{u_h\}_h$ and $\{v_h\}_h$ both converge in $L^1(\Omega)$ to~$u$) and then $k \to \infty$ in the previous estimate we end up with
\begin{align*}
\overbar{\F}_v(u, A^{\prime}_1 \cup A_2) \leq \liminf_{h \to \infty}{\F}_v(w_h, A^{\prime}_1 \cup A_2)  &\leq \lim_{h \to \infty}{\F}_v(u_h, A_1) +  \lim_{h \to \infty} \F_v(v_h,A_2)\\
& \, = \overbar{\F}_v(u, A_1) + \overbar{\F}_v(u, A_2),
\end{align*}
which completes the proof of the lemma.
 \end{proof}

\begin{Prop}\label{sfBV}
For every $u \in BV(\Omega)$ the set function $\overbar{\F}_v(u,\, \cdot \,) \colon \Am \to [0,\infty]$ is the restriction to $\Am$ of a regular Borel measure.
\end{Prop}

\begin{proof} 
We follow the strategy of the proof of \cite[Proposition 3.3]{BBB95}. We first notice that the set function~$\F_v(u,\, \cdot \,)$ is increasing and superadditive (on disjoint sets), which implies that also $\overbar{\F}_v(u,\, \cdot \,)$ is increasing and superadditive. In order to prove the claim, according to the measure property criterion of De Giorgi--Letta from \cite[Theorem 5.6]{DGL77} (see also \cite[Theorem 1.53]{afp}) we only need to show that $\overbar{\F}_v(u,\, \cdot \,)$ is inner regular and also subadditive. In order to prove the inner regularity, we show 
\begin{equation}
\label{eqn_F_relaxation_approx}
{\overbar{\F}}_v(u,A) = \sup \big\{ \overbar{\F}_v(u,B) \colon  B \in \Am, B \Subset A\big\} 
\end{equation}
for every $A \in \Am$. Since $\overbar{\F}_v(u,\, \cdot \,)$ is increasing, we only need to prove the $\leq$-inequality. For this purpose, we fix $\varepsilon >0$ and choose a compact set $K \Subset A$ with $\mathcal{L}^d(A \setminus K) < {\varepsilon}/{2}$ and $|Du|(A \setminus K) < {\varepsilon}/{2}$. For sets $A^{\prime}, A^{\prime \prime} \in \Am$ with $K \Subset A^{\prime} \Subset A^{\prime \prime} \Subset A$ we then infer from Lemma~\ref{weaksubaddBV} and Lemma~\ref{lgBV}
\begin{align*}
\overbar{\F}_v(u,A) = \overbar{\F}_v(u, A^{\prime} \cup( A \setminus K)) & \leq \overbar{\F}_v(u, A^{\prime \prime}) + \overbar{\F}_v(u,  A \setminus K) \\
& \leq  \overbar{\F}_v(u, A^{\prime \prime}) +  C_1\big(\mathcal{L}^d(A \setminus K) + |Du|(A \setminus K) \big)\\
& \leq \sup \{ \overbar{\F}_v(u,B) \colon B \in \Am, B \Subset A \} + C_1 \varepsilon,
\end{align*}
which, in the limit $\varepsilon \to 0$, concludes the proof of the inner regularity of $\overbar{\F}_v(u,\, \cdot \,)$. Finally, we notice that the subadditivity of $\overbar{\F}_v(u,\, \cdot \,)$, i.e.,
\begin{equation*}
\overbar{\F}_v(u,A_1 \cup A_2) \leq \overbar{\F}_v(u,A_1) + \overbar{\F}_v(u,A_2) \quad \text{for all } A_1, A_2 \in \Am,
\end{equation*}
is an immediate consequence of Lemma~\ref{weaksubaddBV} and the inner regularity. 
\end{proof}

Before proving Theorem~\ref{rel2BV}, we still need to study the behavior of $\F_v$ under translations.  

\begin{lemma}\label{moduluscontinuity}
There exists a constant $C_2 = C_2(c_0,\norm{v}_{W^{1,\infty}(\Omega, \Rd)},\Omega)$ such that for all $u \in BV(\Omega)$, $A \in \Am$, $b \in \R$ and $ x_0 \in \R^d$ with $x_0 + A \subset \Omega$ we have
\begin{equation*}
| \overbar{\F}_v(u(\scdot - x_0)+b,x_0+A)-  \overbar{\F}_v(u,A) | \leq C_2 |x_0|  \mathcal{L}^d(A) .
\end{equation*} 
\end{lemma}

\begin{proof}
We first consider arbitrary $u \in L^1(\Omega)$, $A \in \Am$, $b \in \R$ and $x_0 \in \R^d $ with $u\restrict{A}  \in SBV^2(A)\cap L^{\infty}(A)$ and $x_0 + A \subset \Omega$. Then we obtain by substitution and the Lipschitz continuity of $v$ with constant ${\rm Lip}(v)$
\begin{align*}
& |\F_v(u(\scdot - x_0)+b,x_0+A)-  \F_v(u,A) | \\& \leq  \bigg| \int_{x_0+A} |v(x)|^2 \dd x - \int_{A} |v(x)|^2 \dd x \bigg| \\
 & \quad + \bigg| \int_{x_0+A}  c_0 |\nabla u(x-x_0) +v(x)| \dd x  - \int_{A}  c_0 |\nabla u(x) +v(x)| \dd x \bigg| \\
 & \quad +  \bigg|\int_{J_u \cap (x_0 + A)} {\theta}([u(x-x_0)]) \dd \HA(x) - \int_{J_u \cap A} {\theta}([u(x)]) \dd \HA(x) \bigg| \\
 & \leq \bigg| \int_{A} [ (v(y +x_0)+ v(y))(v(y+x_0)-v(y))]  \dd y \bigg| + \bigg| \int_{A}  c_0 |v(y + x_0) - v(y)| \dd y  \bigg|\\
 & \leq \big( 2 \norm{v}_{L^{\infty}(\Omega, \Rd)} {\rm Lip}(v) + c_0 \norm{\nabla v}_{L^{\infty}(\Omega, \R^{d \times d})}\big)  |x_0| \mathcal{L}^d(A) \eqqcolon  C_2 |x_0|  \mathcal{L}^d(A). 
\end{align*}
Via the passage to the relaxation we directly arrive at the claim.
\end{proof}

\begin{proof}[Proof of Theorem~\ref{rel2BV}]
We wish to apply \cite[Theorem 3.12]{BFM98} to the relaxed functional, but we first need to regularize the functional by adding a $\delta$-small total variation measure~$|Du|$, in order to make a suitable lower bound on the functional available. Therefore, we consider the functionals $\overbar{\F}_{v, \delta} \colon BV(\Omega) \times \Am \to [0, \infty]$ defined via
\begin{equation*}
\overbar{\F}_{v, \delta}(u, A) \coloneqq \overbar{\F}_v(u, A) + \delta |Du|(A)
\end{equation*}
for all $u \in L^1(\Omega)$, $A \in \Am$ and $\delta > 0$. We now verify that for every $\delta > 0$ the assumptions for the application of \cite[Theorem 3.12]{BFM98} are verified.
\begin{enumerate}[font=\normalfont, label=(\roman{*}), ref=(\roman{*})]
 \item $\overbar{\F}_{v, \delta}(u, \scdot)$ is the restriction to $\Am$ of a Radon measure, for every $u \in BV(\Omega)$. This is a direct consequence of Proposition~\ref{sfBV} and the definition of~$\overbar{\F}_{v,\delta}$. 
 \item $\overbar{\F}_{v, \delta}(\scdot, A)$ is $L^1(A)$-lower semicontinuous for every $A \in \Am$. We first observe from \cite[Proposition 16.15]{dm} that $\overbar{\F}_{v}$ is local, i.e., for every $A \in \Am$ and $u, \tilde{u} \in BV(\Omega)$ with $u =\tilde{u}$ a.e.~in $A$ there holds $\overbar{\F}_{v}(u,A) = \overbar{\F}_{v}(\tilde{u},A)$. Then the $L^1(A)$-lower semicontinuity of~$\overbar{\F}_v$ is clear from the definition of the relaxation, while the $L^1(A)$-lower semicontinuity of the total variation is well-known. 
 \item For the constant $C_1 = C_1(c_0,\norm{v}_{L^{\infty}(\Omega, \Rd)})$ from Lemma~\ref{propertiesgBV} there holds
 \begin{equation*}
 \delta |Du|(A) \leq \overbar{\F}_{v, \delta}(u,A) \leq (C_1 + \delta ) ( \mathcal{L}^d(A) + |Du|(A))
 \end{equation*}
 for every $u \in BV(\Omega)$ and $A \in \Am$. These bounds follow from Lemma~\ref{lgBV} and the definition of $\overbar{\F}_{v, \delta}$.
 \item For the constant $C_2 = C_2(c_0,\norm{v}_{W^{1,\infty}(\Omega, \Rd)},\Omega)$ from Lemma~\ref{moduluscontinuity} there holds
 \begin{equation*}
 |\overbar{\F}_{v, \delta}(u(\scdot - x_0)+b,x_0+A)-  \overbar{\F}_{v, \delta}(u,A) |\leq C_2|x_0| \big( \mathcal{L}^d(A) + |Du|(A) \big)
 \end{equation*}
 for all $u \in BV(\Omega)$, $A \in \Am$, $b \in \R$ and $ x_0 \in \R^d$ with $x_0 + A \subset \Omega$. Because of $|Du(\cdot - x_0)|(x_0 + A) = |Du|(A)$, this is an immediate consequence of Lemma~\ref{moduluscontinuity}.
\end{enumerate}
In view of \cite[Theorem 3.12]{BFM98} we obtain the representation
\begin{align*}
 \overbar{\F}_{v, \delta}(u, A) & = \int_A b_\delta(x,u,\nabla u) \dd x + \int_{J_u \cap A} \sigma_\delta(x,u^+,u^-,\nu_u) \dd \HA \\
 & \quad + \int_A b_\delta^\infty(x,u, \dd D^c u / \dd |D^c u| ) \dd  |D^c u|
\end{align*}
for every $u \in BV(\Omega)$ and $A \in \Am$, where the integrands~$b_\delta$ and~$\sigma_\delta$ are defined 
via optimization problems of the form
\begin{equation*}
  m_\delta(u,A) \coloneqq \inf \big\{ \overbar{\F}_{v, \delta}(w, A) \colon w \in BV(\Omega), \, w\restrict{\partial A} = u\restrict{\partial A} \big\}
\end{equation*}
with cubes as the set~$A$ and either affine or pure jump functions as boundary values~$u$, and where~$b^\infty_\delta$ is the recession function. Notice that this is designed to capture the interaction between the bulk energy density and the surface energy density. More precisely, we have
\begin{align*}
 b_\delta(x_0,a,\xi) & \coloneqq \limsup_{\varepsilon \to 0} \frac{ m_{\delta}(a + \xi (\scdot -x_0), Q_{\varepsilon}(x_0))}{(2\varepsilon)^d},  \\
 \sigma_\delta(x_0,\lambda,\mu,\nu) & \coloneqq \limsup_{\varepsilon \to 0} \frac{ m_{\delta}(u_{\lambda, \mu}^{\nu}(\scdot -x_0), Q^\nu_{\varepsilon}(x_0))}{(2\varepsilon)^{d-1}},
\end{align*}
for all $x_0 \in \overbar{\Omega}$, $a, \lambda, \mu \in \R$, $\xi \in \Rd$ and $\nu \in \Sm$. 
Here, $u_{\lambda, \mu}^{\nu}$ denotes the pure jump function 
\begin{equation*}
u_{\lambda, \mu}^{\nu}(x)\coloneqq  \begin{cases} \lambda & \text{ if } x \cdot \nu >0,\\
\mu & \text{ if } x \cdot \nu < 0.
\end{cases}
\end{equation*}
From the above representation of~$\overbar{\F}_{v, \delta}$ we next derive the estimate~\eqref{rfBV}, essentially by bounding~$b_\delta$ and~$\sigma_\delta$ from above, as we first plug the affine or pure jump function determining the boundary values into the functional $\overbar{\F}_{v, \delta}$ and then estimate the relaxation~$\overbar{\F}_{v}$ by the original functional~$\F_{v}$. Using also the continuity of $x \mapsto \psi(x,\xi)$, we hence find
\begin{align*}
 b_\delta(x_0,a,\xi) &\leq \limsup_{\varepsilon \to 0}  (2\varepsilon)^{-d}  \overbar{\F}_{v, \delta}(a + \xi (\scdot -x_0), Q_{\varepsilon}(x_0))\\ 
 &\leq \limsup_{\varepsilon \to 0} (2\varepsilon)^{-d}  \big[ \F_{v}(a + \xi (\scdot -x_0), Q_{\varepsilon}(x_0)) + \delta |\xi| \mathcal{L}^d(Q_{\varepsilon}(x_0)) \big] \\
 &= \limsup_{\varepsilon \to 0}(2\varepsilon)^{-d}  \int_{Q_{\varepsilon}(x_0)}  \psi (x, \xi)\dd x +  \delta  |\xi| \\
 & =  \psi (x_0, \xi) + \delta |\xi|
\end{align*} 
and
\begin{align*}
 \sigma_{\delta}(x_0, \lambda, \mu, \nu)& \leq \limsup_{\varepsilon \to 0} (2\varepsilon)^{-d+1} \overbar{\F}_{v, \delta}(u_{\lambda, \mu}^{\nu} ( \scdot - x_0), Q_{\varepsilon}^{\nu}(x_0)) \\ 
 &\leq \limsup_{\varepsilon \to 0}(2\varepsilon)^{-d+1} \big[ \F_{v}(u_{\lambda, \mu}^{\nu} ( \scdot - x_0), Q_{\varepsilon}^{\nu}(x_0)) + \delta |D u_{\lambda, \mu}^{\nu}|(Q_{\varepsilon}^{\nu}(x_0)) \big] \\
 &= \limsup_{\varepsilon \to 0} (2\varepsilon)^{-d+1}  \int_{J_{u_{\lambda, \mu}^{\nu}} \cap Q_{\varepsilon}^{\nu} }{\theta}([u_{\lambda, \mu}^{\nu}]) \dd \HA +  \delta |\lambda-  \mu| \\ 
 & =  \theta(\lambda-  \mu) + \delta |\lambda-  \mu|.
\end{align*} 
Furthermore, we infer $b_\delta^\infty(x_0,a,\xi) \leq c_0 |\xi| + \delta |\xi|$ from Lemma~\ref{propertiesgBV}~\ref{propertiesgBV_3}. Combining these three estimates with the representation and definition of~$\overbar{\F}_{v, \delta}$, we conclude 
\begin{align*}
 \overbar{\F}_v(u, \Omega) & =  \overbar{\F}_{v, \delta}(u, \Omega) - \delta |Du|(\Omega) \\
 & \leq \int_{\Omega}  \psi(x, \nabla u )\dd x + \delta |D^a u|(\Omega) + \Ij {\theta}([u]) \dd \HA  + \delta |D^j u|(\Omega) \\
 & \quad + c_0 |D^cu|(\Omega) + \delta |D^c u|(\Omega) -  \delta |Du|(\Omega) \\
 & =  \int_{\Omega}  \psi(x, \nabla u )\dd x+ \Ij {\theta}([u])  \dd \HA  + c_0 |D^cu|(\Omega)
\end{align*}
for every $u \in BV(\Omega)$, which was the claim of Theorem~\ref{rel2BV}.
\end{proof}

In Section \ref{sec:upper-lim} we will apply a slightly modified version of Theorem \ref{rel2BV}, where we additionally impose an $L^{\infty}$-bound. For a given positive constant $K>0$, we consider the functional $\G_v \colon L^1(\Omega) \to [0, \infty]$, defined by 
\begin{equation*}
\G_v(u) \coloneqq \left\{
\begin{array}{l l}
\Io \psi( x, \nabla u) \dd x + \int_{J_u} {\theta}([u]) \dd \HA &  \text{if } u \in SBV^2(\Omega), \,  \norm{u}_{L^{\infty}(\Omega)} \leq K, \\ 
\infty & \text{otherwise}.
\end{array}
\right.
\end{equation*}

\begin{corollary}\label{rel2conBV}
The relaxed functional~$\overbar{\G}_v$ of~$\G_v$ with respect to strong convergence in $L^1(\Omega)$ satisfies the estimate
\begin{equation}
\overbar{\G}_v(u) = \overbar{\F}_v(u,\Omega) \leq \Io \psi(x, \nabla u) \dd x + \int_{J_u} \theta([u]) \dd \HA  + c_0 |D^cu|(\Omega) \label{rf2BV}
\end{equation}
for every $u \in BV(\Omega)$ with $\norm{u}_{L^{\infty}(\Omega)} \leq K$.
\end{corollary}

\begin{proof}
Since $\F_v(u,\Omega) \leq \G_v(u)$ for all $u \in L^1(\Omega)$, we have $\overbar{\F}_v(u,\Omega) \leq \overbar{\G}_v(u)$ for all $u \in L^1(\Omega)$. In order to establish the reverse inequality for $u \in BV(\Omega)$ with $\norm{u}_{L^{\infty}(\Omega)} \leq K$, we consider a sequence $\{u_h\}_h$ in $SBV^2(\Omega) \cap L^{\infty}(\Omega)$ with $u_h \to u$ in $L^1(\Omega)$ and $\F_v(u_h,\Omega) \to \overbar{\F}_v(u,\Omega)$.
Because of $u_h \to u$ in $L^1(\Omega)$ and $\norm{u}_{L^{\infty}(\Omega)} \leq K$, we can select a sequence $\{\eta_{h}\}_{h}$ in~$\R^+$ with $\eta_{h} \to 0$ and
\begin{equation*}
 \mathcal{L}^d( \{ x \in \Omega: |u_h(x)| \geq K + \eta_{h} \}) \to 0 \quad \text{as } h \to \infty. 
\end{equation*}
We then define sequences $\{\tilde{u}_h\}_h$ and $\{v_h\}_h$, via
\begin{equation*}
\tilde{u}_h \coloneqq  \min \{ \max \{ u_h, -K - \eta_{h} \} , K + \eta_{h} \}  \quad \text{and} \quad v_h(x) \coloneqq \begin{cases} 
v(x) & \text{if } \tilde{u}_h(x) = u_h(x),\\
0 & \text{otherwise},
\end{cases}
\end{equation*}
for every $h \in \N$. Then, we have $\tilde{u}_h \in SBV^2(\Omega) \cap L^{\infty}(\Omega)$ for every $h \in \N$ with $\tilde{u}_h \to u$ in $L^1(\Omega)$, and we further have $v_h \to v$ in $L^1(\Omega, \Rd)$, since there holds
\begin{equation*}
\int_{\Omega} |v_h - v| \dd x \leq \norm{v}_{L^{\infty}(\Omega, \Rd)} \mathcal{L}^d( \{ x \in \Omega \colon|u_h(x)| \geq K + \eta_{h} \}) \to 0 \quad \text{ as } h \to \infty. 
\end{equation*}
We next adjust the $L^\infty$-bound, by defining sequences $\{\hat{u}_h\}_h$ and $\{\hat{v}_h\}_h$, via
\begin{equation*}
\hat{u}_{h} \coloneqq \frac{K}{K + \eta_h} \tilde{u}_h \quad \text{and} \quad \hat{v}_{h} \coloneqq \frac{K}{K + \eta_{h}} v_h 
\end{equation*}
for every $h \in \N$. Since $\eta_h \to 0$, we still have $\hat{u}_{h} \to u$ in $L^1(\Omega)$ and $\hat{v}_{h} \to v$ in $L^1(\Omega, \Rd)$ as $h \to \infty$. From the definitions of $\hat{u}_h$ and $\hat{v}_h$, we have $\hat{u}_h \in SBV^2(\Omega) \cap L^{\infty}(\Omega)$, with the additional bound $\norm{\hat{u}_{h}}_{L^{\infty}(\Omega)} \leq K$ and with
\begin{equation*}
 \int_{\Omega}|\nabla u_h + v| \dd x \geq \int_{\Omega}|\nabla \hat{u}_h + \hat{v}_h| \dd x \quad  \text{ and } \quad 
 \int_{J_u} {\theta}([{u}_h]) \dd \HA \geq  \int_{J_u} {\theta}([\hat{u}_h]) \dd \HA 
\end{equation*} 
for every $h \in \N$. Hence, we deduce 
\begin{align*}
\overbar{\F}_v(u,\Omega) &= \lim_{h \to \infty} {\F}_v({u}_h,\Omega)\\
 &\geq \liminf_{h \to \infty}\left[ \int_{\Omega}|v|^2 \dd x + \int_{\Omega}c_0|\nabla \hat{u}_h + \hat{v}_h| \dd x + \int_{J_u} {\theta}([\hat{u}_h]) \dd \HA\right]  \\
&\geq \liminf_{h \to \infty} {\G}_v(\hat{u}_h) - \limsup_{h \to \infty}  \int_{\Omega}c_0 |\hat{v}_h - v| \dd x \geq \overbar{\G}_v(u).
\end{align*}
Consequently, we have $\overbar{\F}_v(u,\Omega)=\overbar{\G}_v(u)$ for all $u \in BV(\Omega)$ with $\norm{u}_{L^{\infty}(\Omega)} \leq K$. The claim~\eqref{rf2BV} now follows directly from the estimate~\eqref{rfBV} in Theorem \ref{rel2BV}.
\end{proof}


\section{Estimate from below of the $\Gamma$-lower limit}
\label{sec: estimate from below}

We here prove the $\Gamma$-$\liminf$-inequality. By the compactness result of Theorem~\ref{commBV} it is sufficient to consider $(u,\gamma) \in  BV(\Omega) \times \Mam$ with $\norm{u}_{L^{\infty}(\Omega)} \leq K$, $\gamma = D^su+ g \mathcal{L}^d$ and $\nabla u- g \in L^2(\Omega)$.
As in the one-dimensional setting~\cite{AuerVolkmannBeckSchmidt:22}, the estimate of~$E_{\varepsilon}$ from below  splits up into three separate estimates involving a surface term, a volume term and a Cantor term. Our estimate of the surface term is obtained by a blow-up argument, whereas we use the slicing method for the Cantor term. The estimate of the volume term, however, needs to be proved directly, since it is not clear how to apply the slicing method or a blow-up argument to the weakly converging term depending on $\gamma$, as $\Mam$ is merely endowed with the flat topology.

\subsection{Estimate from below of the surface term}

We apply Besicovitch’s differentiation theorem with respect to $\HA \mrs J_u$ to the $\Gamma$-lower limit considered to be a set function for fixed $(u, \gamma) \in BV(\Omega) \times \Mam$ (Lemma \ref{derivativesingularpartBV}).
Next, an appropriate estimate from below for the density of the lower bound is found by a rescaling argument (Lemma \ref{derivativesingularpartBV2}), showing that only the non-local term contributes. Finding such an estimate is then reduced to a minimization problem (Lemma \ref{surfacebv2withoutslicing}) on a suitable class of Radon measures, which can be solved directly by the majorization principle of Hardy, Littlewood and P\'{o}lya (Proposition~\ref{surfacebv2withoutslicing2}).

For every $A \in \Am$ we denote the inner regular envelope of the $\Gamma$-$\liminf$ $E^{\prime}$ by
\begin{equation}
\label{def_Psi_prime}
\Psi^{\prime}(\, \cdot \, , \,\cdot \, , A) \coloneqq \sup \big\{ E^{\prime}( \, \cdot \, , \,\cdot \, , B) \colon B \in \Am, B \Subset A\big\}.
\end{equation}
Clearly, $\Psi^{\prime}( \scdot, \scdot, A)$ is lower semicontinuous in $L^1(\Omega) \times \Mam$ for every $A\in \Am$, and $\Psi^{\prime}( u, \gamma, \scdot)$ is an increasing, inner regular and superadditive set function on~$\Am$ for every $(u, \gamma) \in L^1(\Omega) \times \Mam$.

\begin{lemma}\label{derivativesingularpartBV}
For every $(u,\gamma) \in \Lm \times \Mam$ the set function $\Psi^{\prime}(u, \gamma, \scdot) \colon \Am \to \R \cup \{\infty\}$ is the restriction to $\Am$ of a regular Borel measure on~$\Omega$. Moreover, if $u \in BV(\Omega)$, $\gamma =  D^su + g \mathcal{L}^d $ and $\nabla u - g \in L^2(\Omega, \R^d)$, then for every $A \in \Am$ there holds
\begin{equation*}
\Psi^{\prime}(u, \gamma, A) \geq \int_{J_u \cap A} h \dd  \HA \quad \text{with } h(x) \coloneqq  \lim_{\varrho \searrow 0}\frac{\Psi^{\prime}(u, \gamma, B_{\varrho}(x))}{\omega_{d-1} \varrho^{d-1}} \text{ for $\HA$-a.e. } x \in J_u.
\end{equation*}
\end{lemma}

\begin{proof}
\emph{Step 1: $\{E_{\varepsilon}\}_{\varepsilon}$ satisfies a \textup{(}generalized\textup{)} fundamental estimate: for every $A^{\prime}_1, A_1, A_2 \in \Am$ with $A^{\prime}_1 \Subset A_1$ and every $\eta>0$, there exist $\varepsilon_0 >0$ and $M >0$ such that 
\begin{multline}
E_{\varepsilon}( \varphi u_1 + (1-\varphi) u_2, [\varphi g_1 + (1-\varphi) g_2 + \nabla \varphi (u_1 - u_2)]\mathcal{L}^d, A^{\prime}_1 \cup A_2) \\
 \leq (1 + \eta)  \big[ E_{\varepsilon}( u_1 , \gamma_1, A_1) +  E_{\varepsilon}( u_2 , \gamma_2, A_2) \big] + M \norm{u_1-u_2}_{L^1(\Omega)} \label{fundamentalestimate}
\end{multline}
holds for every $\varepsilon \leq \varepsilon_0$,  $(u_1,\gamma_1), (u_2,\gamma_2) \in L^\infty(\Omega) \times \Mam$ with $\gamma_1= g_1 \mathcal{L}^d, \gamma_2 = g_2 \mathcal{L}^d$ for functions $g_1, g_2 \in L^1(\Omega,\Rd)$ for a suitable cut-off function $\varphi \in C^{\infty}_0 (\Omega)$ with $0 \leq \varphi \leq 1$ in~$\Omega$, $\spt(\varphi) \subset A_1$, $\varphi \equiv 1$ on $A^{\prime}_1$ and $\norm{\nabla \varphi}_{L^{\infty}(\Omega, \Rd)} \leq M$.} We proceed in a similar way as in \cite[Proposition 4.3 and Theorem 4.6]{C98} (and as in Section~\ref{sec:relBV}), but here we additionally have to include the variable $\gamma$. We choose $k \in \N$ with $2/k < \eta$ and $2c_0/k < 1$, set $\delta \coloneqq \dist(A^{\prime}_1, \partial A_1) > 0$ and $M=4(k+2)/\delta$, and consider a chain of open sets 
\begin{equation*}
 B_0 \coloneqq A^{\prime}_1 \Subset B_1 \Subset \ldots \Subset B_{k+1} \Subset B_{k+2} \coloneqq A_1
\end{equation*}
with $\dist(B_i, \partial B_{i+1}) \geq {\delta}/(k+2)$ for every $i \in \{0,\ldots,k+1\}$. For each $i \in \{1,\ldots,k\}$, we select a cut-off function $\varphi_i \in C^{\infty}_c(\Omega)$ between~$B_i$ and~$B_{i+1}$, i.e., with $0 \leq \varphi_i \leq 1$ in~$\Omega$, $\spt(\varphi_i) \subset (B_{i+1})_{\delta/4(k+2)-}$, $\varphi_i \equiv 1$ in an open neighborhood of~$(\overbar{B}_{i})_{\delta/4(k+2)+}$, and $\norm{\nabla \varphi_i}_{L^{\infty}(\Omega)} \leq M$. For $(u_1,\gamma_1), (u_2,\gamma_2) \in L^\infty(\Omega) \times \Mam$ as above we now  show that the claim~\eqref{fundamentalestimate} holds with $\varphi = \varphi_i$ for some $i\in \{0,\ldots,k-1\}$, for the choice $\varepsilon_0 \coloneqq \delta /4(k+2)$. To this end, we set 
\begin{equation*}
u^i \coloneqq \varphi_i u_1 + (1-\varphi_i) u_2 \quad \text{and} \quad \gamma^i \coloneqq g^i \mathcal{L}^d \text{ with }  g^i \coloneqq \varphi_i g_1 + (1-\varphi_i) g_2 + \nabla \varphi_i (u_1 - u_2).
\end{equation*}
Similarly as in the proof of Lemma~\ref{weaksubaddBV} (and taking into account $g^i = g_1$ on $(\overbar{B}_i)_{\varepsilon+}$ and $g^i = g_2$ on $(A_2 \setminus \overbar{B}_{i+1})_{\varepsilon +}$) we have
\begin{align*}
 E_{\varepsilon}(u^i,\gamma^i,A_1' \cup A_2) \leq  E_{\varepsilon}(u_1,\gamma_1,A_1) + E_{\varepsilon}(u^i, \gamma^i, A_2  \cap (B_{i+1}\setminus \overbar{B}_{i})) +E_{\varepsilon}(u_2,\gamma_2, A_2).
\end{align*}
We set $D_i \coloneqq A_2  \cap (B_{i+1}\setminus \overbar{B}_{i})$ and then estimate the contributions of the second term on the right-hand side: firstly, by definitions of~$u^i$ and~$g^i$ as well as the fact that $0 \leq \varphi_i \leq 1$ in~$\Omega$, we find 
\begin{align*}
 \int_{D_i} |\nabla u^i - g^i|^2 \dd x & =  \int_{D_i} |\varphi_i \nabla u_1 + (1-\varphi_i) \nabla u_2 - \varphi_i g_1 + (1-\varphi_i) g_2|^2 \dd x \\
  & \leq \int_{D_i}  |\nabla u_1 - g_1|^2 \dd x + \int_{D_i}  |\nabla u_2 - g_2|^2 \dd x.
\end{align*}
Secondly, by the subadditivity of~$f$ with $f(t) \leq c_0 t$ for all $t \in \R$ and with the bound $\norm{\nabla \varphi_i}_{L^{\infty}(\Omega)} \leq M$, we infer 
\begin{align}
 \lefteqn{ \frac{1}{\varepsilon} \int_{D_i} f \bigg( \varepsilon \dashint_{B_{\varepsilon}(x) \cap \Omega} |g^i| \dd y \bigg) \dd x } \label{eqn_gluing_f} \\  
 & \leq \frac{1}{\varepsilon}  \int_{D_i} \bigg[ f \bigg( \varepsilon \dashint_{B_{\varepsilon}(x) \cap \Omega} |g_1| \dd y \bigg) + f \bigg( \varepsilon \dashint_{B_{\varepsilon}(x) \cap \Omega} |g_2| \dd y \bigg) + f \bigg( M \varepsilon \dashint_{B_{\varepsilon}(x) \cap \Omega} |u_1-u_2| \dd y \bigg) \bigg] \dd x \notag \\
 & \leq \frac{1}{\varepsilon}  \int_{D_i} f \bigg( \varepsilon \dashint_{B_{\varepsilon}(x) \cap \Omega} |g_1| \dd y \bigg) \dd x +  \frac{1}{\varepsilon}  \int_{D_i}  f \bigg( \varepsilon \dashint_{B_{\varepsilon}(x) \cap \Omega} |g_2| \dd y \bigg) \dd x  + c_0 M \int_{(D_i)_{\varepsilon+}} |u_1-u_2| \dd x. \notag
\end{align}
Collecting these estimates, we obtain
\begin{align*}
 E_{\varepsilon}(u^i,\gamma^i,A_1' \cup A_2) & \leq  E_{\varepsilon}(u_1,\gamma_1,A_1)  + 2 E_{\varepsilon}(u_1,\gamma_1,D_i) +E_{\varepsilon}(u_2,\gamma_2, A_2) + 2 E_{\varepsilon}(u_2,\gamma_2,D_i) \\
 & \quad + c_0 M  \norm{u_1-u_2}_{L^1((D_i)_{\varepsilon+})}.
\end{align*}
We then sum this inequality for $i \in\{1, \ldots, k\}$, use the inclusions $\dot{\cup}_{i=1}^{k} D_i \subset A_1$, $\dot{\cup}_{i=1}^{k} D_i \subset A_2$ and the fact that each point in~$\Omega$ is contained in the set $(D_i)_{\varepsilon+}$ for at most two indices~$i$ (recalling $\varepsilon \leq \varepsilon_0=\delta /4(k+2)$ and $\dist(B_i, \partial B_{i+1}) \geq {\delta}/(k+2)$). Taking the average with respect to~$i$, we can then choose as in the proof of Lemma~\ref{weaksubaddBV} an index $i_0 \in\{1, \ldots, k\}$ with
\begin{equation*}
 E_{\varepsilon}(u^{i_0},\gamma^{i_0},A_1' \cup A_2) \leq \Big( 1 + \frac{2}{k} \Big)   \big[ E_{\varepsilon}(u_1,\gamma_1,A_1) + E_{\varepsilon}(u_2,\gamma_2, A_2)\big] + \frac{2 c_0M}{k} \norm{u_1-u_2}_{L^1(\Omega)}.
\end{equation*}
Since $k \in \N$ satisfies by the initial choice the inequalities $2/k < \eta$ and $2c_0/k < 1$, this proves the claim~\eqref{fundamentalestimate} with $\varphi \coloneqq \varphi_{i_0}$.

\emph{Step 2: Subadditivity of $\Psi^{\prime}(u,\gamma, \scdot)$, i.e., for all $A_1, A_2 \in \Am$ there holds}
\begin{equation}
\label{eqn_Psiprime_subadditive}
\Psi^{\prime}(u,\gamma,A_1 \cup A_2) \leq \Psi^{\prime}(u,\gamma,A_1) + \Psi^{\prime}(u,\gamma,A_2).
\end{equation}
We may assume that the right-hand side of~\eqref{eqn_Psiprime_subadditive} is finite since it is trivial otherwise. For $A_1^{\prime}, A_2^{\prime} \in \Am$ with $A_1^{\prime} \Subset A_1$ and  $A_2^{\prime} \Subset A_2$ we choose $A_1^{\prime \prime} \in \Am$ with $A_1^{\prime} \Subset A_1^{\prime \prime} \Subset A_1$. The $\Gamma$-convergence compactness theorem in \cite[Theorem 16.9]{dm} allows to select a subsequence $\varepsilon\to 0$ along which there are sequences $\{(u_{1, \varepsilon},\gamma_{1, \varepsilon})\}_\varepsilon, \{(u_{2, \varepsilon},\gamma_{2, \varepsilon})\}_\varepsilon$ in $L^{\infty}(\Omega) \times \Mam$ with $\gamma_{1, \varepsilon} = g_{1, \varepsilon}\mathcal{L}^d$ and $\gamma_{2, \varepsilon}= g_{2, \varepsilon}\mathcal{L}^d$ for functions $g_{1, \varepsilon}, g_{2, \varepsilon} \in L^1(\Omega,\Rd)$ such that $u_{1, \varepsilon} \to u$ and $u_{2, \varepsilon} \to u$ in $L^1(\Omega)$,  $\gamma_{1, \varepsilon} \to \gamma$ and $\gamma_{2, \varepsilon} \to \gamma$ in the flat norm and
\begin{equation*}
\limsup_{\varepsilon \to 0} E_{\varepsilon}  (u_{1, \varepsilon}, \gamma_{1, \varepsilon}, A^{\prime \prime}_1) \leq \Psi^{\prime}(u, \gamma, A_1) \quad \text{and}\quad 
\limsup_{\varepsilon \to 0} E_{\varepsilon}  (u_{2, \varepsilon}, \gamma_{2, \varepsilon}, A^{\prime}_2) \leq \Psi^{\prime}(u, \gamma, A_2).
\end{equation*}
Due to $\norm{u_{1, \varepsilon}}_{L^{\infty}(\Omega)}, 
\norm{u_{2, \varepsilon}}_{L^{\infty}(\Omega)} \leq K$, we even have $u_{1, \varepsilon} \to u$ and $u_{2, \varepsilon} \to u$ in $L^p(\Omega)$ for all $1 \leq p < \infty$. We now fix $\eta >0$. By Step~1, there exist $\varepsilon_0>0$ and $M >0$ such that
\begin{align}
&E_{\varepsilon}( \varphi_{\varepsilon} u_{1, \varepsilon} + (1-\varphi_{\varepsilon}) u_{2, \varepsilon}, [\varphi_{ \varepsilon} g_{1, \varepsilon} + (1-\varphi_{\varepsilon}) g_{2, \varepsilon} + \nabla \varphi_{\varepsilon} (u_{1, \varepsilon} - u_{2, \varepsilon})] \mathcal{L}^d, A^{\prime}_1 \cup A_2') \notag\\ & \leq (1 + \eta)  \left[ E_{\varepsilon}( u_{1, \varepsilon} , \gamma_{1, \varepsilon}, A^{\prime \prime }_1) +  E_{\varepsilon}( u_{2, \varepsilon} , \gamma_{2, \varepsilon}, A_2') \right]  + M \norm{u_{1, \varepsilon}-u_{2, \varepsilon}}_{L^1(\Omega)} \label{appliedfundamentalestimate}
\end{align}
holds  for every $\varepsilon \leq \varepsilon_0$ for a suitable cut-off function $\varphi_{\varepsilon} \in C_c^{\infty}(A^{\prime \prime}_1)$ with $\varphi_{\varepsilon} \equiv 1$ on $A^{\prime}_1$ and $\norm{\nabla \varphi_\varepsilon}_{L^{\infty}(\Omega, \Rd)} \leq M$.
As $\varepsilon \to 0$, we obtain
\begin{align*}
 \varphi_{\varepsilon} u_{1, \varepsilon} &+ (1-\varphi_{\varepsilon}) u_{2, \varepsilon} \to u \text{ in }L^1(\Omega), \\
 [\varphi_{ \varepsilon} g_{1, \varepsilon} + (1-\varphi_{\varepsilon}) g_{2, \varepsilon} &+ \nabla \varphi_{\varepsilon} (u_{1, \varepsilon} - u_{2, \varepsilon})] \mathcal{L}^d \to \gamma\text{ in the flat norm}.
\end{align*}
Sending $\varepsilon \to 0$ in \eqref{appliedfundamentalestimate}, we get by choice of the sequences $\{(u_{1, \varepsilon},\gamma_{1, \varepsilon})\}_\varepsilon, \{(u_{2, \varepsilon},\gamma_{2, \varepsilon})\}_\varepsilon$
\[ E'(u, \gamma, A^{\prime}_1 \cup A^{\prime}_2) \leq (1+ \eta) \left[ \Psi^{\prime}(u, \gamma, A_1)  + \Psi^{\prime}(u, \gamma, A_2)\right]. \]
With $\eta \searrow 0$ and then $A^{\prime}_1 \uparrow A_1$ and $A^{\prime}_2 \uparrow A_2$ we finally obtain~\eqref{eqn_Psiprime_subadditive}.

\emph{Step 3: Conclusion.} Since we have also verified the subadditivity of~$\Psi^{\prime}$ in Step~$2$, the first statement of the lemma follows directly from the measure property criterion of De Giorgi--Letta from \cite[Theorem 5.6]{DGL77}. Finally, the second statement is shown analogously to the second part of \cite[Proposition 5.1]{LVDD07} (by approximating $J_u$ with $J_{u,k} = \{x \in J_u : |[u](x)| \ge 1/k \}$ for $k \in \N$ and observing that $\Psi^{\prime}(u, \gamma, \scdot) \ge h \HA \mrs J_{u,k}$ with $h = \dd \Psi^{\prime}(u, \gamma, \scdot) / \dd \HA \mrs J_{u,k}$, for each $k \in \N$.). 
\end{proof}

We fix a constant $\bar{K}>0$ and for each $x_0 \in \Omega$ define a family of auxiliary functionals $H_{\varepsilon,x_0} \colon L^1(B_2(x_0)) \times \mathcal{M}(B_2(x_0), \R^{d}) \times \mathcal{A}(B_2(x_0)) \to [0, \infty]$ by
\begin{equation*}
H_{\varepsilon,x_0}(u,\gamma,A) \coloneqq 
\begin{cases} 
\frac{1}{\varepsilon} \int_{A} f(\varepsilon \dashint_{B_{\varepsilon}(x) \cap B_2(x_0)} |g| \dd y) \dd x &  \text{if } u \in W^{1,1}(B_2(x_0)), \, \norm{u}_{L^{\infty}(B_2(x_0))} \leq K, \\   & \quad \phantom{if } \gamma = g \mathcal{L}^d,  \, g \in L^1(B_2(x_0),\R^d), \\  &\quad \phantom{if } \text{and } \norm{\nabla u - g}_{L^2(B_2(x_0), \Rd)} \leq \bar{K},  \\[0.2cm] 
\infty &  \text{otherwise}.
\end{cases}
\end{equation*}
(We omit the subscript $x_0$ in case $x_0 = 0$.) We will now estimate the density $h$ of the previous lemma by the $\Gamma$-lower limit $H_{x_0}^{\prime}$ of $\{H_{\varepsilon,x_0}\}_{\varepsilon}$.

\begin{lemma}\label{derivativesingularpartBV2}
For every $(u,\gamma) \in  BV(\Omega) \times \Mam$  with $\gamma = D^su+ g \mathcal{L}^d$ and $\nabla u- g \in L^2(\Omega,\R^d)$ we have
 \begin{equation*}
\liminf_{\varrho \searrow 0} \frac{\Psi^{\prime}(u, \gamma,B_{\varrho}(x_0)) } { \varrho^{d-1}} \geq H_{x_0}^{\prime} (u_0, \gamma_0, B_1(x_0)) \quad \text{for } \HA \text{-a.e. } x_0 \in J_u, 
\end{equation*}
where $u_0$ is defined by
\begin{equation*}
u_0(x) \coloneqq
\begin{cases} u(x_0+) & \text{ if } (x-x_0) \cdot \nu \geq 0\\
u(x_0-) & \text{ if } (x-x_0) \cdot \nu < 0.
\end{cases}
\end{equation*}
with $\nu \coloneqq \nu_{u}(x_0)$ and where $\gamma_0 \coloneqq D^su_0$.
\end{lemma}

\begin{proof}
We follow the strategy of the proof of \cite[Proposition 5.2]{LVDD07}, with a number of modifications due to the additional variable~$\gamma$.
We may without loss of generality assume that $\lim_{\varrho \searrow 0} \frac{\Psi^{\prime}(u, \gamma,B_{\varrho}(x_0)) }{ \varrho^{d-1}}$ exists, see Lemma~\ref{derivativesingularpartBV}, and that  
\begin{equation}
\limsup_{\varrho \searrow 0} \frac{|Du|(B_\varrho(x_0))}{ \varrho^{d-1}}  < \infty \quad \text{and} \quad \lim_{\varrho \searrow 0} \frac{1}{ \varrho^{d-1} } \int_{B_{\varrho}(x_0)} |g - \nabla u| \dd x =0, \label{absolutelyfaelltweg}
\end{equation}
see \cite[(2.40) \& (2.41)]{afp}, as these conditions are satisfied $\HA$-a.e.~on~$J_u$. For easy notation we also assume $x_0=0$.
Let $\delta > 1$ be arbitrary. By the definition~\eqref{def_Psi_prime} of $\Psi^{\prime}$ as the inner regular envelope of $E^{\prime}$ and the $\Gamma$-convergence compactness theorem in \cite[Theorem 16.9]{dm}, there exists a subsequence $\{(u_{\varepsilon},\gamma_{\varepsilon})\}_{\varepsilon} \in  W^{1,1}(\Omega) \times \Mam$ with $\norm{\ue}_{L^{\infty}(\Omega)} \leq K$ and with $\gae = g_{\varepsilon} \mathcal{L}^d$ for a function $g_{\varepsilon} \in L^1(\Omega, \Rd)$ for $\varepsilon >0$ such that $\ue \to u$ in $L^1(\Omega)$, $\gae \to \gamma$ in the flat norm and 
\begin{align*}
\lim_{\varepsilon \to 0}E_{\varepsilon}(\ue, \gae, B_{\varrho}(0)) \leq \Psi'(u, \gamma, B_{\delta\varrho}(0)) 
\end{align*}
for every $\varrho >0$ with $B_{\varrho}(x_0) \subset \Omega$. Now let $\{\varrho_k\}_k$ be an arbitrary decreasing sequence in $(0,1)$ converging to~$0$ such that $B_{2\varrho_1}(0) \subset \Omega$. From our subsequence $\eps$ we then select $\{\varepsilon_k\}_{k}$ with ${\varepsilon_k} \leq \varrho_k / k$ and $\{(u_{\varepsilon_k},\gamma_{\varepsilon_k})\}_{k} \in  W^{1,1}(\Omega) \times \Mam$, still with $\norm{u_{\varepsilon_k}}_{L^{\infty}(\Omega)} \leq K$ and $\gamma_{\varepsilon_k} = g_{\varepsilon_k} \mathcal{L}^d$ for a function $g_{\varepsilon_k} \in L^1(\Omega, \Rd)$ for $k \in \N$, such that 
\begin{align*}
E_{\varepsilon_k}(u_{\varepsilon_k}, \gamma_{\varepsilon_k}, B_{\ell \varrho_k}(0)) & \leq \Psi'(u, \gamma, B_{\ell\delta\varrho_k}(0)) + \frac{\varrho_k^{d-1}}{k},\quad  \ell \in \{1,2\}, \\
\norm{u_{{\varepsilon_k}}-u}_{L^1(\Omega)} \leq \frac{\varrho_k^d}{k}  \quad & \text{and}  \quad \int_{B_2(0)} |u_{{\varepsilon_k}}(\varrho_k x) - u(\varrho_k x)| \dd x \leq \frac{1}{k}, \\
\norm{\gamma_{{\varepsilon_k}}-\gamma}_{\text{flat}} \leq \frac{\varrho_k^d}{k} \text{ on } \Omega \quad & \text{and} \quad 
\norm{\gamma_{\varepsilon_k}^{\varrho_k} - \gamma^{\varrho_k}}_{\text{flat}} \leq \frac{1}{k} \text{ on }
B_2(0). 
\end{align*}
Here, the superskript~$\varrho_k$ denotes the rescaling of the measures~$\gamma_{\varepsilon_k}$ and~$\gamma$ by the factor~$\varrho_k$, i.e., with $\gamma^{\varrho_k}(A) \coloneqq \varrho_k^{1-d} \gamma(\varrho_k A)$ for every open set $A \subset \varrho_k^{-1} \Omega$. From the first inequality we obtain
\begin{align}
\limsup_{k \to \infty} \frac{E_{\varepsilon_k}(u_{\varepsilon_k}, \gamma_{\varepsilon_k}, B_{\ell \varrho_k}(0)) }  { \varrho_{k}^{d-1}} \leq \lim_{k \to \infty} \frac{\Psi'(u, \gamma, B_{\ell\delta\varrho_k}(0))}  { \varrho_{k}^{d-1}} 
 = (\ell\delta)^{d-1}\lim_{\varrho \searrow 0} \frac{\Psi'(u, \gamma, B_{\varrho}(0))}  { \varrho^{d-1}} \label{blowupfunction}
\end{align} 
for $\ell \in \{1,2\}$. From the second inequality in combination with the triangle inequality and the definition of $u^{\pm}$ (cf. \cite[Remark 3.72]{afp}) it follows that
\begin{align*}
\int_{B_2(0)} |u_{\varepsilon_k}(\varrho_k x) - u_0(\varrho_k x)| \dd x  \leq \frac{1}{k} + \int_{B_2(0)} |u(\varrho_k x) - u_0(x)| \dd x \to 0  \quad \text{ as } k \to \infty,
\end{align*}
meaning that $u_{\varepsilon_k}(\varrho_k \, \cdot) \to u_0$ in $L^1(B_2(0))$. Correspondingly, the third inequality gives
\begin{equation*}
 \norm{\gamma_{\varepsilon_k}^{\varrho_k} - \gamma_0}_{\text{flat}} \leq \frac{1}{k} + \norm{\gamma^{\varrho_k} - \gamma_0}_{\text{flat}} \text{ on }
B_2(0).
\end{equation*}

Hence, if we show that the second term on the right-hand side vanishes in the limit $k \to \infty$, then $\{\gamma_{\varepsilon_k}^{\varrho_k}\}_k$ converges to $\gamma_0$ in the flat norm on $B_2(0)$.
Note that $|Du|^{\varrho_k}(B_2(0)) = |Du|(B_{2\varrho_k}(0)) / \varrho_k^{d-1}$ is bounded uniformly in~$k \in \N$ by our assumption~\eqref{absolutelyfaelltweg} on $x_0$. Hence, the sequence $\{(Du)^{\varrho_k}\}_{k}$ of measures in $\mathcal{M}(B_2(0))$ converges weakly$^{\ast}$ to $Du_0=D^s u_0 =\gamma_0$ and therefore, by Lemma~\ref{Lemma_weak_negativ_flat}~\ref{Lemma_weak_negativ_flat_2} also in the flat norm on $B_2(0)$. Thus, it remains to show that $ \gamma^{\varrho_k}-(Du)^{\varrho_k} \to 0$ in the flat norm.
Employing $\gamma - Du = (g - \nabla u) \mathcal{L}^d$ and once again our assumption~\eqref{absolutelyfaelltweg} for $x_0=0$, we have indeed  
\begin{align*}
 \bigg \vert\int_{B_2(0)} \varphi(x) \dd ({\gamma}^{\varrho_k} - (Du)^{\varrho_k}) \bigg \vert 
 & =  \bigg \vert\int_{B_{2\varrho_k}(0)} \varrho_k^{1-d} \varphi(\varrho_k^{-1} x) (g(x) - \nabla u (x)) \dd x \bigg \vert \\
 & \leq  2 \varrho_k^{1-d} \norm{\nabla \varphi}_{L^{\infty}(\Omega, \Rd)} \int_{B_{2\varrho_k}(0)} |g - \nabla u | \dd x  \to 0
\end{align*}
as $k \to \infty$ for every $\varphi \in W^{1,\infty}_0(B_2(0))$ with $\norm{\varphi}_{W^{1,\infty}_0(B_2(0))} \leq 1$. To conclude the proof of the lemma, we still need to estimate the left-hand side of~\eqref{blowupfunction} from below via the $H$-functional on the rescaled ball. To this end, we rewrite via the substitution $x = \varrho_k z$
\begin{equation*}
\frac{E_{\varepsilon_k}(u_{\varepsilon_k}, \gamma_{\varepsilon_k}, B_{\ell \varrho_k}(0)) }  { \varrho_{k}^{d-1}} = \frac{1}{\varrho_k} \int_{B_{\ell}(0)} \vert \nabla w_{\varepsilon_k} - \bar{g}_{\varepsilon_k}|^2 \dd z  + \frac{\varrho_k}{\varepsilon_k} \int_{B_{\ell}(0)} f \bigg( \frac{\varepsilon_k}{\varrho_k} \dashint_{B_{\varepsilon_k/\varrho_k}(z)} \vert \bar{g}_{\varepsilon_k}\vert  \dd y \bigg) \dd z
\end{equation*}
with $w_{\varepsilon_k} \coloneqq u_{\varepsilon_k}(\varrho_k \, \cdot)$ and $\bar{g}_{\varepsilon_k} \coloneqq \varrho_k g_{\varepsilon_k}(\varrho_k \, \cdot)$. For the restrictions to $B_2(0)$ with $w_{\varepsilon_k} \in W^{1,1}(B_2(0))$ and $\gamma_{\varepsilon_k}^{\varrho_k} = \bar{g}_{\varepsilon_k} \mathcal{L}^d \mrs B_2(0)$ with $\bar{g}_{\varepsilon_k} \in L^1(B_2(0), \Rd)$ we find for $\ell = 2$, respectively, $\ell = 1$, 
\begin{align*}
\frac{E_{\varepsilon_k}(u_{\varepsilon_k}, \gamma_{\varepsilon_k}, B_{2\varrho_k}(0)) }  { \varrho_{k}^{d-1}} 
&\ge \frac{1}{\varrho_k} \int_{B_{2}(0)} \vert \nabla w_{\varepsilon_k} - \bar{g}_{\varepsilon_k}|^2 \dd z \quad\mbox{and}\\ 
\frac{E_{\varepsilon_k}(u_{\varepsilon_k}, \gamma_{\varepsilon_k}, B_{\varrho_k}(0)) }  { \varrho_{k}^{d-1}} 
&\ge  \frac{\varrho_k}{\varepsilon_k} \int_{B_{1}(0)} f \bigg( \frac{\varepsilon_k}{\varrho_k} \dashint_{B_{\varepsilon_k /\varrho_k}(z)} \vert \bar{g}_{\varepsilon_k} \vert  \dd y \bigg) \dd z. 
\end{align*}
Since we may assume that the left hand side of~\eqref{blowupfunction} is finite, the first inequality yields $\| \nabla w_{\varepsilon_k} - \bar{g}_{\varepsilon_k} \|_{L^2(B_2(0),\R^d)} \to 0$. Since $0 \leq \varepsilon_k /\varrho_k \leq 1/k \to 0$, $w_{\varepsilon_k} \to u_0$ in $L^1(B_2(0))$ and $\gamma_{\varepsilon_k}^{\varrho_k} \to \gamma_0$ in the flat norm on $B_2(0)$, \eqref{blowupfunction} for $\ell = 1$ and the second inequality now imply
\begin{equation*}
\delta^{d-1}\lim_{\varrho \searrow 0} \frac{\Psi'(u, \gamma, B_{\varrho}(0))}  { \varrho^{d-1}}
 \ge \liminf_{k \to \infty} H_{\varepsilon_k/\varrho_k}(w_{\varepsilon_k},\gamma_{\varepsilon_k}^{\varrho_k}, B_1(0)) \geq H^{\prime}(u_0, \gamma_0, B_1(0)).
\end{equation*}
Thus, the desired estimate follows from the arbitrariness of $\delta > 1$. \qedhere
\end{proof}

As an immediate consequence of Lemma \ref{derivativesingularpartBV}
and Lemma \ref{derivativesingularpartBV2} we obtain:

\begin{corollary}\label{surfacebvCor}
Let $A \subset \Omega$ be open.  For every $(u,\gamma) \in  BV(\Omega) \times \Mam$  with $\gamma = D^su+ g \mathcal{L}^1$ and $\nabla u- g \in L^2(\Omega,\R^d)$ we have 
\begin{equation*}
\Psi^{\prime}(u, \gamma, A) \geq \int_{J_u \cap A} h' \dd  \HA \quad \text{with }
h'(x_0) \coloneqq \frac{H^{\prime}_{x_0}(u_0, \gamma_0, B_1(x_0))}{\omega_{d-1}} \text{ for } \HA \text{-a.e. } x_0 \in J_u,
\end{equation*}
where $u_0$ and $\gamma_0$ are defined as in Lemma~\ref{derivativesingularpartBV2}.
\end{corollary}

We now aim to find an appropriate estimate of $H^{\prime}(u_0, \gamma_0, B_1(x_0))$ from below.
Let $\nu \in \Sm$. For every $y \in \R^d$ we denote by $y_{\nu}$ and $y_{\nu^{\perp}}$ the projections onto the subspaces $V = \lbrace t \nu \colon t \in \R \rbrace$ and $V^{\perp} = \lbrace x \in \Rd : x \cdot \nu = 0\rbrace$, respectively. For $\varrho >0$ and $x \in \Rd$ we define cylinders
\begin{equation}
C_{\varrho}^{\nu}(0)\coloneqq \lbrace y \in \Rd : |y_{\nu}| <\varrho , |y_{\nu^{\perp}}| < \varrho   \rbrace, \quad C_{\varrho}^{\nu}(x)\coloneqq  x + C_{\varrho}^{\nu}(0). \label{zylinder}
\end{equation}

In the next lemma, which is analogous to \cite[Lemma 5.4]{LVDD07}, we consider the situation of a pure jump function $u_0$ and correspondingly $\gamma_0 = D^s u_0$. We here show the existence of a suitable sequence $\{(\ue, \gae)\}_{\varepsilon}$ converging to $(u_0,\gamma_0)$ and approximating the energy, i.e., $\lim_{\varepsilon \to 0} H_{\varepsilon}(\ue, \ggae, A) = H^{\prime}(u_0, \gamma_0, A) $, with $
\ue=u_0$ and $\gae= \gamma_0$ outside of an infinitesimal neighborhood of the jump set of~$u_0$.

\begin{lemma}[approximation of pure jump functions]\label{modificationsurface}
Let $a \neq b \in \R$, $\nu \in \Sm$,
\begin{align*}
u_0(x)=
\begin{cases} a & \text{ if } x \cdot \nu \geq 0,\\
 b & \text{ if } x \cdot \nu < 0,
\end{cases}
\end{align*}
and $\gamma_0=D^su_0.$ 
For every open subset $A$ of $C_1^{\nu}(0)$ there exists a sequence $\{(u_{\varepsilon},\gamma_\varepsilon)\}_{\varepsilon}$ in $W^{1,1}(B_2(0)) \times \mathcal{M}(B_2(0), \R^d)$ with $\norm{\ue}_{L^{\infty}(B_2(0))} \leq K$, with $\gae=\ggae \mathcal{L}^d$ for a function $\ggae \in L^1(B_2(0),  \R^d)$ and with $\norm{\nabla \ue - g_{\varepsilon}}_{L^2(B_2(0), \Rd)} \leq \bar{K} $ for every $\varepsilon >0$ such that $\ue \to u_0$ in $L^1(B_2(0))$ and $\gae \to \gamma_0$ in the flat norm on $B_2(0)$ and such that, for a subsequence, there holds
\begin{gather*}
\lim_{\varepsilon \to 0}H_{\varepsilon}(\ue, \gae, A) = H^{\prime}(u_0, \gamma_0, A),\\
\ue(x)= a \text{ if } x \cdot \nu \geq \alpha_{\varepsilon} \text{ and } \ue(x)= b \text{ if } x \cdot \nu \leq - \beta_{\varepsilon},\\
\spt(\ggae) \subset \{ x \in \R^d \colon x \cdot \nu \in (-\beta_{\varepsilon},\alpha_{\varepsilon})\},
\end{gather*}
where $\{\alpha_{\varepsilon}\}_{\varepsilon}$, $\{\beta_{\varepsilon}\}_{\varepsilon}$ are suitable nonnegative infinitesimal sequences.
\end{lemma}

\begin{proof}
It is not restrictive to assume $\nu=e_1$, $|a| \leq K$ and $|b| \leq K$. Let $\{(u_{\varepsilon},\gamma_\varepsilon)\}_{\varepsilon}$ be a sequence in $W^{1,1}(B_2(0)) \times \mathcal{M}(B_2(0), \R^d)$ with $\norm{\ue}_{L^{\infty}(B_2(0))} \leq K$, $\gae=\ggae \mathcal{L}^d$ for a function $\ggae \in L^1(B_2(0),  \R^d)$ and $\norm{\nabla \ue - g_{\varepsilon}}_{L^2(B_2(0), \Rd)} \leq \bar{K} $ for every $\varepsilon >0$ such that $\ue \to u_0$ in $L^1(B_2(0))$, $\gae \to \gamma_0$ in the flat norm on $B_2(0)$ and
\begin{equation*}
H_{\varepsilon}(\ue, \gae,A) \to H^{\prime}(u_0, \gamma_0,A) < \infty
\end{equation*}
(with all elements being finite). We may assume $\varepsilon < \dist(C_1^{\nu}(0), \partial B_2(0))$ in what follows. We shall now modify this sequence in such a way that $\ue(x)= b$ holds for $x_1 \leq - \beta_\varepsilon$ and $\spt(g_\varepsilon) \subset (-\beta_\varepsilon,\infty) \times \R^{d-1}$. By analogous arguments applied on the other half-space the full statement of the lemma is then established.  

For $\sigma > 0$ we consider the continuous function $\phi \in C(\R^d)$ given by
\begin{align*}
\phi(x)= \phi_{\varepsilon, \sigma}(x)\coloneqq 
\begin{cases}0 & \text{ if }x_1 \leq -2 \varepsilon - \sigma,\\
\text{affine} & \text{ if } -2 \varepsilon - \sigma \leq x_1 \leq -2 \varepsilon,\\
1 & \text{ if } x_1 \geq -2 \varepsilon.
\end{cases}
\end{align*}
In particular, we have $|\nabla \phi| \leq 1 / \sigma$ on $\Omega$ and $\norm{\phi}_{L^{\infty}(\Rd)} = 1$. We define $(\tilde{u}_{\varepsilon},\tilde{\gamma}_\varepsilon)$ in $W^{1,1}(B_2(0)) \times \mathcal{M}(B_2(0), \R^d)$ with $\tgae= \tggae \mathcal{L}^d$ via
\begin{equation*}
\tilde{u}_\varepsilon \coloneqq \phi u_\varepsilon +(1-\phi) b \quad \text{and} \quad  \tilde{g}_\varepsilon \coloneqq \phi g_\varepsilon +(1-\phi) \cdot 0 + \nabla \phi (u_\varepsilon - b).
\end{equation*} 
By construction, we have $\norm{\tilde{u}_\varepsilon}_{L^{\infty}(B_2(0))}\leq K$ and 
\begin{equation*}
\int_{B_2(0)} |\nabla \tilde{u}_{\varepsilon}- \tilde{g}_{\varepsilon}|^2  \dd x  = \int_{B_2(0)} |\phi (\nabla u_{\varepsilon} - g_{\varepsilon}) |^2  \dd x \leq \bar{K}.
\end{equation*}
Exactly as in Step~1 in the proof of Lemma~\ref{derivativesingularpartBV}, see~\eqref{eqn_gluing_f}, we find that the non-local term for this modification is estimated by
\begin{align*}
\frac{1}{\varepsilon}\int_{A} f \bigg( \varepsilon \dashint_{B_{\varepsilon}(x)} |\tilde{g}_\varepsilon| \dd y \bigg) \dd x & \leq \frac{1}{\varepsilon} \int_{A } f  \bigg(\varepsilon \dashint_{B_{\varepsilon}(x)} |g_\varepsilon| \dd y \bigg) \dd x \\ & \quad + \frac{1}{\varepsilon}\int_{A} f \bigg( \varepsilon \dashint_{B_{\varepsilon}} |0| \dd y \bigg) \dd x + \frac{c_0}{\sigma} \int_{\{x \in A \colon \nabla \phi \neq 0\}_{\varepsilon+}} |u_{\varepsilon}-b| \dd x \\
& \leq \frac{1}{\varepsilon} \int_{A } f  \bigg(\varepsilon \dashint_{B_{\varepsilon}(x)} |g_\varepsilon| \dd y \bigg) \dd x + \frac{c_0}{\sigma} \int_{B_2^-(0)} |u_{\varepsilon}-b| \dd x,
\end{align*}
where $B_2^-(0) \coloneqq \{ x \in B_2(0) \colon x_1 < 0\}$. Hence, we have 
\begin{equation*}
H_{\varepsilon}(\tilde{u}_{\varepsilon},\tgae,A)  \leq H_{\varepsilon}(\ue,\gae,A) + \frac{c_0}{\sigma}\int_{B_2^-(0)} { |\ue - b|} \dd x. 
\end{equation*}
Since $\{u_{\varepsilon}\}_\varepsilon$ converges to~$u_0$ in $L^1(B_2(0))$ and $u_0 = b$ on~$B_2^-(0)$, the second integral on the right-hand side of this inequality vanishes in the limit~$\varepsilon \to 0$. 
Recalling that $\gamma_\varepsilon \to \gamma_0$ in the flat norm on~$B_2(0)$, we can therefore find to a given positive infinitesimal sequence $\{\sigma_{k}\}_k$ a positive infinitesimal sequence $\{\varepsilon_{k}\}_k$ such that
\begin{align}
\frac{1}{\sigma_{k}} \int_{B_2^-(0)} { |u_{\varepsilon_{k}} - b|} \dd x \to 0 \quad \text{and} \quad \frac{1}{\sigma_{k}} \norm{\gamma_{\varepsilon_k} - \gamma_0}_{\textnormal{flat}} \to 0 \quad \text{ as } k \to \infty. \label{sigmahchoice}
\end{align}
This shows 
\begin{align*}
\liminf_{k \to \infty}H_{\varepsilon_k}(\tilde{u}_{\varepsilon_k},\tilde{\gamma}_{\varepsilon_k},A) \leq \liminf_{k \to \infty}H_{\varepsilon_k}(u_{\varepsilon_k},{\gamma}_{\varepsilon_k},A) = H^{\prime}(u_0,\gamma_0,A),
\end{align*}
where $\tilde{u}_{\varepsilon_k},\tilde{\gamma}_{\varepsilon_k}$ are defined as $\tilde{u}_{\varepsilon},\tilde{\gamma}_{\varepsilon}$ with $(\varepsilon, \sigma)$ replaced by $(\varepsilon_k, \sigma_k)$. In order to conclude with $H_{\varepsilon_k}(\tilde{u}_{\varepsilon_k},\tilde{\gamma}_{\varepsilon_k},A) \to H^{\prime}(u_0,\gamma_0,A)$ as claimed, it only to remains to verify $\tilde{u}_{\varepsilon_k} \to u_0$ in $L^1(B_2(0))$ and $\tilde{\gamma}_{\varepsilon_k} \to \gamma_0$ in the flat norm on $B_2(0)$, as $k \to \infty$. The first convergence is obvious from  $\ue \to u_0$ in $L^1(B_2(0))$  and $u_0 = b$ on the set $\{x \in B_2(0) \colon \phi_{\varepsilon_k,\sigma_k} \neq 1\}$. For the second one, we consider $\varphi \in W^{1, \infty}_0(B_2(0))$ with $\norm{\varphi}_{W^{1, \infty}_0(B_2(0))} \leq 1$ and we first rewrite
\begin{multline*}
 \int_{B_2(0)} \varphi \dd (\tilde{\gamma}_{\varepsilon_k} - \gamma_0) \\ = \int_{B_2(0)} \varphi \phi_{\varepsilon_k,\sigma_k} \dd (\gamma_{\varepsilon_k}-\gamma_0) + \int_{B_2(0)} \varphi \nabla \phi_{\varepsilon_k,\sigma_k} (u_{\varepsilon_k} -b )   \dd x - \int_{B_2(0)} \varphi (1-\phi_{\varepsilon_k,\sigma_k}) \dd \gamma_0.
\end{multline*}
Since $\phi_{\varepsilon_k,\sigma_k} \equiv 1$ on $B_2(0) \setminus B_2^-(0)$, the relevant domain of integration of the second integral is only~$B_2^-(0)$, while the third integral does not contribute at all. Moreover, due to $|\nabla \phi_{\varepsilon_k,\sigma_k}| \leq 1/\sigma_k$ on~$B_2(0)$, the product $\varphi \phi_{\varepsilon_k,\sigma_k}$ is in $W^{1, \infty}_0(B_2(0))$ with $\norm{\psi \phi_{\varepsilon_k,\sigma_k}}_{W^{1, \infty}_0(B_2(0))}\leq C/\sigma_k$ for a constant~$C$ which is independent of~$k \in \N$. Thus,~\eqref{sigmahchoice} implies
\begin{equation*}
\bigg \vert \int_{B_2(0)} \varphi \dd (\tilde{\gamma}_{\varepsilon_k} - \gamma_0) \bigg \vert
 \leq   \frac{C}{\sigma_k} \norm{\gamma_{\varepsilon_k} - \gamma_0}_{\textnormal{flat}} + \frac{\norm{\varphi}_{L^\infty(B_2(0))}}{\sigma_k} \int_{B_2^-(0)}  |u_{\varepsilon_k} -b | \dd x  \to 0 \quad \text{as } k \to \infty.
\end{equation*}
Thus, we have verified $\tilde{\gamma}_{\varepsilon_k} \to \gamma_0$ in the flat norm on~$B_2(0)$ as well.  Finally, we notice that by construction, if we set $\beta_{\varepsilon_k} \coloneqq 2 \varepsilon_k + \sigma_k$, the sequence $\{\beta_{\varepsilon_k}\}_k$ is infinitesimal, and we have
\begin{equation*}
 \tilde{u}_{\varepsilon_k} = b \text{ if } x_1 \leq - \beta_{\varepsilon_k} \quad \text{and} \quad \spt(\tilde{g}_{\varepsilon_k}) \subset \{ x \in \R^d \colon x_1 > -\beta_{\varepsilon_k}\}
\end{equation*}
for each $k \in \N$. This finishes the (one-sided) desired modification of the original sequences.
\end{proof}

\begin{lemma}\label{equalityonset}
For $\HA$ a.e. $x_0 \in J_u$ we have
\begin{equation*}
H_{x_0}^{\prime}(u_0, \gamma_0, B_1(x_0)) = H_{x_0}^{\prime}(u_0, \gamma_0, C_{1}^{\nu}(x_0)),
\end{equation*}
where $\nu=\nu_{u_0}(x_0)$ and $u_0$ and $\gamma_0$ are defined as in Lemma \ref{derivativesingularpartBV2}.
\end{lemma}

\begin{proof}
We follow (up to a change in the scaling factor which disappears in the limit $\delta \nearrow 1$ and which allows to preserve the $L^2$-bounds) the proof of \cite[Proposition 5.5]{LVDD07}.
Note that we only have to show  the ``$\geq$''-inequality since the other inequality is trivially satisfied due to the facts that $B_1(x_0) \subset C_1^\nu(x_0)$ and that $H^{\prime}_{x_0}(u_0, \gamma_0,\scdot)$ is an increasing set function.  We again assume $x_0=0$ and $\nu=e_1$. We consider the sequence $\{(u_{\varepsilon},\gamma_\varepsilon)\}_{\varepsilon}$  with $\gae=\ggae \mathcal{L}^d$ for a function $\ggae \in L^1(B_2(0),  \R^d)$ for every $\varepsilon > 0$ from the previous Lemma~\ref{modificationsurface} for the choice $A=B_1(0)$. Then, we have $\ue(x)=a$ if $x_1 \geq \alpha_{\varepsilon}$, $\ue(x)=b$ if $x_1 \leq - \beta_{\varepsilon}$, and $\spt(g_\varepsilon) \subset R^{e_1}_{-\beta_{\varepsilon}, \alpha_{\varepsilon}}$ for 
\begin{equation}
R^{e_1}_{-\beta_{\varepsilon}, \alpha_{\varepsilon}}\coloneqq  \{ x \in \Rd \colon x_1 \in (- \beta_{\varepsilon}, \alpha_{\varepsilon})\},\label{definitionlayer}
\end{equation}
where $\{\alpha_{\varepsilon}\}_{\varepsilon}$ and $\{\beta_{\varepsilon}\}_{\varepsilon}$ are suitable nonnegative infinitesimal sequences.
Fix $0 < \delta < 1$. For sufficiently small $\varepsilon >0$ we have $C_{\delta}^{e_1}(0) \cap R^{e_1}_{-\beta_{\varepsilon} - \varepsilon, \alpha_{\varepsilon}  + \varepsilon } \Subset B_1(0)$, and since precisely the subsets of $R^{e_1}_{-\beta_{\varepsilon} - \varepsilon, \alpha_{\varepsilon}  + \varepsilon }$ are the relevant domains of integration for the localized $H_\varepsilon$-functional for our sequence $\{(u_{\varepsilon},\gamma_\varepsilon)\}_{\varepsilon}$, we then find
\begin{equation}
H_{\varepsilon}(\ue, \gae, B_1(0)) \geq H_{\varepsilon}(\ue, \gae, C_{\delta}^{e_1}(0) \cap B_1(0)) = H_{ \varepsilon}(\ue, \gae, C_{\delta}^{e_1}(0)). \label{estimateballcylinder}
\end{equation}
If we set $\ve(x)\coloneqq  \delta^d \ue(\delta x)$ and $\tgae\coloneqq    \tggae \mathcal{L}^d$ with $\tggae(x)=\delta^{d+1}\ggae(\delta x)$, we have $\ve \to u_0^{\delta}\coloneqq  \delta^d u_0$ in $L^1(B_2(0))$ and $\tgae \to \gamma^{\delta}_0\coloneqq  \delta^{d} \gamma_0$  in the flat norm in~$B_2(0)$, and
\begin{equation*}
\norm{\nabla \ve- \tilde{g}_{\varepsilon}}_{L^2(B_2(0), \Rd)}^2 \leq \delta^{d+2} \norm{\nabla \ue- {g}_{\varepsilon}}_{L^2(B_2(0), \Rd)}^2 \leq \bar{K}.
\end{equation*}
Moreover, by a change of variables and the fact that $f$ in non-decreasing, we obtain
\begin{align*}
H_{\varepsilon}(\ue, \gae, C_{\delta}^{e_1}(0)) 
&=\delta^d \frac{1}{\varepsilon} \int_{C_{1}^{e_1}(0)} f \bigg( \frac{\varepsilon}{\delta} \dashint_{B_{{\varepsilon/\delta}}(z)} \vert \delta g_{\varepsilon}(\delta y) \vert  \dd y \bigg) \dd z \\
& \geq \delta^{d-1} \frac{\delta}{\varepsilon} \int_{C_{1}^{e_1}(0)} f \bigg( \frac{\varepsilon}{\delta} \dashint_{B_{{\varepsilon/\delta}}(z)} \vert \tilde{g}(y) \vert  \dd y \bigg) \dd z
=\delta^{d-1} H_{\varepsilon/\delta}(\ve, \tgae, C_{1}^{e_1}(0)).
\end{align*}
Hence, by passing on the left-hand side of \eqref{estimateballcylinder} to the limit in the subsequence provided by Lemma~\ref{modificationsurface}, we get
\begin{equation*}
H^{\prime}(u_0, \gamma_0, B_1(0))\geq \delta^{d-1} \liminf_{\varepsilon \to 0} H_{\varepsilon/\delta}(\ve, \tgae, C_{1}^{e_1}(0)) \geq \delta^{d-1}H^{\prime}(u_0^{\delta}, \gamma_0^{\delta},  C_{1}^{e_1}(0)). 
\end{equation*}
Taking $\delta \nearrow 1$ and using the lower semicontinuity of $H^{\prime}$ concludes the proof.
\end{proof}

Hereafter, we denote by $0_{d-1}$ the zero vector in $\R^{d-1}$. 
Next, we will find the asserted estimate from below of the surface term by formulating a minimization problem on a suitable class of Radon measures satisfying the following disintegration property involving the $(d-1)$ dimensional Lebesgue measure~$\mathcal{L}^{d-1}$. 
 
\begin{definition}
Let $a,b \in \R$ with $a \neq b$. A measure $\mu \in \mathcal{M}( \R \times B^{d-1}_1(0_{d-1}))$ belongs to $\Mned$ if it satisfies the following disintegration property: 
There exists an $\mathcal{L}^{d-1}$-almost everywhere uniquely determined  family $\{\mu_{x'}\}_{x' \in B^{d-1}_1(0_{d-1})}$ of probability measures in $\mathcal{M}(\R)$ such that $x' \mapsto \mu_{x'}$ is Borel measurable and such that $\mu = |a-b| \mathcal{L}^{d-1} \otimes \mu_{x'}$, i.e., for every Borel measurable set $A \times B \subset \R \times B^{d-1}_1(0_{d-1})$ there holds 
 \begin{equation*}
\mu(A \times B)= |a-b| \int_{B} \int_{A} \dd \mu_{x'}(x_1)  \dd x'. 
\end{equation*}
\end{definition}



In what follows, we denote  the convolution of a function $\chi$ and a measure $\mu$ defined on $\R^d$ (possibly after extension) by $(\mu \ast \chi)(x) = \int_{\R^d} \chi(x - y) \dd \mu(y)$.

\begin{lemma}\label{surfacebv2withoutslicing}
Let $u_0$ and $\gamma_0$ be as in Lemma \ref{modificationsurface} with $a \neq b \in \R$ and $\nu= e_1$. Then,
\begin{align*}
H^{\prime}_{x_0}(u_0, \gamma_0, B_1(x_0)) \geq \omega_{d-1} \inf_{\mu \in \Mned} G(\mu),
\end{align*}
where
\begin{equation}
\label{eqn_def_G}
G(\mu) \coloneqq \int_{\R} f \bigg( \frac{1}{\omega_d}(\mu \ast \chi_{B_1(0)})(t, 0_{d-1}) \bigg) \dd t.
\end{equation} 
\end{lemma}

\begin{proof}
The first part of this proof is similar to the one of \cite[Proposition 5.6]{LVDD07}. Since there are some modification necessary because of the second variable $\gamma$, we state the main steps. It is not restrictive to assume $a>0$ with $|a| \leq K$, $b=0$ and $x_0=0$. For simplicity, we denote by~$C$ the set $C^{e_1}_1(0)$ from~ \eqref{zylinder} in what follows. We consider the sequence $\{(u_{\varepsilon},\gamma_\varepsilon)\}_{\varepsilon}$  with $\gae=\ggae \mathcal{L}^d$ for a function $\ggae \in L^1(B_2(0),  \R^d)$ for every $\varepsilon > 0$ from Lemma~\ref{modificationsurface} for the choice $A=C$. Note that $\ue(x)=a$ if $x_1 \geq \alpha_{\varepsilon}$, $\ue(x)=0$ if $x_1 \leq - \beta_{\varepsilon}$, and $\spt(g_\varepsilon) \subset R^{e_1}_{-\beta_{\varepsilon}, \alpha_{\varepsilon}}$ with $R^{e_1}_{-b_{\varepsilon}, a_{\varepsilon}}$ defined in~\eqref{definitionlayer}, for suitable nonnegative infinitesimal sequences $\{\alpha_{\varepsilon}\}_{\varepsilon}$ and $\{\beta_{\varepsilon}\}_{\varepsilon}$. We may assume that $\varepsilon<1$ is sufficiently small such that $\alpha_\varepsilon, \beta_\varepsilon < 1$. Recalling the notation $x'=(x_2, \dots, x_{d})$, we set
\begin{equation*}
\ve(x_1,x')\coloneqq \min \bigg\{ \int_{- 1}^{x_1} \max \Big\{ \frac{\partial}{\partial s} \ue(s,x') ,0 \Big\}  \dd s, a \bigg\}
\end{equation*}
and  
\begin{equation*}
\tgae= \tggae \mathcal{L}^d  \quad \text{with } \tggae \coloneqq  \begin{cases}
\ggae  & \text{if }\frac{\partial  v _{\varepsilon}}{\partial x_1} =\frac{\partial u _{\varepsilon}}{\partial x_1},\\
0 & \text{otherwise}.
\end{cases}
\end{equation*} 
Since $\ue \in W^{1,1}(B_2(0))$ with $\norm{\ue}_{L^{\infty}(B_2(0))}\leq K$ and $g_\varepsilon \in L^1(B_2(0))$, we clearly have $\ve \in W^{1,1}(B_2(0))$ with $\norm{\ve}_{L^{\infty}(B_2(0))} \leq K$ and $\tggae \in L^1(B_2(0), \R^d)$ for all $\varepsilon >0$. Moreover, we have $\ve(x) = u_0(x)$ if $x \notin R^{e_1}_{-b_{\varepsilon}, a_{\varepsilon}}$ and $\ve$ is non-decreasing in direction of~$e_1$. Furthermore, if we set  $\partial_{1} v _{\varepsilon}\coloneqq  \partial v_{\varepsilon} /{\partial x_1}$ and $\hat{g}_{1,\varepsilon}\coloneqq  \min\{ | \tilde{g}_{\varepsilon} \cdot e_1|, \partial_{1} v _{\varepsilon} \} \geq 0$, we have 
\begin{align*}
  \int_{B_2(0)} | \partial_{1} v _{\varepsilon} - \hat{g}_{1,\varepsilon} |^2  \dd x & \leq \int_{B_2(0)}  \Big \vert \frac{\partial \ve}{\partial x_1} - \tggae \cdot e_1 \Big \vert^2 \dd x \\ & \leq  \int_{B_2(0)} \Big \vert \frac{\partial \ue}{\partial x_1} - \ggae \cdot e_1 \Big \vert^2 \dd x \leq  \int_{B_2(0)} \vert \nabla \ue  - \ggae \vert^2 \dd x \leq \bar{K} . 
\end{align*}
This implies, by the Cauchy--Schwarz inequality and $\mathcal{L}^d(R^{e_1}_{-b_{\varepsilon}, a_{\varepsilon}}) \to 0$, 
\begin{align*}
\int_{B_2(0)} |\partial_{1} v _{\varepsilon} - \hat{g}_{1,\varepsilon}| \dd x = \int_{R^{e_1}_{-b_{\varepsilon}, a_{\varepsilon}}} |\partial_{1} v _{\varepsilon} - \hat{g}_{1,\varepsilon}| \dd x \leq  \bar{K} \mathcal{L}^d(R^{e_1}_{-b_{\varepsilon}, a_{\varepsilon}})^{\frac{1}{2}} \to 0 \quad \text{as } \varepsilon \to 0, 
\end{align*}
i.e., $\partial_{1} v _{\varepsilon} - \hat{g}_{1,\varepsilon} \to 0$ in $L^1(B_2(0))$. Moreover, we have 

\begin{align*}
 H_{\varepsilon}(\ue, \gae,C) \geq  \frac{1}{\varepsilon} \int_C f \bigg( \varepsilon \dashint_{B_{\varepsilon}(x)} | \tggae \cdot e_1| \dd y \bigg) \dd x \geq  \frac{1}{\varepsilon} \int_C  f \bigg(\varepsilon \dashint_{B_{\varepsilon}(x)} \hat{g}_{1,\varepsilon} \dd y \bigg) \dd x .
\end{align*}
Therefore, using the inequality $f(c-a) \geq f(c) - c_0 a$ for all $c \geq a \geq 0$ in combination with  $  \partial_{1} v _{\varepsilon} \geq \partial_{1} v _{\varepsilon} - \hat{g}_{1,\varepsilon}\geq 0$ on~$B_2(0)$, we obtain
\begin{align}
 H_{\varepsilon}(\ue, \gae,C) 
 & \geq \frac{1}{\varepsilon} \int_C f \bigg(\varepsilon \dashint_{B_{\varepsilon}(x)}    \partial_{1} v _{\varepsilon} \dd y - \varepsilon \dashint_{B_{\varepsilon}(x)} (  \partial_{1} v _{\varepsilon} - \hat{g}_{1,\varepsilon} ) \dd y \bigg) \dd x \notag\\
& \geq  \frac{1}{\varepsilon} \int_C f \bigg( \varepsilon \dashint_{B_{\varepsilon}(x)}  \partial_{1} v _{\varepsilon} \dd y \bigg) \dd x - c_0  \int_C \dashint_{B_{\varepsilon}(x)} (\partial_{1} v _{\varepsilon} - \hat{g}_{1,\varepsilon} )  \dd y \dd x. \label{splitsurfaceestimate}
\end{align}
Since $C=(-1,1) \times B^{d-1}_{1}(0_{d-1})$, we have by Fubini's theorem
\begin{align*}
 \frac{1}{\varepsilon} \int_C f \bigg( \varepsilon \dashint_{B_{\varepsilon}(x)}  \partial_{1} v _{\varepsilon}  \dd y \bigg) \dd x & = \frac{1}{\varepsilon} \int_{B^{d-1}_1(0_{d-1})} \int_{-1}^{1} f \bigg(\varepsilon \dashint_{B_{\varepsilon}(x)}    \partial_{1} v _{\varepsilon}  \dd y \bigg) \dd x_1 \, \dd x' \\
 & \geq \frac{\omega_{d-1}}{\varepsilon} \min_{x' \in \bar{B}^{d-1}_1(0_{d-1})} \int_{-1}^{1} f \bigg( \varepsilon \dashint_{B_{\varepsilon}(x_1,x')} \partial_{1} v _{\varepsilon} \dd y \bigg) \dd x_1.
\end{align*}
If the minimum on the right-hand side is attained at $x' \in \bar{B}^{d-1}_1(0_{d-1})$, then, by changing variables, we have
\begin{align*}
 \frac{1}{\varepsilon} \int_C f \bigg( \varepsilon \dashint_{B_{\varepsilon}(x)} \partial_{1} v _{\varepsilon}    \dd y \bigg) \dd x & \geq \frac{\omega_{d-1}}{\varepsilon} \int_{-1}^{1} f \bigg( \frac{1}{\omega_{d} \varepsilon^{d-1}} \int_{B_{\varepsilon}(x_1,x')}    \partial_{1} v _{\varepsilon}  \dd y \bigg) \dd x_1 \\
 &  =\omega_{d-1} \int_{-\frac{1}{\varepsilon}}^{\frac{1}{\varepsilon}} f \bigg( \frac{\varepsilon}{\omega_{d} } \int_{B_{1}(0)}    \partial_{1} v _{\varepsilon}(\varepsilon y + (\varepsilon t, x'))    \dd y \bigg) \dd t.
\end{align*}
We now define $w_{\varepsilon}(y)\coloneqq  \ve(\varepsilon y + (0,x'))$. Note that $w_{\varepsilon}$ is non-decreasing in the first variable, with $w_\varepsilon = u_0$ outside of $(-\beta_\varepsilon/\varepsilon,\alpha_\varepsilon/\varepsilon) \times \R^{d-1}$. If we extend $w_{\varepsilon}$ to $\R \times  B^{d-1}_1(0_{d-1})$ by the values $0$ and $a$, respectively, we obtain $w_{\varepsilon} \in W^{1,1}_{\text{loc}}(\R \times B^{d-1}_1(0_{d-1}))$. Morover, we have $\partial_1 w_{\varepsilon} \mathcal{L}^d \in \Mneds$, by observing via Fubini's theorem that $\partial_1 w_{\varepsilon} \mathcal{L}^d = a  \mathcal{L}^{d-1} \otimes (\tfrac{1}{a}\partial_1 w_{\varepsilon}(\cdot,y')   \mathcal{L}^{1}) $, where, by construction,  $\tfrac{1}{a}\partial_1 w_{\varepsilon}(\cdot,y')  \mathcal{L}^1$ is a probability measure in $\mathcal{M}(\R)$  for $ \mathcal{L}^{d-1}$-almost every $y' \in B^{d-1}_1(0_{d-1})$. This allows to estimate the non-local integral in terms of $w_\varepsilon$ and the functional~$G$ from~\eqref{eqn_def_G} 
\begin{align*}
 \frac{1}{\varepsilon} \int_C f \bigg( \varepsilon \dashint_{B_{\varepsilon}(x)}    \partial_{1} v _{\varepsilon}    \dd y \bigg) \dd x 
  & \geq \omega_{d-1} \int_{\R} f \bigg( \frac{1}{\omega_{d} } \int_{B_{1}(0)}    \partial_{1} w _{\varepsilon}(y+ (t, 0_{d-1}))    \dd y  \bigg) \dd t \\
  & = \omega_{d-1} \int_{\R} f \bigg( \frac{1}{\omega_{d} } \left( \partial_{1} w _{\varepsilon} \ast \chi_{B_1(0)} \right) (t, 0_{d-1}) \bigg) \dd t \\
  & = \omega_{d-1} G(\partial_1 w_{\varepsilon} \mathcal{L}^d).
\end{align*}
In conclusion, we then obtain by Lemma \ref{equalityonset}, the choice of (a suitable subsequence of) $\{(u_{\varepsilon},\gamma_\varepsilon)\}_{\varepsilon}$ according to Lemma~\ref{modificationsurface} and the estimate~\eqref{splitsurfaceestimate} 
\begin{align*}
H^{\prime}(u_0, \gamma_0, B_1(0)) & =H^{\prime}(u_0, \gamma_0, C) =  \liminf_{\varepsilon \to 0} H_{ \varepsilon}(\ue, \gae, C) \\ & \geq \liminf_{\varepsilon \to 0} \bigg[ \frac{1}{\varepsilon} \int_C  f \bigg( \varepsilon \dashint_{B_{\varepsilon}(x)}  \partial_{1} v _{\varepsilon} \dd y \bigg) \dd x - c_0 \int_C \dashint_{B_{\varepsilon}(x)} (\partial_{1} v _{\varepsilon} - \hat{g}_{1,\varepsilon} )  \dd y \dd x \bigg] \\ & \geq  \liminf_{\varepsilon \to 0} \omega_{d-1} G(\partial_1 w_{\varepsilon} \mathcal{L}^d) - \limsup_{\varepsilon \to 0} c_0\ \int_C  \dashint_{B_{\varepsilon}(x)} (\partial_{1} v _{\varepsilon} - \hat{g}_{1,\varepsilon} )  \dd y \dd x \\ & \geq \omega_{d-1}  \inf_{\mu \in \Mneds} G(\mu ) ,
\end{align*}
where the last inequality follows from the facts that $\partial_1 w_{\varepsilon} \mathcal{L}^d \in \Mneds$ and that $\partial_{1} v _{\varepsilon} - \hat{g}_{1,\varepsilon} \to 0$ in $L^1(B_2(0))$.
\end{proof}

Our next goal is to solve the minimization problem
\begin{equation}
\inf_{\mu \in \Mned} G(\mu) = \inf_{\mu \in \Mned} \int_{\R} \ff{\frac{1}{\omega_{d} } \left( \mu \ast \chi_{B_1(0)} \right) (t, 0_{d-1})  } \dd t \label{minimierungsproblemmeasure}
\end{equation}
from Lemma \ref{surfacebv2withoutslicing}. In this regard, we show that a solution~$\mu_0 \in \Mned$ exists (see Proposition~\ref{surfacebv2withoutslicing2}). The candidate $\mu_0$ for the minimizer of \eqref{minimierungsproblemmeasure} results from the following two considerations. Firstly, one can show that the extreme points of the set $\Mned$ are characterized by $\mu_{y'}=\delta_{t(y')}$ for $\HA$-a.e. $y' \in B^{d-1}_1(0_{d-1})$. Roughly speaking, it is secondly optimal to choose $\mu$ in such a way that $  \mu \ast \chi_{B_1(0)} $ takes the largest possible values, since $f$ cuts them off. But this is tantamount to choose $\mu$ such that the support of $  \mu \ast \chi_{B_1(0)} $ is minimal, since $\| \mu \ast  \chi_{B_1(0)} (\, \cdot \, , 0_{d-1}) \|_{L^1(\R)}= |a-b|$ for all $\mu \in \Mned$ (cf. Lemma \ref{majorisation} (i)). It is easy to see that this happens if $\delta_{t(y')}= \delta_c$ for $\HA$-a.e. $y \in B^{d-1}_1(0_{d-1})$, where $c$ is a given constant. Due to the translations invariance of $G$, we are free to choose a certain value for $c$. In order to solve the minimization problem~\eqref{minimierungsproblemmeasure} rigorously, we use the majorization principle. The key elements are shown in the following lemma.

\begin{lemma}\label{majorisation}
Let $a,b \in \R$ with $a \neq b$. The measure $\mu_0 = |a-b| \, \HA \mrs \, (\{0\} \times {B}^{d-1}_{1}(0_{d-1}))$ belongs to $\Mned$. Moreover, we have for all $\mu \in \Mned$ and for all $c \geq 0$
\begin{enumerate}[font=\normalfont, label=(\roman{*}), ref=(\roman{*})]
\item\label{majorisation_1} $ \begin{aligned} \int_{  \R }{ \frac{1}{\omega_d}(\mu_0 \ast \chi_{B_1(0)})(t, 0_{d-1})} \dd t =  \int_{  \R } { \frac{1}{\omega_d}(\mu \ast \chi_{B_1(0)})(t,0_{d-1})} \dd t  = |a-b|\end{aligned},$
\item\label{majorisation_2} $ \begin{aligned} \int_{  \R } \min \bigg \lbrace{ \frac{1}{\omega_d}(\mu_0 \ast \chi_{B_1(0)})(t, 0_{d-1})} , c \bigg  \rbrace \dd t \leq  \int_{  \R } \min  \bigg \lbrace{ \frac{1}{\omega_d}(\mu \ast \chi_{B_1(0)})(t,0_{d-1})}, c \bigg \rbrace \dd t \end{aligned},$
\item\label{majorisation_3} $ \begin{aligned} \int_{  \R }\left(  \frac{1}{\omega_d}(\mu_0 \ast \chi_{B_1(0)})(t, 0_{d-1}) -c \right) ^+ \textup{d} t \geq  \int_{  \R } \left(  \frac{1}{\omega_d}(\mu \ast \chi_{B_1(0)})(t, 0_{d-1}) -c \right)^+ \textup{d} t \end{aligned}.$
\end{enumerate}
\end{lemma}

\begin{proof}
We first notice $\mu_0 \in \Mned$ (with constant probability measures~$\delta_0$ in $\mathcal{M}(\R)$), since for every Borel measurable set $A \times B \subset \R \times {B}^{d-1}_{1}(0_{d-1})$ we have 
\begin{equation*}
 \mu_0(A \times B) = |a-b|  \int_{B} \int_{A} \dd \delta_0(x_1)  \dd x'.
\end{equation*}

Consider $\mu \in \Mned$ arbitrary.
Since 
\begin{equation*}
 (\mu \ast \chi_{B_1(0)})(t,0_{d-1})  =  |a-b| \int_{B^{d-1}_1(0_{d-1})} \int_{\R} \chi_{(- \sqrt{1-|x'|^2}, \sqrt{1-|x'|^2})}(x_1 -t) \dd \mu_{x'}(x_1) \dd x'
\end{equation*}
for all $t\in \R$, Fubini's theorem, the fact that $\{\mu_{x'}\}_{x' \in B^{d-1}_1(0_{d-1})}$ is a family of probability measures on~$\R$ and Cavalieri's principle (notice that $\omega_1=2$) then implies the claim in~\ref{majorisation_1}:
\begin{align*}
& \int_{\R}\frac{1}{\omega_d}(\mu \ast \chi_{B_1(0)})(t,0_{d-1}) \dd t \\
&=  \frac{|a-b|}{\omega_d} \int_{B^{d-1}_1(0_{d-1})}  \int_{\R}  \int_{\R} \chi_{(- \sqrt{1-|x'|^2}, \sqrt{1-|x'|^2})}(x_1 -t)\dd t \dd \mu_{x'}(x_1)  \dd x'\\ 
& = \frac{|a-b|}{\omega_d} \int_{B^{d-1}_1(0_{d-1})}  2 \sqrt{1-|x'|^2}   \dd x' = |a-b|.
\end{align*}
Let $c \geq 0$ be fixed. Note that if $c \geq |a-b| ({\omega_{d-1}} /{\omega_{d}})$, then the assertion in~\ref{majorisation_2} is obviously satisfied due to~\ref{majorisation_1} and the fact that $\| \mu \ast \chi_{B_1(0)} \|_{L^{\infty}(\Rd)} \leq |a-b| \omega_{d-1}$.  Hence, we may assume $c < |a-b| ({\omega_{d-1}} /{\omega_{d}})$. We choose $r<1$ such that 
\begin{equation*}
{|a-b| ({\omega_{d-1}} /{\omega_{d}}) r^{d-1} = c}.
\end{equation*}
We define functions $f_c \colon \R \to \R$ and $g_\ell \colon \R \to [0, \infty)$ for every $\ell \in (0,1]$ by $f_c(t)\coloneqq  \min\lbrace t, c \rbrace$ and 
\begin{equation}
g_\ell(t)\coloneqq  \frac{|a-b|}{\omega_d} \int_{B^{d-1}_\ell(0_{d-1})} \chi_{ (  - \sqrt{1-|x'|^2} , + \sqrt{1-|x'|^2} ) } (t) \dd x' \label{specialgl}
\end{equation} 
for all $t \in \R$. In particular, we then have $g_1(t) = ({1}/{\omega_d})(\mu_0 \ast \chi_{B_1(0)})(t,0_{d-1})$ for $t \in \R$. We claim that 
\begin{equation}
 g_r(t) = f_c(g_r(t))=  f_c(g_1(t))  \quad \text{for all }t \in \R. \label{equalityforgr}
\end{equation}
The first equality follows immediately from the observation
\begin{align*}
g_r(t) \leq \frac{|a-b|}{\omega_d} \int_{B^{d-1}_r(0_{d-1})} \dd x' = \frac{|a-b|}{\omega_d} \omega_{d-1} r^{d-1} = c \quad \text{for all }t \in \R.
\end{align*} 
For the second equality, we distinguish two cases. If $t \in [  - \sqrt{1-r^2} , + \sqrt{1-r^2}]$, then $g_r(t)=c$ and therefore $c \geq f_c(g_1(t)) \geq f_c(g_r(t)) = f_c(c) =c$. Otherwise if $t \notin [  - \sqrt{1-r^2} , + \sqrt{1-r^2}]$, then $g_1(t)=g_r(t)$ follows from
\begin{align*}
\frac{|a-b|}{\omega_d} \int_{ B^{d-1}_1(0_{d-1}) \setminus B^{d-1}_r(0_{d-1})} \chi_{ (  - \sqrt{1-|x'|^2} , + \sqrt{1-|x'|^2} ) } (t) \dd x' = 0.
\end{align*}

We get
\begin{align}
 \frac{1}{\omega_d}(\mu \ast \chi_{B_1(0)})(y_1, 0_{d-1}) \notag & = 
\frac{|a-b|}{\omega_d} \int_{ B^{d-1}_1(0_{d-1})}  \int_{\R} \chi_{(- \sqrt{1-|x'|^2} ,  \sqrt{1-|x'|^2}) }(x_1-t) \dd \mu_{x'}(x_1) \dd x' \notag \\
 & \geq \frac{|a-b|}{\omega_d} \int_{ B^{d-1}_r(0_{d-1})} \int_{\R} \chi_{( - \sqrt{1-|x'|^2} , \sqrt{1-|x'|^2} ) }(x_1-t) \dd \mu_{x'}(x_1) \dd x'. \label{ggeneral}
\end{align}
Since the integrand of the right-hand side is bounded by~$1$, we observe that the right-hand side is bounded by $\tfrac{|a-b|}{\omega_d} \omega_{d-1} r^{d-1} = c$, meaning that $f_c$ acts as identity on it. Therefore, it follows  from \eqref{ggeneral}, \eqref{specialgl} and \eqref{equalityforgr} that 
\begin{align*}
& \int_{  \R } f_c \bigg(  \frac{1}{\omega_d}(\mu \ast \chi_{B_1(0)})(t, 0_{d-1})\bigg) \dd  t \\
& \geq  \int_{  \R } f_c \bigg(  \frac{|a-b|}{\omega_d} \int_{ B^{d-1}_r(0_{d-1})}  \int_{\R}\chi_{( - \sqrt{1-|x'|^2} ,  \sqrt{1-|x'|^2}) }(x_1-t) \dd \mu_{x'}(x_1) \dd x' \bigg) \dd  t\\
& = \int_{  \R } \frac{|a-b|}{\omega_d} \int_{ B^{d-1}_r(0_{d-1})} \int_{\R}\chi_{(  - \sqrt{1-|x'|^2} ,  \sqrt{1-|x'|^2}) }(x-t) \dd \mu_{x'}(x_1) \dd x' \dd  t\\
& = \int_{\R} \frac{|a-b|}{\omega_d} \int_{ B^{d-1}_r(0_{d-1})} \chi_{ ( - \sqrt{1-|x'|^2} , \sqrt{1-|x'|^2}) } (t) \dd x' \dd t \\
&= \int_{\R}  g_r(t)\dd t =  \int_{\R} f_c\left( g_1(t)\right)\dd t = \int_{  \R } f_c\bigg(  \frac{1}{\omega_d}(\mu_0 \ast \chi_{B_1(0)})(t, 0_{d-1})\bigg) \dd  t ,
\end{align*}
where we have used the Fubini theorem (twice) and the fact that $\mu_{x'}$ are probability measures for $x' \in B^{d-1}_1(0_{d-1})$ in order to pass from the third to the forth line. 
This shows assertion~\ref{majorisation_2}.
Finally, we infer~\ref{majorisation_3} from~\ref{majorisation_1} and~\ref{majorisation_2} together with $(t-c)^+ = t - \min\{t,c\}$ for $t \in \R$ and $c \geq 0$.
\end{proof}

We are now able to identify $\mu_0$ as a solution of the minimization problem~\eqref{minimierungsproblemmeasure} and to compute the minimal value:

\begin{proposition}
\label{surfacebv2withoutslicing2}
Let   $a,b \in \R$ with $a \neq b$.  Then $\mu_0 = |a-b| \, \HA \mrs \, (\{0\} \times {B}^{d-1}_{1}(0_{d-1}))$ satisfies
\begin{equation*}
\inf_{\mu \in \Mned} G(\mu) = G(\mu_0) = \theta(a-b),
\end{equation*}
where $G$ is defined in~\eqref{eqn_def_G} and $\theta$ in~\eqref{definitiontheta}.
\end{proposition}

\begin{proof}
According to Lemma~\ref{majorisation}~\ref{majorisation_2} with $c=1$, we have $G(\mu_0) \leq  G(\mu)$ for all $\mu \in \Mned$. Using the definition of~$\mu_0$ we then conclude 
\begin{align*}
\inf_{\mu \in \Mned} G(\mu)  =  G(\mu_0) & = \int_{\R} f \bigg( \frac{1}{\omega_d}(\mu_0 \ast \chi_{B_1(0)})(y_1, 0_{d-1}) \bigg) \dd y_1 \\ &= \int_{\R} f \bigg(\frac{|a-b|}{\omega_d} \int_{B^{d-1}_1(0_{d-1})} \chi_{ \left(  - \sqrt{1-|x'|^2} , + \sqrt{1-|x'|^2} \right) } (t) \dd x' \bigg) \dd t\\
&= 2 \int_{0}^1 f \bigg(\frac{\omega_{d-1}}{\omega_d}|a-b| (1-t^2)^{\frac{d-1}{2}}  \bigg) \dd t  = \theta(a-b). \qedhere
\end{align*}
\end{proof}

We are now in the position to establish the estimate from below for the surface term:

\begin{Prop}\label{surfacebv}
Let $A \subset \Omega$ be open. For every $(u,\gamma) \in  BV(\Omega) \times \Mam$  with $\gamma = D^su+ g \mathcal{L}^1$ and $\nabla u- g \in L^2(\Omega,\R^d)$ we have
\begin{equation*}
E^{\prime}(u,\gamma,A) \geq \int_{J_u \cap A} \theta ([u]) \dd \HA.
\end{equation*}
\end{Prop}

\begin{proof}
According to Corollary \ref{surfacebvCor}, we have
\begin{align*}
E^{\prime}(u,\gamma,A) \geq \int_{J_u \cap A} h' \dd \HA, 
\end{align*} 
with $h'(x_0) \coloneqq \omega_{d-1}^{-1} H^{\prime}(u_0, \gamma_0, B_1(x_0))$ for $\HA$-a.e. $x_0 \in J_u$. Combining Proposition \ref{surfacebv2withoutslicing} and Proposition \ref{surfacebv2withoutslicing2}, we have $h'(x_0) \geq \theta([u](x_0))$, which finishes the proof of the proposition.
\end{proof}

\begin{remark}
\label{rem_majorization_generalization}
From Lemma \ref{majorisation}~\ref{majorisation_1} and~\ref{majorisation_3} we obtain that $\mu_0$ satisfies a majorization principle (originally expressed in terms of decreasing rearrangements). From the Hardy, Littlewood and P\'{o}lya inequality (also known as Karamata's inequality), see e.g.~\cite[Proposition H.1.a]{MOA11}, we then find 
\begin{equation*}
\int_{\R} \hh{\frac{1}{\omega_{d} } \left( \mu_0 \ast \chi_{B_1(0)} \right) (t, 0_{d-1})  } \dd t = \inf_{\mu \in \Mned} \int_{\R} \hh{\frac{1}{\omega_{d} } \left( \mu \ast \chi_{B_1(0)} \right) (t, 0_{d-1})  } \dd t
\end{equation*}
for every continuous concave function $h \colon [0, \infty) \to [0, \infty)$ with $h(0)=0$, which generalizes the first equality in Proposition~\ref{surfacebv2withoutslicing2}. Hence, we obtain in fact the stronger result that the minimizer of \eqref{minimierungsproblemmeasure} is also a solution of a generalized optimization problem where the function~$f$ in the functional~$G$ is replaced by such a function~$h$.
\end{remark}

\begin{remark}
A closer look at Lemma \ref{surfacebv2withoutslicing}, Lemma \ref{majorisation} and Remark \ref{rem_majorization_generalization} shows that the concept of majorization leads to an alternative proof of the corresponding results in \cite{LVDD07}.
\end{remark}

\subsection{Estimate from below of the volume term}

Now, we turn to the estimate from below of the volume term. For this purpose we first show the following auxiliary lemma:

\begin{lemma}\label{estimatevolumeax}
Let $C > 0$ and let $U \subset \R^d$ be bounded and open. Let $u \in BV(U)$ and let $\{v_k\}_k$ be a sequence in $SBV(U)$ with $\sup_{k}\mathcal{H}^{d-1}(J_{v_k}) < \infty$ and $v_k \weakstar u$ in $BV(U)$. Let $w \in L^2(U,\Rd)$ and let $\{w_k\}_k$ be a sequence in $L^2(U, \Rd)$ with $w_k \weakly w$ in $L^2(U, \Rd)$. Then 
$$ \liminf_{k \to \infty} \int_{U} | \nabla v_k + w_k | \dd x 
   \ge \int_{U} | \nabla u + w | \dd x. $$ 
\end{lemma} 

\begin{proof} 
\emph{Step 1: Construction of modifications of~$v_k$ which are equal to~$v_k$ on regions where these functions are close to the limit~$u$.} 
Fix $\eps > 0$. For $0 < \eta < \eps$ we define $g_{\eta}(t) = \max \{ -\eta, \min \{ t, \eta \} \}$ for $t \in \R$. By monotone convergence we can choose $\eta$ so small that 
\begin{align}\label{eq:smalljumps} 
  |g_{\eta} \circ [u] \mathcal{H}^{d-1} \lfloor J_u|(U) 
  = \int_{J_u} \min\{ |[u]|, \eta \} \, \mathrm{d} \mathcal{H}^{d-1} 
  < \eps. 
\end{align}
Now we set 
\begin{equation*}
\tilde{v}_k 
   \coloneqq  u + g_{\eta} \circ (v_k - u) \quad \text{for } k \in \N. 
\end{equation*}   
By the chain rule for $BV$-functions (see, e.g., \cite[Thm.\ 3.99]{afp}), $g_{\eta} \circ (v_k - u)$ belongs to $BV(U)$ with 
\begin{align*}
  D ( g_{\eta} \circ (v_k - u) ) 
  &= g_{\eta}'(v_k - u) \nabla (v_k - u) \mathcal{L}^d + g_{\eta}'(\overline{v_k - u}) D^c (v_k - u) \\ 
  &\qquad + [g_{\eta} \circ (v_k - u)] \nu_{v_k - u} \mathcal{H}^{d-1} \mrs  J_{v_k - u}. 
\end{align*}
Here $\overline{(v_k - u)}(x)$ denotes the approximate limit of $v_k - u$ at $x$ (see Section~\ref{sec_preliminaries}). Notice that $S_{v_k - u}$ is $|D^c (v_k - u)|$-negligible. Moreover, the sets ${\{ x \in U \colon | v_k(x) - u(x) | = \eta \}}$ and $\{ x \in U \colon | \overline{(v_k - u)}(x) | = \eta \}$ are negligible with respect to $\nabla (v_k - u) \mathcal{L}^d$ and $D^c (v_k - u)$, respectively. Therefore, setting $G_k = \{ x \in U \colon | v_k(x) - u(x) | < \eta \}$ and $\overline{G}_k = \{ x \in U \colon | \overline{(v_k - u)}(x) | < \eta \}$, we obtain 
\begin{equation*}
  D ( g_{\eta} \circ (v_k - u) ) 
  = \chi_{G_k} \nabla (v_k - u) \mathcal{L}^d + \chi_{\overline{G}_k}  D^c (v_k - u) + [g_{\eta} \circ (v_k - u)] \nu_{v_k - u} \mathcal{H}^{d-1} \lfloor J_{v_k - u} 
\end{equation*}
and so 
\begin{equation*}
  D \tilde{v}_k 
  = ( \chi_{G_k} \nabla v_k + \chi_{U \setminus G_k} \nabla u ) \mathcal{L}^d + \chi_{U \setminus \overline{G}_k} D^c u + D^j u + [g_{\eta} \circ (v_k - u)] \nu_{v_k - u} \mathcal{H}^{d-1} \mrs J_{v_k - u}. 
\end{equation*}
We now estimate the total variation of the last term on the right hand side. We first note that for $\mathcal{H}^{d-1} \lfloor (J_u \setminus J_{v_k})$-a.e.\ $x$ there holds
\begin{align*}
  |[g_{\eta} \circ (v_k - u)](x)| 
  &= |g_{\eta}(\overline{v_k}(x) - u^+(x)) - g_{\eta}(\overline{v_k}(x) - u^-(x))| \\ 
  &\le 2 |g_{\eta}(u^-(x) - u^+(x))| 
   = 2 |g_{\eta}([u](x))|, 
\end{align*}
where we have used the elementary estimate $|g_{\eta}(t) - g_{\eta}(s)| \le 2 \min\{|t-s|,\eta\} = 2|g_{\eta}(t-s)|$ for all $s,t\in\mathbb{R}$. With \eqref{eq:smalljumps} and $C = \sup_{k} |D \tilde{v}_k|(U)$ we thus obtain   
\begin{align}\label{eq:smalljumps-est}
\begin{split}
  &|[g_{\eta} \circ (v_k - u)] \nu_{v_k - u} \mathcal{H}^{d-1} \lfloor J_{v_k - u}|(U) \\  
  &\qquad \le 2 \eta \mathcal{H}^{d-1}(J_{v_k}) +  |[g_{\eta} \circ (v_k - u)] \mathcal{H}^{d-1} \lfloor (J_u \setminus J_{v_k})|(U) 
  \le 2 C \eta + 2 \eps. 
\end{split}
\end{align}
For later use we remark that this in particular shows that $\sup_{k} |D \tilde{v}_k|(U) < \infty$, which in turn implies that $\{\tilde{v}_k\}_{k}$ is bounded in $BV(U)$. Because of $v_k \to u$ in $L^1(U)$ 
we thus have also that  
$$ \tilde{v}_k \weakstar u \mbox{ in } BV(U). $$

\emph{Step 2: Proof of the claim with $\{v_k\}_k$ replaced by $\{\tilde{v}_k\}_k$ up to an error of order~$\eps$.} We have 
\begin{align*}
  D^s \tilde{v}_k 
  &= \chi_{U \setminus \overline{G}_k} D^c u + D^j u + [g_{\eta} \circ (v_k - u)] \nu_{v_k - u} \mathcal{H}^{d-1} \mrs J_{v_k - u}. 
\end{align*}
Passing if necessary to a subsequence (not relabeled), we may assume that 
\begin{equation*}
\chi_{U \setminus \overline{G}_k} \frac{\textnormal{d} D^c u}{\textnormal{d} |D^c u|} \weakly h \text{ in }L^2(|D^c u|, \Rd)
\end{equation*}
and 
\begin{equation*}
[g_{\eta} \circ (v_k - u)] \nu_{v_k - u} \mathcal{H}^{d-1} \mrs J_{v_k - u} \weakstar f \mathcal{L}^d + \mu^s\text{ in }\mathcal{M}(U, \Rd)
\end{equation*}
for some $h \in L^2(|D^c u|, \Rd)$, $f \in L^1(U, \Rd)$   and a $\mu^s \in \mathcal{M}(U,\Rd)$ which is singular to $\mathcal{L}^d$.
Then $D^s \tilde{v}_k \weakstar f \mathcal{L}^d + \varrho^s$ in $\mathcal{M}(U, \Rd)$, where $\varrho^s = h |D^c u| + D^j u + \mu^s$ is singular to $\mathcal{L}^d$ as well. From \eqref{eq:smalljumps-est} and $\eta \leq \eps$ it follows that $\| f \|_{L^1(\Omega)}\le 2 (C+1) \eps$. But then 
\begin{align*} 
  \nabla \tilde{v}_k + w_k 
  &= D \tilde{v}_k - D^s \tilde{v}_k + w_k \\ 
  & \weakstar Du - f \mathcal{L}^d - \varrho^s + w 
    = (\nabla u - f + w) \mathcal{L}^d + D^s u - \varrho^s \text{ in }\mathcal{M}(U, \Rd),
\end{align*} 
which implies the claim
\begin{equation*}
 \liminf_{k \to \infty} \int_{U} | \nabla \tilde{v}_k + w_k | 
   \ge \int_{U} | \nabla u - f + w | \ge \int_{U} | \nabla u + w | - 2 (C + 1) \eps.
\end{equation*}   

\emph{Step 3: Proof of the claim for the original sequence $(v_k)_k$.} To this end, we first note that, since $v_k \to u$ in $L^1(U)$ and thus in measure, we have $\lim_{k \to \infty} | U \setminus G_k | = 0$. Using $\nabla \tilde{v}_k = \chi_{G_k} \nabla v_k + \chi_{U \setminus G_k} \nabla u$, we estimate 
\begin{equation*}
  \int_{U} | \nabla v_k + w_k | 
  \ge \int_{G_k} | \nabla \tilde{v}_k + w_k | 
  = \int_{U} | \nabla \tilde{v}_k + w_k | - \int_{U \setminus G_k} | \nabla u + w_k |  
\end{equation*}
and passing to the $\liminf$ we obtain 
\begin{equation*}
\liminf_{k \to \infty} \int_{U} | \nabla v_k + w_k | 
   \ge \int_{U} | \nabla u + w | - 2 (C+1) \eps, 
\end{equation*}   
where we have used that $(\nabla u + w_k)_k$ is equiintegrable. Since~$\eps>0$ was arbitrary, the assertion follows. 
\end{proof}

Now, we can show the estimate from below for the volume part.

\begin{Prop}\label{ACPPMBV}
Let $A \subset \Omega$ be open. For every $(u,\gamma) \in  BV(\Omega) \times \Mam$ with $\norm{u}_{L^{\infty}(\Omega)} \leq K$, $\gamma = D^su+ g \mathcal{L}^d$ and $\nabla u- g \in L^2(\Omega)$, we have
\begin{equation}
E^{\prime}(u,\gamma,A) \geq \int_A {|\nabla u- g|^2 \dd x } + c_0 \int_A { |g| \dd x}.  \label{ACPMDBV}
\end{equation}
\end{Prop}

\begin{proof}
Let $\{(\ue,\gamma_\varepsilon)\}_{\varepsilon}$ be a sequence in $W^{1,1}(\Omega) \times \Mam$ with $\norm{\ue}_{L^{\infty}(\Omega)} \leq K$, $\gae = \ggae \mathcal{L}^d$ for a function $\ggae \in L^1(\Omega,\R^d)$ satisfying $\nabla \ue- \ggae \in L^2(\Omega,\R^d)$ for every $\varepsilon >0$ such that $u_\varepsilon \to u$ in $L^1(\Omega)$, $\gamma_\varepsilon \to \gamma$ in the flat norm on~$\Omega$ and
\begin{equation*}
E_{\varepsilon}(\ue,\gae, A) \to E^{\prime}(u,\gamma, A) < \infty
\end{equation*}
(with all elements being finite). Note that \eqref{ACPMDBV} is trivial if $ E^{\prime}(u,\gamma, A)= \infty$. Let $\delta \in (0,1)$ be arbitrary. In view of Proposition~\ref{clmBV}, we find a sequence $\{(\ve,\breve{\gamma}_{\varepsilon} = \breve{g}_{\varepsilon}\mathcal{L}^d)\}_{\varepsilon}$ in $SBV(A) \times \mathcal{M}(A, \R^{d})$  such that $\ve \weakstar u$ in $BV(A)$ and with $\norm{\ve}_{L^{\infty}(A)} \leq \norm{\ue}_{L^{\infty}(A)}\leq K$, $\breve{g}_{\varepsilon} \in L^1(A, \Rd)$, ${\HA(J_{\ve} \cap A_{6 \varepsilon}-) \leq c E_{\varepsilon}(\ue,\gae,A)}$, and
\begin{equation}
E_{\varepsilon}(\ue,\gae,A) \geq \int_A {|\nabla \ve- \breve{g}_{\varepsilon}|^2 \dd x } + (1- \delta) c_0 \,\int_A { |\breve{g}_{\varepsilon}| \dd x} \label{estimatevolumetotal}
\end{equation}
for every $\varepsilon > 0$. We next observe that $\{D^j\ve +\breve{g}_{\varepsilon} \mathcal{L}^d \}_{\varepsilon}$ converges to $D^su + g \mathcal{L}^d$ in the flat norm on~$A$. This follows by rewriting $D^j\ve +\breve{g}_{\varepsilon} \mathcal{L}^d = D\ve + (\breve{g}_{\varepsilon}  - \nabla \ve ) \mathcal{L}^d$ for $\varepsilon > 0$, combined with $\ue, \ve \to u$ in $L^1(\Omega)$, $\ggae \mathcal{L}^d \to \gamma =  D^su+ g \mathcal{L}^d$ in the flat norm, the uniform boundedness of $\{\breve{g}_{\varepsilon}  - \nabla \ve\}_\varepsilon$ in $L^2(\Omega)$ and  
\[
|A \setminus A_{6\varepsilon- }| +  \big| \{x \in A_{6 \varepsilon -} \colon \nabla v_\varepsilon(x) - \breve{g}_\varepsilon(x) \neq \nabla u_\varepsilon(x) - g_\varepsilon(x) \} \big| \to 0  \quad \text{ as } \varepsilon \to 0. \]
From the equiboundedness of $\{E_{\varepsilon}(\ue,\gae,A)\}_\varepsilon$ and from the uniqueness of the limit we get that $\{\nabla \ve - \breve{g}_{\varepsilon}\}_{\varepsilon} $ weakly converges, up to subsequences, to $\nabla u - g$ in $L^2(A, \Rd).$ 
Hence, we have
\begin{equation}
\liminf_{\varepsilon \to 0} \int_A {|\nabla \ve- \breve{g}_{\varepsilon}|^2 \dd x }  \geq \int_A |\nabla u - g|^2 \dd x. \label{estimatevolumefirstpart}
\end{equation}
Let $B$ be an open subset of $\R^d$ with $B \Subset A$ and $7 \varepsilon < \text{dist}(B, \partial A)$, which implies $B \subset A_{6\varepsilon}-$ and in turn that $\HA(J_{\ve} \cap B)$ is bounded independently of $\varepsilon$. We now consider the sequence $\{w_{\varepsilon}\}_{\varepsilon}$ defined by $w_{\varepsilon} \coloneqq  \breve{g}_{\varepsilon} - \nabla v_{\varepsilon}$ for $\varepsilon >0$. Since $\{w_{\varepsilon}\}_{\varepsilon}$ converges weakly to $w \coloneqq g - \nabla u$ in $L^2(A, \Rd)$, we can apply Lemma~\ref{estimatevolumeax} with the sequence $\{\ve, w_{\varepsilon}\}_{\varepsilon}$ on the set $U=B$. In this way, we obtain 
\begin{equation}
\label{estimatevolumeBVh1}
 \liminf_{\varepsilon \searrow 0} \int_B { |\breve{g}_{\varepsilon}| \dd x} = \liminf_{\varepsilon \searrow 0} \int_B { | \nabla \ve +w_{\varepsilon} | \dd x}  \geq \int_B { |\nabla u +w| \dd x} = \int_B { |g| \dd x}
\end{equation}
Summing up \eqref{estimatevolumetotal}, \eqref{estimatevolumefirstpart} and \eqref{estimatevolumeBVh1}, we arrive at
\begin{equation*}
\liminf_{\varepsilon \searrow 0} E_{\varepsilon}(\ue,\gae,A) \geq \int_A |\nabla u - g|^2 \dd x + (1- \delta) c_0 \,\int_B { |g| \dd x}. \label{estimatevolumetotal2}
\end{equation*}
Since $B \Subset A$ is arbitrary, the assertion~\eqref{ACPMDBV} then follows by taking the supremum over $\delta > 0$.
\end{proof}

\subsection{Estimate from below of the Cantor term}

We finally deal with the estimate from below of the Cantor term. We basically follow the idea of Lussardi and Vitali in \cite{LVDD07}. First, we replace the average over the $d$-dimensional ball by an average over a $d$-dimensional cube in the second term. Via Fubini's theorem we then split this average into an average on a $(d-1)$-dimensional cube and an average on an interval, which in turn allows to apply the slicing method and to employ the corresponding one-dimensional result from~\cite{AuerVolkmannBeckSchmidt:22}. For this splitting, we use the notation 
\begin{align*}
{Q}_{r}^{d-1}(y') \coloneqq \big\{ z' \in \R^{d-1} \colon |z'_i - y'_i| \leq r,\text{ for } i \in \{1, \dots, d-1\} \big\}
\end{align*}
for the $(d-1)$-dimensional (axis-aligned) cube with center $y' \in \R^{d-1}$ and side length~$2r$. We start by recalling some facts on averaging a sequence of functions on (shrinking) $(d-1)$-dimensional cubes. 

\begin{lemma}[Lemma 4.4 in \cite{LVDD07}]\label{hlf1BV2}
Let $A \subset \R^{d-1}$ be open and $a,b \in \R$ with $a < b$. Let $\{\ue\}_{\varepsilon}$ be a sequence in $L^1((a,b) \times A)$ such that $u_{\varepsilon} \to u$ in $L^1((a,b) \times A )$ for some $u \in L^1((a,b) \times A)$.
For a.e.~$t \in (a,b)$, for every $y' \in A$  and for every $\varepsilon>0$ with $\varepsilon < ({1}/{\sqrt{d-1}}) \textnormal{ dist}(y',\partial A)$ we define
\begin{align*}
\widehat{u}_{\varepsilon}^{y'}(t) \coloneqq  \dashint_{\Qd_{\varepsilon}(y')} u_{\varepsilon}(t,z') \dd z'.
\end{align*}
Then, there exists a \textup{(}not relabeled\textup{)} subsequence  such that $\widehat{u}_{\varepsilon}^{y'} \to  u(\cdot,y')$ in $L^1(a,b)$ for a.e.~$y' \in A$.
\end{lemma}

\begin{remark}
\label{remark_slices_bounded_differentiable}
The averaging over the $(d-1)$-dimensional cubes in Lemma~\ref{hlf1BV2} also preserves regularity. In particular, there holds 
\begin{itemize}
 \item if $\ue \in L^{\infty}((a,b) \times A)$, then $\widehat{u}_{\varepsilon}^{y'} \in L^\infty(a,b)$ for a.e.~$y' \in A$, with 
 \begin{equation*}
  \|\widehat{u}_{\varepsilon}^{y'}\|_{L^{\infty}(a,b)} \leq \norm{\ue}_{L^{\infty}((a,b) \times A)},
 \end{equation*}
 \item if $\ue \in W^{1,1}((a,b) \times A)$, then $\widehat{u}_{\varepsilon}^{y'} \in W^{1,1}(a,b)$ for a.e.~$y' \in A$, with 
 \begin{equation*}
 (\widehat{u}_{\varepsilon}^{y'})^\prime(t) =   \dashint_{\Qd_{\varepsilon}(y')}  \nabla u_{\varepsilon}(t,z') \cdot e_1 \dd z' \text{ for }\mathcal{L}^1\text{-a.e. } t \in (a,b).
 \end{equation*}
\end{itemize}
\end{remark}

We now come to the estimate from below for the Cantor part.

\begin{Prop}\label{CPPMBV}
Let $A \subset \Omega$ be open.
For every $(u,\gamma) \in  BV(\Omega) \times \Mam$ with $\norm{u}_{L^{\infty}(\Omega)} \leq K$, $\gamma= D^su +g \mathcal{L}^d$ and $\nabla u - g \in L^2(\Omega, \R^d)$ we have
\begin{equation}
\label{cantorbv}
E^{\prime}(u,\gamma,A) \geq c_0 |D^c u|(A).  
\end{equation}
\end{Prop}

\begin{proof}
Let $\{(\ue,\gae)\}_{\varepsilon}$ be a sequence in $W^{1,1}(\Omega) \times \Mam$ with $\norm{\ue}_{L^{\infty}(\Omega)} \leq K$, $\gae=\ggae \mathcal{L}^d$ for a function $\ggae \in L^1(\Omega,  \R^d)$ and $\nabla \ue - \ggae \in L^2(\Omega, \Rd)$ for every $\varepsilon >0$ such that $\ue \to u$ in $L^1(\Omega)$, $\gae \to \gamma_0$ in the flat norm on $\Omega$ and
\begin{equation*}
E_{\varepsilon}(\ue,\gae, A) \to E^{\prime}(u,\gamma, A) < \infty 
\end{equation*}
(with all elements being finite). Note that if $ E^{\prime}(u, \gamma, A)= \infty$, the estimate in \eqref{cantorbv} is trivial. We shall establish the lower bound via slicing. For this purpose, we fix a slicing direction $\xi \in \Sm$. The linear hyperplane orthogonal to $\xi$ is denoted by $\Pi_{\xi}$. For every $y  \in \Pi_{\xi}$ and  every open subset $B$ of $\R^d$ we consider the one-dimensional set 
\begin{equation*}
B^ {\xi, y} \coloneqq \lbrace t \in \R: y +t \xi \in B \rbrace.
\end{equation*}
In order to recover all of~$B$, we further consider the projection of $B$ on $\Pi_\xi$, that is, the set $\Pi_{\xi}^B \coloneqq \{ y \in \Pi_{\xi} \colon \exists t \in \R \colon y+t \xi \in B \}$. In what follows, it is not restrictive to assume $\xi = e_1$ for notational convenience, for which we simply have $\Pi_{\xi}= \{0\} \times \R^{d-1}$ and hence $\Pi_{\xi}^B =   \{0\} \times B^{P_1}$ with
\begin{equation*}
 B^{P_1} \coloneqq \{ y' \in \R^{d-1} \colon \exists t \in \R \colon (t,y') \in B\}.
\end{equation*}
We now estimate the two terms in~$E_{\varepsilon}(\ue,\gae,A)$ separately. To this end, we fix a set $B \Subset A$ with  $\sqrt{d-1} \varepsilon < \text{dist}(B, \partial A)$.  By Fubini and Jensen's inequality, we first observe for all $h \in \N$
\begin{align}
\int_A |\nabla u_{\varepsilon}-g_{\varepsilon}|^2 \dd x & = \int_{A^{P_1}}  \int_{A^{e_1,(0,x')}} \left| \nabla u_{\varepsilon}(t,x') -  \ggae(t,x') \right|^2 \dd t \dd x' \notag \notag\\
&\geq \int_{B^{P_1}}  \int_{B^{e_1,(0,x')}}  \dashint_{Q^{d-1}_{\varepsilon/h}(x')} \left| \nabla u_{\varepsilon}(t,z') -  \ggae(t,z') \right|^2 \dd z' \dd t \dd x' \notag \\
& \geq \int_{B^{P_1}}  \int_{B^{e_1,(0,x')}} \bigg| \, \dashint_{Q^{d-1}_{\varepsilon/h}(x')} \left(  \nabla u_{\varepsilon}(t,z') - \ggae(t,z')  \right) \cdot e_1 \dd z' \bigg|^2 \dd t \dd x' \notag\\
& \geq \int_{B^{P_1}}  \int_{B^{e_1,(0,x')}}|(\widehat{u}_{\varepsilon,h}^{e_1,x'})^{\prime}(t)-\widehat{g}_{\varepsilon,h}^{e_1,x'}(t)|^2 \dd t \dd x', \label{estimatefirstterm}
\end{align}
where we have defined 
\begin{align*}
 \widehat{u}_{\varepsilon,h}^{e_1,x'} (s) & \coloneqq \dashint_{Q^{d-1}_{\varepsilon/h}(x')} u_{\varepsilon}(s,z') \dd z' ,\\
 \widehat{g}_{\varepsilon,h}^{e_1,x'}(s) & \coloneqq \dashint_{Q^{d-1}_{\varepsilon/h}(x')}  g_{\varepsilon}(s,z')\cdot e_1\dd z' 
\end{align*}
for a.e.~$s \in B^{e_1,(0,x')}$ and $x' \in B^{P_1}$. Notice that we have here used the corresponding formula for the derivative of~$\widehat{u}_{\varepsilon,h}^{e_1,x'}$ from Remark~\ref{remark_slices_bounded_differentiable}, for every $h \in \N$ fixed.
Secondly, by a slightly modified version of  \cite[Lemma 4.3]{LuVi97} (where in the integrand the gradient of a $W^{1,1}$-function~$u$ is replaced by the $L^1$-function~$g_\varepsilon$), there exists a sequence $\{c_h\}_h$ of positive real numbers in $(0,1)$ with $c_h \nearrow 1$ as $h \to \infty$ such that there holds
\begin{equation*}
\frac{1}{\varepsilon} \int_A f \bigg( \varepsilon \dashint_{B_{\varepsilon}(x) \cap \Omega} |\ggae|\dd y \bigg) \dd x  \geq \frac{1}{\varepsilon} \int_{B} f \bigg( c_h \varepsilon  \dashint_{Q_{\varepsilon/h}(x)} |g_{\varepsilon}| \dd z \bigg)  \dd x.
\end{equation*}
By Fubini's theorem and the triangle inequality we further find for the inner integral on the right-hand side of the previous inequality in a similar way as above
\begin{equation*}
\dashint_{Q_{\varepsilon/h}(t,x')} |g_{\varepsilon}| \dd z  \geq  \dashint_{t -  \frac{\varepsilon}{h}}^{t +  \frac{\varepsilon}{h}} \bigg| \,  \dashint_{Q^{d-1}_{\varepsilon/h}(x')} g_{\varepsilon}(s,z') \cdot e_1  \dd z' \bigg| \dd s  = \dashint_{t -  \frac{\varepsilon}{h}}^{t +  \frac{\varepsilon}{h}}  \vert \widehat{g}_{e,h}^{e_1,x'}(s)\vert \dd s.
\end{equation*}
Thus, by monotonicity of~$f$ and once again Fubini's theorem we have
\begin{equation}
 \frac{1}{\varepsilon} \int_A f \bigg( \varepsilon \dashint_{B_{\varepsilon}(x) \cap \Omega} |\ggae|\dd y \bigg) \dd x \geq \frac{1}{\varepsilon}  \int_{B^{P_1}}   \int_{B^{e_1, x'}} f \bigg( c_h \varepsilon \dashint_{t -  \frac{\varepsilon}{h}}^{t +  \frac{\varepsilon}{h}}  |\widehat{g}_{e,h}^{e_1,x'}| \dd s \bigg)  \dd t \dd x'. \label{Cantor_estimatesecondterm}
\end{equation}
Combining~\eqref{estimatefirstterm} and \eqref{Cantor_estimatesecondterm} we infer  
\begin{align}
 E_{\varepsilon}(\ue,\gae,A)  & \geq  \int_{B^{P_1}}  \int_{B^{e_1,(0,x')}}|(\widehat{u}_{\varepsilon,h}^{e_1,x'})^{\prime}-\widehat{g}_{\varepsilon,h}^{e_1,x'}|^2 \dd t \dd x' \notag \\ & \quad +\frac{1}{\varepsilon}  \int_{B^{P_1}}   \int_{B^{e_1, (0,x')}} f \bigg( c_h \varepsilon \dashint_{t -  \frac{\varepsilon}{h}}^{t +  \frac{\varepsilon}{h}}  |\widehat{g}_{e,h}^{e_1,x'}| \dd s \bigg)  \dd t \dd x',
 \label{eqn_Cantor_intermediate}
 \end{align}
and by Fatou's lemma, we then obtain
\begin{align*}
 \lefteqn{E^{\prime}(u,\gamma, A) = \lim_{\varepsilon \to 0} E_{\varepsilon}(\ue,\gae, A)} \\
 & \geq \liminf_{\varepsilon \to 0}  \int_{B^{P_1}} \bigg[  \int_{B^{e_1,(0,x')}}|(\widehat{u}_{\varepsilon,h}^{e_1,x'})^{\prime}-\widehat{g}_{\varepsilon,h}^{e_1,x'}|^2 \dd t   + \frac{1}{\varepsilon}  \int_{B^{e_1,(0,x')}} f \bigg( c_h \varepsilon  \dashint_{t - \frac{\varepsilon}{h}}^{t + \frac{\varepsilon}{h}}  |\widehat{g}_{\varepsilon,h}^{e_1,x'}| \dd s \bigg)  \dd t \bigg] \dd x' \\
 & \geq \int_{B^{P_1}} \liminf_{\varepsilon \to 0} \bigg[  \int_{B^{B^{e_1,(0,x')}}}|(\widehat{u}_{\varepsilon,h}^{e_1,x'})^{\prime}-\widehat{g}_{\varepsilon,h}^{e_1,x'}|^2 \dd t  +\frac{1}{\varepsilon}  \int_{B^{e_1,(0,x')}} f \bigg( c_h \varepsilon   \dashint_{t - \frac{\varepsilon}{h}}^{t + \frac{\varepsilon}{h}} |\widehat{g}_{\varepsilon,h}^{e_1,x'}| \dd s \bigg)  \dd t \bigg] \dd x'.
\end{align*}

We next want to estimate the right-hand side via (a pointwise application of) the one-dimensional result in \cite{AuerVolkmannBeckSchmidt:22}. To this end, we first observe that, for $\mathcal{L}^{d-1}$-a.e.~$x' \in B^{P_1}$, the function $u^{e_1,x'} \coloneqq u(\cdot,x')$ belongs to $BV(B^{e_1,(0,x')})$, cf.~\cite[Theorem 3.103]{afp}, with uniform bound $\|u^{e_1,x'}\|_{L^\infty(B^{e_1,(0,x')})} \leq K$. Moreover, by Lemma~\ref{hlf1BV2}, we have the convergence $\widehat{u}_{\varepsilon,h}^{e_1,x'} \to u^{e_1,x'}$ 
for every $h \in \N$ fixed. We now rewrite the inner integral of the previous inequality as
\begin{align*}
\lefteqn{ \int_{B^{e_1,(0,x')}}|(\widehat{u}_{\varepsilon,h}^{e_1,x'})^{\prime}-\widehat{g}_{\varepsilon,h}^{e_1,x'}|^2 \dd t  +\frac{1}{\varepsilon}  \int_{B^{e_1, (0,x')}} f \bigg( c_h \varepsilon   \dashint_{t - \frac{\varepsilon}{h}}^{t + \frac{\varepsilon}{h}} |\widehat{g}_{\varepsilon,h}^{e_1,x'}| \dd s \bigg)  \dd t } \\
& = \frac{1}{c_h^2 h^2} \bigg[ \int_{B^{e_1,(0,x')}}|(c_h h \widehat{u}_{\varepsilon,h}^{e_1,x'})^{\prime}-c_h h \widehat{g}_{\varepsilon,h}^{e_1,x'}|^2 \dd t  + \frac{h}{\varepsilon}  \int_{B^{e_1, (0,x')}} (c_h^2 h f) \bigg( \frac{\varepsilon}{h}   \dashint_{t - \frac{\varepsilon}{h}}^{t + \frac{\varepsilon}{h}} |c_h h \widehat{g}_{\varepsilon,h}^{e_1,x'}| \dd s \bigg)  \dd t \bigg]
\end{align*}
and notice that the function $(c_h^2 h f)(t)$ is of the same form as $f$, with $c_0$ replaced by $c_h^2 h c_0$. Passing to a subsequence (possibly depending on $x'$) we may achieve that the $\liminf$ as $\eps \to 0$ of this expression is actually a limit. Moreover, seeking for a lower bound, we may without loss of generality assume that this limit is finite. We now apply the one-dimensional compactness result in \cite[Theorem~4.2]{AuerVolkmannBeckSchmidt:22} to extract a further subsequence for which $\widehat{g}_{\varepsilon,h}^{e_1,x'} \to \gamma^{e_1,x'}$ in the flat norm for some measure $\gamma^{e_1,x'}$ on $B^{e_1, (0,x')}$ whose singular part equals $D^s  (u^{e_1, x'})$. Together with $\widehat{u}_{\varepsilon,h}^{e_1,x'} \to u^{e_1,x'}$ the one-dimensional lower bound in \cite[Proposition 5.5]{AuerVolkmannBeckSchmidt:22} shows that the $\liminf$ above is estimated from below by
\begin{equation*}
 \frac{1}{c_h^2 h^2} \bigg[  c_h^2 h c_0 |D^c (c_h h u^{e_1, x'})|(B^{e_1,(0,x')}) \bigg] = c_h c_0  | D^c (u^{e_1, x'})|(B^{e_1,(0,x')}) 
\end{equation*}
for $\mathcal{L}^{d-1}$-a.e.~$x' \in B^{P_1}$. Therefore, we have 
\begin{equation*}
E^{\prime}(u,\gamma, A)  \geq c_h  c_0\int_{B^{P_1}} |D^cu^{e_1, x'}|(B^{e_1,(0,x')}) \dd x' = c_h c_0 | D^cu \cdot e_1|(B),
\end{equation*}
where the last equality follows from the fact that the Cantor part of the derivative can be recovered from the corresponding parts of the derivatives of the restrictions according to \cite[Theorem 3.108]{afp}. With $c_h \nearrow 1$ as $h \to \infty$ and the arbitrariness of $B \Subset A$ we then end up with $E^{\prime}(u,\gamma,A) \geq c_0| D^cu \cdot  e_1|(A)$. Since this inequality holds for any $\xi \in \Sm$ instead of $e_1$, we have proved so far
\begin{equation}
E^{\prime}(u,\gamma,A) \geq c_0 |D^c u \cdot \xi|(A) \quad \text{for every } \xi \in \Sm. \label{cantorbv_prelim} 
\end{equation}
Finally, let $\lambda\coloneqq  D^cu$, and for $\xi \in \Sm$ we define $\lambda_{\xi}\coloneqq  \lambda \cdot \xi$, implying that $|\lambda_{\xi}| = |  \tfrac{\dd \lambda}{\dd |\lambda|} \cdot \xi|  \dd |\lambda|$. By \eqref{cantorbv_prelim}  we have
\begin{equation*}
E^{\prime}(u,\gamma,A) \geq c_0| D^cu \cdot \xi|(A) = c_0 \int_A \psi_{\xi} \dd |\lambda|,
\end{equation*}
where $\psi_{\xi}= | \tfrac{\dd \lambda}{\dd |\lambda|} \cdot \xi |$.
Via the measure theoretic result \cite[Lemma 15.2]{br} (applied with the set function $\mu(\cdot) \coloneqq E^{\prime}(u,\gamma,\cdot)$) we then  deduce the claim
\begin{align*}
E^{\prime}(u,\gamma,A) &  \geq c_0 \int_A \sup_{\xi} \psi_{\xi} \dd |\lambda| =  c_0 |\lambda|(A)= c_0 |D^cu|(A). \qedhere
\end{align*}
\end{proof}

\subsection{Conclusion and proof of Theorem~\ref{mainresultmBV}~\ref{mainresultmBV_1}}

In this section we complete the proof of the $\Gamma$-$\liminf$-inequality. For this purpose, we combine the previous results by means of measure theory.

\begin{theorem}\label{liminfnBV}
For every $(u,\gamma) \in  BV(\Omega) \times \Mam$ with $\norm{u}_{L^{\infty}(\Omega)} \leq K$, $\gamma = D^su+ g \mathcal{L}^d$ and $\nabla u- g \in L^2(\Omega, \Rd)$ we have
\begin{equation*}
E^{\prime}(u,\gamma) \geq E(u, \gamma).  
\end{equation*}
\end{theorem}

\begin{proof}
By the Propositions~\ref{ACPPMBV}, \ref{CPPMBV} and~\ref{surfacebv} we have proved the following lower bounds for the volume, the Cantor and the jump part
\begin{enumerate}
 \item $E^{\prime}(u,\gamma,A) \geq \int_A {|\nabla u - g |^2\dd x } + c_0 \,\int_A { |g| \dd x}$,
 \item $E^{\prime}(u,\gamma,A) \geq c_0 |D^c u|(A)$,
 \item $E^{\prime}(u,\gamma,A) \geq \int_{J_u \cap A} \theta ([u]) \dd \HA$,
\end{enumerate}
for every open subset $A$ of $\Omega$. We now consider the Radon measure~$\lambda$ given by 
\begin{equation*}
\lambda \coloneqq  \mathcal{L}^d + \HA \mrs  J_u + |D^cu|. 
\end{equation*}
Let $C$ be a Borel subset of $\Omega\setminus J_u$ with $\mathcal{L}^d( C ) =0$ such that $\vert D^cu\vert (\Omega \setminus C)=0$. Then, we obtain
\begin{equation*}
\mu(A)\coloneqq E^{\prime}(u,\gamma,A)\geq \int_A \psi_i(x) \dd \lambda
\end{equation*}
for $ 1\leq i \leq 3$ and for every open subset $A$ of $\Omega$, where
\begin{align*}
 \psi_1 & \coloneqq  \big( |\nabla u - g|^2 + c_0 |g| \big) \chi_{\Omega \setminus (J_u  \cup C)}, \\
 \psi_2 & \coloneqq  \theta ([u]) \chi_{J_u}, \\
 \psi_3 & \coloneqq  c_0 \chi_C. 
\end{align*}
Next, we define
\begin{align*}
\psi(x)\coloneqq \sup_i \psi_i(x) =  \left\{
\begin{array}{l l}
\vert \nabla u(x) -g(x)\vert^2 +c_0 \vert g(x)\vert & \quad \text{if } x \in \Omega \setminus (J_u  \cup C),  \\ 
\theta ( [u](x))  & \quad \text{if } x \in J_u, \\
c_0 & \quad \text{if } x \in C.
\end{array}
\right.
\end{align*} 
By a measure theoretic result (see e.g.~\cite[Lemma 15.2]{br} applied with the set function $\mu(\cdot) \coloneqq E^{\prime}(u,\gamma,\cdot)$) we conclude that
\begin{equation*}
E^{\prime}(u,\gamma,A) \geq \int_A \sup_i \psi_i \dd \lambda = \int_A \psi \dd\lambda = E(u, \gamma,A)
\end{equation*}
for every open subset $A$ of $\Omega$. With the choice $A = \Omega$, this finishes the proof of the theorem. 
\end{proof}

\section{Estimate from above of the $\Gamma$-upper limit}\label{sec:upper-lim}

We now turn to the estimate from above of the $\Gamma$-upper limit $E^{\prime \prime}$ in order to conclude the proof of Theorem \ref{mainresultmBV}.
The basic tools for the proof are the relaxation result of Corollary~\ref{rel2conBV} and a suitable version of the density result \cite[Theorem~3.1]{CT99} for $SBV$-functions by Cortesani and Toader, which allows to reduce the problem of finding a recovery sequence to the case where the function~$u$ belongs to $\Wde$, that is, it is a function in $SBV(\Omega)$ with $u \in W^{k, \infty}(\Omega \setminus \overbar{J_u})$ for every $k$, satisfies $\HA(\overbar{J_u}\setminus {J_u} ) = 0$ and for which $\overbar{J_u}$ is a finite union of disjoint $(d-1)$-simplices with  $\overbar{J_u}\subset\Omega$. The main idea of the construction of the recovery sequence is to first find an approximation $\{\ue\}_{\varepsilon}$ of~$u$  in an energetically optimal way and to then take $\{\gae\}_{\varepsilon}$ defined by $\gae\coloneqq  D^s\ue + \ggae \mathcal{L}^d$ with $g_{\varepsilon} \coloneqq g - \nabla u +\nabla \ue$ for every $\varepsilon >0$ as approximation for~$\gamma$. We now give the precise construction under the aforementioned constraint $u  \in \Wde$.

\begin{Prop}\label{WdrBV}
For every $(u,\gamma) \in \Wd \times \Mam$  with $\gamma = D^su+ g \mathcal{L}^d$ and $\nabla u- g \in L^2(\Omega,\R^d)$ we have 
\begin{equation*}
E^{\prime \prime} (u,\gamma) \leq E(u, \gamma).
\end{equation*}
\end{Prop}

\begin{proof}
Since $u \in  \Wde$, we have that $J_u=J^1 \cup \ldots \cup J^m$ for pairwise disjoint ($d-1$)-simplices $J^1, \ldots , J^m$. It is not restrictive to assume $m=1$ and $J\coloneqq J^1 \subset \{ x \in \R^d \colon x_1= 0\}$.
Let $\delta > 0$. We denote by $J_{\delta} \coloneqq J_{\delta+}$ the $\delta$-neighborhood of $J$. Moreover, we will also consider $J_{ \varepsilon +\delta}$ and $J_{2 \varepsilon +\delta}$. Now, setting $\delta\coloneqq \varepsilon^2$ we may assume that $J_{2\varepsilon +\delta} \subset \Omega$ for $\varepsilon >0$ sufficiently small, and we next define functions~$\ue$, $\aee$ and $g_\varepsilon$ for such $\varepsilon > 0$. We define $\ue \colon \Omega \to \R$ by setting $\ue \coloneqq u$ on $ \Omega \setminus J_{\delta}$, while we define it via linear interpolation on~$J_{\delta}$ as follows: 
If $x \in J_{\delta}$ lies on the segment $(z_1,z_2)$ for $z_1,z_2 \in \partial J_{\delta}$ such that $z_2=z_1 + \lambda e_1$ for some $0 \leq \lambda \leq 2 \delta$, then
\begin{equation*}
\ue(x) \coloneqq  \frac{|x-z_1|}{|z_2-z_1|}u(z_2) + \frac{|z_2-x|}{|z_2-z_1|}u(z_1).
\end{equation*}
Since the pair $(z_1, z_2)$ depends on $x$, we write $(z_1(x),z_2(x))$ hereafter in order to express this dependence clearly.

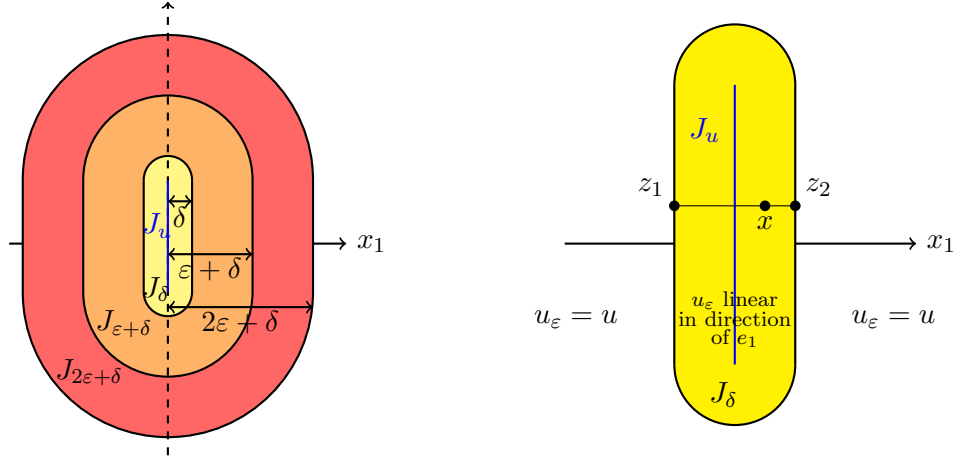
\begin{center}
\begin{tikzpicture}[scale= 0.2]

\draw[->, thick] (-10.5,0) -- (11.8,0) node[right]  {$x_1$};

\fill[red!60] (0,4.2) circle (9.6);
\fill[red!60] (0,-3.2) circle (9.6);
\fill[red!60] (-9.6,-3.2) rectangle (9.6,4.2);

\fill[orange!60] (0,-3.2) circle (5.6);
\fill[orange!60] (0,4.2) circle (5.6);
\fill[orange!60] (-5.6,-3.2) rectangle (5.6,4.2);

\fill[yellow!60] (0,-3.2) circle (1.6);
\fill[yellow!60] (0,4.2) circle (1.6);
\fill[yellow!60] (1.6,-3.2) rectangle (-1.6,4.2);

\draw[thick] (9.6,4.2) arc (0:180:9.6);
\draw[thick] (9.6,-3.2) arc (0:-180:9.6);
\draw[thick] (9.6,4.2) -- (9.6,-3.2);
\draw[thick] (-9.6,4.2) -- (-9.6,-3.2);

\draw[thick] (5.6,4.2) arc (0:180:5.6);
\draw[thick] (5.6,-3.2) arc (0:-180:5.6);
\draw[thick,-] (5.6,-3.2) -- (5.6,4.2);
\draw[thick,-] (-5.6,-3.2) -- (-5.6,4.2);

\draw[thick] (1.6,4.2) arc (0:180:1.6);
\draw[thick] (1.6,-3.2) arc (0:-180:1.6);
\draw[thick,-] (1.6,-3.2) -- (1.6,4.2);
\draw[thick,-] (-1.6,-3.2) -- (-1.6,4.2);

\draw[->, thick, dashed] (0,-14) -- (0,16);

\draw[-,thick, blue] (0,-3.2) -- (0,4.2) ;

\node at (-0.7,1.2) {\textcolor{blue}{$J_u$}};

\draw[<->, thick] (0,2.8) -- (1.6, 2.8);
\node at (0.8, 1.8) {$\delta$};
\node at (-0.8,-3) {$J_{\delta}$};

\draw[<->, thick] (0,-0.7) -- (5.6, -0.7);
\node at (2.8, -1.7) {$\varepsilon + \delta$};
\node at (-2.9,-5.3) {$J_{\varepsilon + \delta}$};

\draw[<->, thick] (0,-4.2) -- (9.6, -4.2);
\node at (4.8, -5.2) {$2\varepsilon + \delta$};
\node at (-5.2,-8.5) {$J_{2\varepsilon + \delta}$};

\begin{scope}[scale =2.5, xshift=15cm]

\draw[->, thick] (-4.5,0) -- (4.8,0) node[right]  {$x_1$};
 
\fill[yellow] (0,4.2) circle (1.6);
\fill[yellow] (0,-3.2) circle (1.6);
\fill[yellow] (-1.6,-3.2) rectangle (1.6,4.2);

\draw[thick] (1.6,4.2) arc (0:180:1.6);
\draw[thick,-] (1.6,-3.2) -- (1.6,4.2);
\draw[thick,-] (-1.6,-3.2) -- (-1.6,4.2);
\draw[thick] (1.6,-3.2) arc (0:-180:1.6);

\node[blue] at (-0.8,3) {$J_{u}$};
\draw[thick, blue, -] (0,-3.2) -- (0,4.2);

\node at (-0.3,-4) {$J_{\delta}$};

\coordinate [label={[label distance=0cm]135: $z_1$}](A) at (-1.6,1);
\coordinate [label={[label distance=0cm]45: $z_2$}](B) at (1.6,1);

\coordinate [label=below: $x$](C) at (0.8,1);
\node (H1) at (-4.2,-2) {${u_{\varepsilon}=u}$};
\node (H2) at (4.2,-2) {${u_{\varepsilon}=u}$};
\node (H5) at (0,-2) {${\substack{u_{\varepsilon} \text{ linear} \\ \text{in direction} \\ \text{of }e_1 }}$};
\draw (A) --(B);
\fill (A) circle (4pt);
\fill (B) circle (4pt);
\fill (C) circle (4pt);

\draw[-] (-1.6,-3.2) -- (-1.6,4.2);
\draw[-] (1.6,-3.2) -- (1.6,4.2);

\end{scope}
\end{tikzpicture}
\end{center}
Moreover, we define $\aee \colon \Omega \to \R$ by $\aee \coloneqq 0$ on $\Omega \setminus J_{\delta}$ and
\begin{equation}
 \aee (x) \coloneqq \frac{u(z_2(x))-u(z_1(x))}{|z_2(x)-z_1(x)|} \quad \text{ for all } x \in J_{\delta}. \label{definitionaepsilon}
\end{equation}
Finally, we define $\gae = g_{\varepsilon} \mathcal{L}^d \in \Mam$ by $g_{\varepsilon} \coloneqq g - \nabla u +\nabla \ue$ (which is equal to~$g$ on~$\Omega\setminus J_{\delta}$). By these definitions, we immediately find $\ue \to u$ in $L^1(\Omega)$ and thus $D\ue = \nabla \ue \mathcal{L}^d \to Du$ and in turn $\gae \to \gamma=  D^su + g \mathcal{L}^d$ in the flat norm on~$\Omega$. Moreover, we have $\nabla \ue -\ggae = \nabla u- g$ for every $\varepsilon>0$ and consequently
\begin{equation}
 \int_{\Omega} |\nabla \ue - \ggae |^2 \dd x =  \int_{\Omega} |\nabla u - g|^2 \dd x.
 \label{estimateabove0}
\end{equation}
In order to estimate $E_{\varepsilon}(\ue,\gae)$ from above, we still need to control the non-local contribution. We first notice by subadditivity of $f$ combined with $f(r) \leq c_0 r$ for $r \geq 0$ and keeping in mind that $J_{2 \varepsilon + \delta} \subset \Omega$
\begin{align}
 \lefteqn{ \frac{1}{\varepsilon} \int_{\Omega} f \bigg( \varepsilon \dashint_{B_{\varepsilon}(x) \cap \Omega} |\ggae| \dd y \bigg) \dd x } \notag \\ 
 & = \frac{1}{\varepsilon} \int_{\Omega \setminus J_{\varepsilon + \delta}} f \bigg( \varepsilon \dashint_{B_{\varepsilon}(x) \cap \Omega} |g| \dd y \bigg) \dd x + \frac{1}{\varepsilon} \int_{J_{\varepsilon +\delta}} f \bigg( \varepsilon \dashint_{B_{\varepsilon}(x)} |g - \nabla u + \nabla \ue| \dd y \bigg) \dd x \notag \\
 & \leq \int_{\Omega \setminus J_{\varepsilon +\delta}}  c_0 \dashint_{B_{\varepsilon}(x) \cap \Omega} |g| \dd y \, \dd x +  \int_{J_{\varepsilon +\delta}} c_0  \dashint_{B_{\varepsilon}(x)} |g - \nabla u| \dd y  \, \dd x \notag\\
& \quad +\int_{J_{\varepsilon +\delta}} c_0{ \dashint_{B_{\varepsilon}(x)} | \nabla \ue - \aee e_1| \dd y} \dd x +  \frac{1}{\varepsilon} \int_{J_{\varepsilon +\delta}} f \bigg( \varepsilon \dashint_{B_{\varepsilon}(x)} |\aee| \dd y \bigg) \dd x.
\label{estimateabove1}
\end{align}
In view of $\nabla u- g \in L^2(\Omega, \Rd)$, we have for the second term on the right-hand side
\begin{align}
\int_{J_{\varepsilon +\delta}} c_0 \dashint_{B_{\varepsilon}(x)} |g - \nabla u| \dd y  \, \dd x \to 0  \quad \text{ as } \varepsilon \to 0. \label{estimateabove2}
\end{align}
Furthermore, $|\nabla \ue - \aee  e_1|$ is bounded on $\Omega$ by a uniform constant due to the regularity of~$u$, and hence, we get for the third term
\begin{align}
\int_{J_{\varepsilon +\delta}} c_0{ \dashint_{B_{\varepsilon}(x)} \betrag{ \nabla \ue - \aee \, e_1} \dd y} \dd x \to 0 \quad \text{ as } \varepsilon \to 0. \label{estimateabove3}
\end{align}

\noindent
\begin{minipage}{0.66\textwidth}
In order to find an appropriate estimate for the fourth term, we distinguish the two cases, whether we do or do not reach a jump point of $u$ when we go from one point of $\partial J_{\delta}$ to the other point of $\partial J_{\delta}$ (opposite) in $e_1$-direction. For this purpose we define 
\begin{align*}
\tilde{J}_{\delta} \coloneqq  \big\{ y \in J_{\delta} \colon & \exists x \in J_u, \, \lambda \in [-\delta,\delta] \text{ with } y + \lambda e_1 =x\big\}.
\end{align*}
We make three observations, which directly follow from the fact that $u~\in~W^{k, \infty}(\Omega \setminus \overbar{J_u})$. For fixed $\varrho > 0$ we have for sufficiently small $\varepsilon$:
\end{minipage}\hfill
\begin{minipage}{0.3\textwidth}
 
\begin{center}
\begin{tikzpicture}[scale= 0.4]
\draw[thick,->] (-3.8,0) -- (3.8,0) node[right]  {$x_1$};

\fill[violet!50] (0,4.2) circle (1.6);
\fill[violet!50] (0,-3.2) circle (1.6);

\fill[yellow] (-1.6,-3.2) rectangle (1.6,4.2);

\draw[thick] (1.6,4.2) arc (0:180:1.6);
\draw[thick] (1.6,-3.2) arc (0:-180:1.6);
\draw[thick,-] (-1.6,-3.2) -- (-1.6,4.2);
\draw[thick,-] (1.6,-3.2) -- (1.6,4.2);

\draw[-, thick, blue] (0,-3.2) -- (0,4.2) ;
\node[blue]  at (-0.8,1.5) {$J_u$};

\node (A) at (0.8,-2.3) {$\tilde{J}_{\delta}$};
\node[violet] at (-3.2, 4.5) {${J_{\delta} \setminus \tilde{J}_{\delta}}$};

\end{tikzpicture}
\end{center} 
 
\end{minipage}

\begin{enumerate}[font=\normalfont, label=(\roman{*}), ref=(\roman{*})]
 \item\label{limsup_observation_1} For all $x_1, x_2, x^0_1, x^0_2 \in \partial {J}_{\delta}$ with $|x_1 - x^0_1| \leq \varepsilon$ and $|x_2 - x^0_2| \leq \varepsilon$ such that $x_2 = x_1 + \mu e_1$ and $x^0_2 = x^0_1 + \mu^0 e_1$ for $0 \leq \mu ,\mu^0 \leq  2 \delta$, we have 
 \begin{equation*}
 |(u(x_2) - u(x_1))- (u(x^0_2)- u(x^0_1)) | \leq \varrho.
 \end{equation*}
 \item\label{limsup_observation_2} For all $x_1, x_2 \in \partial \tilde{J}_{\delta}  \cap \partial J_{\delta} $ such that $x_2 = x_1 + \mu e_1$ for $0 \leq \mu  \leq  2 \delta$ and all $x \in J_u$ on the segment $(x_1,x_2)$, we have 
 \begin{equation*}
 |(u(x_2) - u(x_1))- (u(x+)- u(x-)) | \leq \varrho.
 \end{equation*}
 \item\label{limsup_observation_3} For all $x_1, x_2 \in \partial J_{\delta} \setminus \partial \tilde{J_{\delta}}$  such that $x_2 = x_1 + \mu e_1$ for $0 < \mu  \leq  2 \delta$ there exists $L>0$ such that
 \begin{equation*}
 \frac{|u(x_2) - u(x_1)|}{|x_2 -x_1|}\leq L.
 \end{equation*}
\end{enumerate}
Since $a_{\varepsilon}=0$ on $\Omega \setminus J_{\delta}$, we deduce from the subadditivity of $f$ for the last term in~\eqref{estimateabove1}
\begin{align}
 \frac{1}{\varepsilon} \int_{J_{\varepsilon + \delta}} \ff{\varepsilon \dashint_{B_{\varepsilon}(x)} |\aee| \dd y } \dd x & \leq \frac{1}{\varepsilon} \int_{J_{\varepsilon + \delta}} f \bigg( \frac{1}{\omega_d \, \varepsilon^{d-1}}\int_{B_{\varepsilon}(x)\cap (J_{\delta}\setminus \tilde{J}_{\delta})} | \aee | \dd y \bigg) \dd x \notag \\
 & \quad + \frac{1}{\varepsilon} \int_{J_{\varepsilon + \delta}} f \bigg( \frac{1}{\omega_d \, \varepsilon^{d-1}}\int_{B_{\varepsilon}(x)\cap \tilde{J}_{\delta}} |\aee| \dd y \bigg) \dd x. \label{splittingsmoothnotsmooth}
\end{align}
Consistently with \eqref{definitionaepsilon}, we denote by $y_1(x)$ and $y_2(x)$ the points on $\partial J_{\delta}$ such that a given point $x\in J_{\delta}$ lies on the segment $(y_1(x), y_2(x))$.  Using the definition~\eqref{definitionaepsilon} of~$a_{\varepsilon}$ and~\ref{limsup_observation_3} we infer that the integrand $| \aee|$ of the inner integral of the first term on the right-hand side of~\ref{splittingsmoothnotsmooth} is bounded by~$L$. Hence, we get
 \begin{equation}
  \frac{1}{\varepsilon} \int_{J_{\varepsilon + \delta}} f \bigg( \frac{1}{\omega_d \, \varepsilon^{d-1}}  \int_{B_{\varepsilon}(x)\cap (J_{\delta}\setminus \tilde{J}_{\delta})} |\aee| \dd y \bigg) \dd x \leq    \frac{1}{\varepsilon} \int_{J_{\varepsilon + \delta}} \ff{ \varepsilon L} \dd x \leq c_0 \, L |J_{\varepsilon + \delta}|. \label{estimateaboveniceboundary}
 \end{equation}
 Let 
$J^{P_1}_{\varepsilon + \delta} \subset \R^{d-1}$ be the projection of $J_{\varepsilon + \delta}$ onto $\lbrace x_1= 0\rbrace$. Since $J_{\varepsilon + \delta} \subset (-\varepsilon - \delta, \varepsilon + \delta) \times J^{P_1}_{\varepsilon + \delta}$, Fubini's theorem yields for the second term
 \begin{multline}
 \frac{1}{\varepsilon} \int_{J_{\varepsilon + \delta}} f \bigg(\frac{1}{\omega_d \, \varepsilon^{d-1}} \int_{B_{\varepsilon}(x)\cap \tilde{J}_{\delta}}\betrag{\aee} \dd y \bigg) \dd x  \\
   \leq\frac{1}{\varepsilon} \int_{J^{P_1}_{\varepsilon + \delta}}  \int_{-\varepsilon - \delta}^{\varepsilon + \delta} f \bigg( \frac{1}{\omega_d \, \varepsilon^{d-1}}\int_{B_{\varepsilon}(s,x') \cap \tilde{J}_{\delta}} |\aee| \dd y \bigg) \,\dd s \, \dd x'. \label{estimateabovecriticalpart1}
 \end{multline}
For all $y=(y_1,y') \in \tilde{J}_{\delta}$ we have $z_1(y) -z_2(y) = 2 \delta$ and $(z_1(y), z_2(y)) = (z_1(0,y'), z_2(0,y'))$.
By observations~\ref{limsup_observation_1} and~\ref{limsup_observation_2} and again by the definition of $a_{\varepsilon}$, we thus obtain 
\begin{align*}
 \int_{B_{\varepsilon}(s,x')\cap \tilde{J}_{\delta}} |\aee| \dd y 
  & = \int_{B_{\varepsilon}(s,x')\cap \tilde{J}_{\delta}} \frac{\betrag{u(z_2(0,y'))-u(z_1(0,y'))}}{2 \delta} \dd (y_1,y')\\
       & \leq \int_{B_{\varepsilon}(s,x') \cap \tilde{J}_{\delta} }\frac{\betrag{u(y_2(0,x'))-u(y_1(0,x'))}+ \varrho}{2 \delta} \dd (y_1,y')\\
    &   \leq \int_{B_{\varepsilon}(s,x') \cap \tilde{J}_{\delta} }\frac{\betrag{u((0,x')+)-u((0,x')-)} + 2 \varrho }{2 \delta} \dd (y_1,y')\\
  & \leq  \omega_{d-1} \big[ \varepsilon^2 - ( \max\{ |s|-\delta, 0 \} ) ^2 \big]^{\frac{d-1}{2}}  \big( |u((0,x')+)-u((0,x')-)| + 2 \varrho\big).
\end{align*}
The expression $\max\lbrace{|s|-\delta, 0}\rbrace$ is necessary to distinguish between two possible cases. In the first case, if $(s,x') \notin \tilde{J}_{\delta}$ (and hence $|s| - \delta \geq 0$), the projection of $B_{\varepsilon}(s,x')\cap \tilde{J}_{\delta}$ onto $\{x_1=0\}$ is contained in a $(d-1)$-dimensional ball of radius $\sqrt{\varepsilon^2 - ( |s|-\delta)^2 })$ (see the left figure). In the second case, if $(s,x') \in \tilde{J}_{\delta}$ (and hence $|s| - \delta < 0$), this projection can only be guaranteed to lie in a $(d-1)$-dimensional ball of radius $\varepsilon$ (see the right figure).
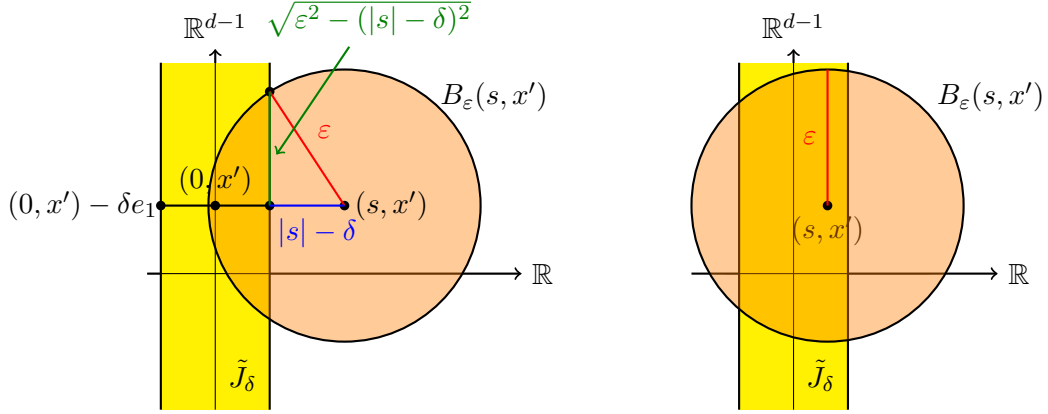
\begin{center}
\begin{tikzpicture}[scale= 0.45]

\draw[thick,->] (-2,0) -- (9,0) node[right]  {$\R$};
\draw[thick,->] (0,-4.0) -- (0,6.6) node[above]	{$\R^{d-1}$};

\fill[yellow] (-1.6,-4.0) rectangle (1.6,6.2);

\draw[thick,-] (1.6,-4.0) -- (1.6,6.2);
\draw[thick,-] (-1.6,-4.0) -- (-1.6,6.2);

\draw[-] (0,-4.0) -- (0,6.2);
\draw[-] (-2,0) -- (9,0);

\node (A) at (0.8,-3) {$\tilde{J}_{\delta}$};

\fill[orange, opacity= 0.4] (3.8,2) circle (4);
\draw[thick] (3.8,2) circle (4);

\node at (8.2,5.2) {$B_{\varepsilon}(s,x')$};

\coordinate [label= right: {$(s,x')$}](D) at (3.8,2);
\fill(D) circle (4pt);

\coordinate (E) at (1.6,2);
\fill(E) circle (4pt);

\coordinate (G) at (1.6,5.35);
\fill(G) circle (4pt);

\draw [red, thick](D) -- (G)  node[midway, above right] {$\varepsilon$};

\coordinate(H) at (-1.6,2);
\fill(H) circle (4pt);

\draw[thick] (E) -- (H);

\coordinate  [label= above: {$(0,x')$}](I) at (0,2);
\fill(I) circle (4pt);

\draw[green!50 !black, thick] (E)--(G);

\draw[blue,thick] (E) -- (D);

\node (Z) at (3,1.2) [blue] {$|s| - \delta$};

\node (Y) at (4.5,7.5) [green!50 !black] {$\sqrt{\varepsilon^2 - (|s| - \delta)^2}$};
\draw[->, thick,green!50 !black] (Y) --(1.8,3.5);

\node (X1) at (-3.9,2){$(0,x')-\delta e_1$};

\begin{scope}[xshift=17cm]

\draw[thick,->] (-3,0) -- (6,0) node[right]  {$\mathbb{R}$};
\draw[thick,->] (0,-4.0) -- (0,6.6) node[above] {$\mathbb{R}^{d-1}$};

\fill[yellow] (-1.6,-4.0) rectangle (1.6,6.2);

\draw[thick,-] (1.6,-4.0) -- (1.6,6.2);
\draw[thick,-] (-1.6,-4.0) -- (-1.6,6.2);

\draw[-] (0,-4.0) -- (0,6.2);
\draw[-] (-2,0) -- (4,0);

\node (A) at (0.8,-3) {$\tilde{J}_{\delta}$};

\coordinate [label= below: {$(s,x')$}](D) at (1,2);
\fill[orange, opacity= 0.4] (1,2) circle (4);
\draw[thick] (1,2) circle (4);
\fill(D) circle (4pt);

\node (K) at (5.8,5.2) {$B_{\varepsilon}(s,x')$};

\draw[red,thick] (D) -- (1,6) node[midway, left]{$\varepsilon$};
\end{scope}

\end{tikzpicture}
\end{center}
This allows us to continue to estimate~\eqref{estimateabovecriticalpart1}. Via the substitution  $r= (s-\delta)/\varepsilon$ (and recalling that $\delta = \varepsilon^2$) we then obtain 
\begin{align*}
\lefteqn{\frac{1}{\varepsilon} \int_{J_{\varepsilon + \delta}} f \bigg(\frac{1}{\omega_d \, \varepsilon^{d-1}} \int_{B_{\varepsilon}(x)\cap \tilde{J}_{\delta}}\betrag{\aee} \dd y \bigg) \dd x} \\
 & \leq \frac{1}{\varepsilon} \int_{J^{P_1}_{\varepsilon + \delta}} 2 \int_{0}^{\varepsilon + \delta} f \bigg( \frac{\omega_{d-1}}{\omega_d \varepsilon^{d-1}} \big[ \varepsilon^2 - ( \max\{ s-\delta, 0 \} ) ^2 \big]^{\frac{d-1}{2}}  ( \vert[u(0,x')]\vert+ 2\varrho) \bigg) \dd s \,\dd x' \\
 & =  \int_{J^{P_1}_{\varepsilon + \delta}}  2 \int_{-\varepsilon}^{1} f\bigg(\frac{\omega_{d-1}}{\omega_d}\big[1- (\max \{r,0\} )^2\big] ^{\frac{d-1}{2}}  (\vert[u(0,x')]\vert+ 2\varrho)  \bigg)\dd r \,\dd x'.
\end{align*}
Hence, taking into account~\eqref{estimateaboveniceboundary} and the fact that $J^{P_1}_{1+ \varepsilon}$ was defined as the projection of $J_{\varepsilon + \delta}$ onto $\{x_1=0\}$, we deduce from~\eqref{splittingsmoothnotsmooth}
\begin{align*}
&\frac{1}{\varepsilon} \int_{J_{\varepsilon + \delta}}f  \bigg( \varepsilon \Ibb |\aee| \dd y \bigg) \dd x \\ 
&\leq  c_0 L |J_{\varepsilon + \delta}| +  \int_{ \lbrace x_1=0\rbrace}  2  \int_{-\varepsilon }^{1}  f\bigg(\frac{\omega_{d-1}}{\omega_d}\big[1- ( \max \{ r,0\})  ^2\big] ^{\frac{d-1}{2}} (\vert[u(x)]\vert+ 2\varrho) \bigg)\dd r  \,\dd  \HA(x).
\end{align*}
In the passage to the $\limsup$ as $\varepsilon \to 0$, the first term disappears and the inner integral of the second term yields $\tfrac{1}{2} \theta([u]+2\varrho)$ by definition~\eqref{definitiontheta} of~$\theta$. Taking then the limit $\varrho \searrow 0$, we conclude 
\begin{equation*}
\limsup_{\varepsilon \to 0}\frac{1}{\varepsilon} \int_{J_{\varepsilon + \delta}} f \bigg(\varepsilon \Ibb | \aee| \dd y \bigg) \dd x \leq  \int_{J_u} \theta([u])\dd \HA,
\end{equation*}
since $[u]$ vanishes $\HA$-almost everywhere outside of~$J_u$. 
Together with \eqref{estimateabove0}, \eqref{estimateabove1}, \eqref{estimateabove2} and~\eqref{estimateabove3} we then find 
\begin{align*}
\limsup_{\varepsilon \to 0} E_{\varepsilon}(\ue,\gae) &
 \leq \limsup_{\varepsilon \to 0} \int_{\Omega } \betrag{\nabla u - g }^2 \dd x +   \limsup_{\varepsilon \to 0}
 \int_{\Omega \setminus J_{\varepsilon +\delta}}  c_0 \dashint_{B_{\varepsilon}(x) \cap \Omega} |g| \dd y \, \dd x \\ 
 & \quad + \limsup_{\varepsilon \to 0}\frac{1}{\varepsilon} \int_{J_{\varepsilon+ \delta}} f \bigg( \varepsilon \dashint_{B_{\varepsilon}(x)} |\aee| \dd y \bigg) \dd x
 \\& \leq \int_{\Omega} \betrag{\nabla u - g }^2 \dd x+ \int_{\Omega}  c_0 |g| \dd x + \int_{J_u} \theta([u])\dd \HA = E(u, \gamma),
\end{align*}
which concludes the proof of Proposition \ref{WdrBV}.
\end{proof}

Next, we use the density result \cite[Theorem~3.1]{CT99} and the relaxation result of Corollary \ref{rel2conBV} to generalize the statement of Proposition \ref{WdrBV} from $\Wde$ to $BV(\Omega)$. Together with Theorem \ref{liminfnBV}, this shows the $\Gamma$-convergence result of Theorem \ref{mainresultmBV} (recall that the estimate is only nontrivial for pairs $(u, \gamma) \in BV(\Omega) \times \Mam$ with $\norm{u}_{L^{\infty}(\Omega)}  \leq K$, $\gamma= D^su + g \mathcal{L}^d$ and $\nabla u -g \in L^2(\Omega,\Rd)$ considered here).

\begin{theorem}
For every $(u,\gamma) \in  BV(\Omega) \times \Mam$  with $\norm{u}_{L^{\infty}(\Omega)} \leq K$, $\gamma = D^su+ g \mathcal{L}^d$ and $\nabla u- g \in L^2(\Omega,\Rd)$ we have
\begin{equation*}
E^{\prime \prime}(u,\gamma) \leq E(u,\gamma).
\end{equation*}
\end{theorem}

\begin{proof}
We first consider $(u,\gamma) \in SBV^2(\Omega) \times \Mam$ with $\norm{u}_{L^{\infty}(\Omega)} \leq K$, $\gamma= D^su + g \mathcal{L}^d$ and $\nabla u - g \in L^2(\Omega, \Rd)$.
According to \cite[Theorem~3.1 and Remark~3.5]{CT99} and \cite[Remark~6.2]{Crismale_19}, there exists a sequence $\{\ue\}_{\varepsilon}$ in $\Wde$ such that $\ue \to u$ in $L^1(\Omega)$, $\nabla u_{\varepsilon} \to \nabla u$ in $L^2(\Omega, \R^d)$, $\norm{u_{\varepsilon}}_{L^{\infty}(\Omega)} \leq \norm{u}_{L^{\infty}(\Omega)}  \leq K$ for every $\varepsilon>0$ (possibly after multiplying $u_\varepsilon$ with the factor  $\tfrac{\| u \|_{L^\infty}(\Omega)}{\| u_{\varepsilon} \|_{L^\infty}(\Omega)}$  converging to~$1$) and 
\begin{align}
&\limsup_{\varepsilon \to 0} \int_{J_{u_{\varepsilon}}} \theta([u_{\varepsilon}])  \dd\HA \leq  \int_{J_{u}} \theta([u])  \dd\HA.  \label{WdeBV}
\end{align}
We next define a sequence $\{\gae\}_{\varepsilon}$ in $\Mam$ via $\gae\coloneqq  D^s\ue + \ggae \mathcal{L}^d$ with $\ggae\coloneqq g - \nabla u +\nabla \ue$ for every $\varepsilon>0$. In view of $\ue \to u$ in $L^1(\Omega)$, we notice that $D\ue = \nabla \ue \mathcal{L}^d +D^s\ue \to Du$ and in turn $\gae \to \gamma$ in the flat norm on~$\Omega$. In addition, since $\nabla \ue - \ggae = \nabla u - g$ for all $\varepsilon >0$, we have $\ggae \to g$ in $L^2(\Omega, \Rd)$ and then in particular also in $L^1(\Omega, \Rd)$. By the lower semicontinuity of the upper $\Gamma$-limit $E^{\prime \prime}$, Proposition \ref{WdrBV} and~\eqref{WdeBV}, we obtain 
\begin{align*}
E^{\prime \prime}(u  , \gamma) &\leq \liminf_{\varepsilon \to 0} E^{\prime \prime}(u_{\varepsilon}, \gamma_{\varepsilon}) \leq \liminf_{\varepsilon \to 0} E (u_{\varepsilon}, \gamma_{\varepsilon})\\
 &\leq \int_{\Omega} { |\nabla u - g|^2 \dd x } + \lim_{\varepsilon \to 0} c_0 \int_{\Omega} |\ggae| \dd x  + \limsup_{\varepsilon \to 0}  \int_{J_{u_{\varepsilon}}} \theta([\ue])  \dd\HA\\
 &\leq \int_{\Omega} { |\nabla u - g|^2 \dd x } +  c_0 \int_{\Omega} |g| \dd x  +  \int_{J_{u}} \theta([u])  \dd\HA.
\end{align*}

In a second step, we consider $(u,\gamma) \in BV(\Omega) \times \Mam$  as in the formulation of the theorem but with $g-\nabla u  \eqqcolon v \in C^{\infty}_0(\oOmega, \Rd)$ in addition. We consider the function $\psi \colon \oOmega \times \Rd \to [0,\infty)$ defined by $\psi(x, \xi)= |v(x)|^2 + c_0|\xi + v(x)|$ for $(x,\xi) \in \oOmega \times \Rd$ as in Section~\ref{sec:relBV}. Now, we apply Corollary~\ref{rel2conBV} combined with the lower semicontinuity of $E^{\prime \prime}$ along a suitable sequence $\{(\ue,\gae)\}_{\varepsilon}$ in $SBV^2(\Omega) \times \Mam$ with $\norm{\ue}_{L^{\infty}(\Omega)} \leq K$ and $\gae\coloneqq  D^s\ue + (g - \nabla u +\nabla \ue) \mathcal{L}^d$ for every $\varepsilon>0$ to deduce that
\begin{align*}
E^{\prime \prime}(u, \gamma) &\leq  \Io  \psi (x, \nabla u )\dd x + \Ij {\theta}([u])\dd \HA  + c_0 |D^cu|(\Omega) \notag \\
& = \Io |\nabla u - g|^2 \dd x + c_0 \Io |g| \dd x + \Ij {\theta}([u])\dd \HA  + c_0 |D^cu|(\Omega) \notag \\
&= E(u, \gamma). 
\end{align*}
Since $C^{\infty}_0(\oOmega, \Rd)$ is dense in $L^2(\Omega, \Rd)$, there exists a sequence $\{v_k\}_k$ in  $C^{\infty}_0(\oOmega, \Rd)$ with $v_k \to g - \nabla u$ in $L^2(\Omega, \Rd)$. Defining $g_k\coloneqq  \nabla u + v_k$, we then infer from the previous inequality (applied with $\psi_k$ instead of~$\psi$) $E^{\prime \prime}(u, D^su + g_k \mathcal{L}^d) \leq E(u,  D^su + g_k \mathcal{L}^d)$ for all $k \in \N$. By the lower semicontinuity of $E^{\prime \prime}$, we finally obtain
\begin{align*}
 E^{\prime \prime}(u, \gamma) 
 & \leq \liminf_{k \to \infty}  E^{\prime \prime}(u, D^su + g_k \mathcal{L}^d) 
 \leq  \liminf_{k \to \infty} E(u,  D^su + g_k \mathcal{L}^d) 
 \\
 & \leq \lim_{k \to \infty} \Io |\nabla u - g_k|^2 \dd x +  \lim_{k \to \infty} c_0 \Io |g_k | \dd x   + \Ij {\theta}([u])\dd \HA  + c_0 |D^cu|(\Omega) \\
 & =  E(u, D^su + g \mathcal{L}^d) = E(u,\gamma), 
\end{align*}
which completes the proof.
\end{proof}

\section{$\Gamma$-convergence for the minimal energies with respect to $\gamma$}\label{sec_optimzed-gamma}

In this final section we show the $\Gamma$-convergence result in Corollary \ref{mainresult2BV} for the minimal energies with respect to the second variable $\gamma$, i.e., for the energies~$\tilde{E}_{\varepsilon}$ and~$\tilde{E}$ from~\eqref{def_minimal_energy_eps} and~\eqref{def_minimal_energy_limit}, respectively. We note that, since the function~$g^*$ from \eqref{gammaaopt} solves the optimization problem in~\eqref{optimization_problem}, for every $u \in BV(\Omega)$ there holds
\begin{equation}\label{optgammaBV}
\tilde{E}(u)=E(u, D^su+ g^* \mathcal{L}^1) = E(u,\go).
\end{equation} 
For completeness, we also state the corresponding compactness result.

\begin{corollary}[Compactness of the minimal energies with respect to~$\gamma$]
Let $\{u_{\varepsilon}\}_{\varepsilon}$ be a sequence in $\Lme$ with
\begin{equation*}
\tilde{E}_{\varepsilon}(\ue)\leq C_0 \quad \text{for all } \varepsilon>0
\end{equation*}
for a positive constant $C_0$. There exists a function $u \in BV(\Omega)$ with $\norm{u}_{L^{\infty}(\Omega)} \leq K$ such that, up to a subsequence, $\{u_{\varepsilon}\}_{\varepsilon}$ converges to~$u$ in $\Lme$.
\end{corollary}

\begin{proof}
We choose a low energy sequence $\{\gae\}_{\varepsilon}$ in $\Mam$ with $E_{\varepsilon}(\ue,\gae) \leq \tilde{E}_{\varepsilon}(\ue) + 1$ for all $\varepsilon>0$. In view of $E_{\varepsilon}(\ue,\gae) \leq C_0 +1$ for all $\varepsilon>0$, the claim then follows from Theorem~\ref{commBV}.
\end{proof}

\begin{proof}[Proof of Corollary \ref{mainresult2BV}]
It is again sufficient to establish the $\Gamma$-$\liminf$-inequality and the $\Gamma$-$\limsup$-inequality only for $u \in BV(\Omega)$ with $\norm{u}_{L^{\infty}(\Omega)} \leq K$ since the estimates are trivial otherwise. We start with the proof of the $\Gamma$-$\liminf$-inequality. Let $\{\ue\}_{\varepsilon}$ be an arbitrary sequence in $W^{1,1}(\Omega)$ with $\ue \to u$ in $\Lme$, for which we may assume $\tilde{E}_{\varepsilon}(\ue)\leq C_0$ for some positive constant~$C_0$  and all~$\varepsilon >0$. 
We further take a low energy sequence $\{\gae\}_{\varepsilon}$ in~$\Mam$ with $E_\varepsilon(u_\varepsilon,\gamma_\varepsilon) \leq \tilde{E}_{\varepsilon}(\ue) + \varepsilon$ for all~$\varepsilon >0$. Possibly by passing to a subsequence, we may assume
\begin{equation*}
\liminf_{\varepsilon \to 0} E_{\varepsilon}(\ue,\gae) = \lim_{\varepsilon \to 0} E_{\varepsilon}(\ue,\gae).
\end{equation*}
Since $u_{\varepsilon} \to u$ in $L^1(\Omega)$, we infer from the compactness result of Theorem~\ref{commBV} that there exists a function $g \in L^1(\Omega, \Rd)$ with $\nabla u - g \in L^2(\Omega, \Rd)$ such that, up to a further subsequence, $\gae \to D^su + g \mathcal{L}^{d}$ in the flat norm. By Theorem \ref{mainresultmBV}~\ref{mainresultmBV_1} (applied for an appropriate subsequence of $\{E_\varepsilon\}_\varepsilon$), we find
\begin{align*}
 \lim_{\varepsilon \to 0} E_{\varepsilon}(\ue,\gae) = \liminf_{\varepsilon \to 0 } E_{\varepsilon}(\ue,\gae) \geq E(u,D^su + g \mathcal{L}^{d} ).
\end{align*} 
Since $\{\gamma_\varepsilon\}_\varepsilon$ is a low energy sequence, this yields
\begin{equation*}
\liminf_{\varepsilon \to 0 }\tilde{E}_{\varepsilon}(u)\geq \tilde{E}(u).
\end{equation*}

Concerning the $\Gamma$-$\limsup$-inequality, we notice that by Theorem~\ref{mainresultmBV}~\ref{mainresultmBV_2} there exists a recovery sequence $\{(\ue,\gae)\}_{\varepsilon}$ of $(u,\go)$ in $L^{1}(\Omega) \times \Mam$. Because of~\eqref{optgammaBV}, the claim then follows via
\begin{equation*}
\limsup_{\varepsilon \to 0 } \tilde{E}_{\varepsilon}(\ue)\leq \limsup_{\varepsilon \to 0 } E_{\varepsilon}(\ue,\gae) \leq E(u, \go) = \tilde{E}(u). \qedhere
\end{equation*}
\end{proof}


\bibliographystyle{alpha} 
\bibliography{Bibliography} 

\end{document}